\documentclass[10pt, a4paper]{amsart}
\usepackage[arXiv]{rseaineng}
\usepackage[cal]{rseaincmds}
\usepackage{hhline, array, bm}

\usepackage{cellspace}
\usepackage{tikz}
\tikzstyle{rt} = [circle, draw, thin, inner sep = 2.5pt, minimum size = 2pt, fill = white]
\tikzstyle{hrt} = [circle, draw, thin, inner sep = 2.5pt, minimum size = 2pt, fill = black]

\tikzstyle{v} = [circle, draw, inner sep=2pt, minimum size=3pt, fill = white]
\tikzstyle{vvv} = [circle, draw, ultra thick, inner sep=2pt, minimum size=2pt, fill = black]
\tikzstyle{v3} = [circle, draw, ultra thick, inner sep=2pt, minimum size=2pt, fill = black]
\tikzstyle{v2} = [circle, draw, ultra thick, inner sep=2pt, minimum size=2pt, fill = black]
\tikzstyle{v6} = [circle, draw, ultra thick, inner sep=2pt, minimum size=2pt, fill = black]
\tikzstyle{vv} = [circle, draw, thick, inner sep=5pt, minimum size=9pt, fill = white]
\tikzstyle{bv} = [circle, draw, ultra thick, inner sep=2pt, minimum size=2pt, fill = black]
\tikzstyle{rv} = [circle, draw, inner sep=1pt, minimum size=3pt, fill = red]

\newcommand{\eDAi}{
	\begin{tikzpicture}
		\draw (1,0) node(a1)[hrt]{};
		\draw (2,0) node(a2)[rt]{};
		\draw (1,-1/3) node{$\alpha_0$};
		\draw (2,-1/3) node{$\alpha_1$};
		\draw (1,1/3) node{\scriptsize $1$};
		\draw (2,1/3) node{\scriptsize $1$};
		\draw (a1) -- (a2);
\end{tikzpicture}}
\newcommand{\eDAell}{
	\begin{tikzpicture}
		\draw (1,0) node(a1)[rt]{};
		\draw (2,0) node(a2)[rt]{};
		\draw (3,0) node(a3)[rt]{};
		\draw (5,0) node(a4)[rt]{};
		\draw (6,0) node(a5)[rt]{};
		\draw (7,0) node(a6)[rt]{};
		\draw (4,1.3) node(a0)[hrt]{};
		\draw (1,-1/3) node{$\alpha_1$};
		\draw (2,-1/3) node{$\alpha_2$};
		\draw (3,-1/3) node{$\alpha_3$};
		\draw (5,-1/3) node{$\alpha_{\ell-2}$};
		\draw (6,-1/3) node{$\alpha_{\ell-1}$};
		\draw (7,-1/3) node{$\alpha_\ell$};
		\draw (4,1.3-1/3) node{$\alpha_0$};
		\draw (1,1/3) node{\scriptsize $1$};
		\draw (2.1,1/3) node{\scriptsize $1$};
		\draw (3,1/3) node{\scriptsize $1$};
		\draw (5,1/3) node{\scriptsize $1$};
		\draw (6-0.1,1/3) node{\scriptsize $1$};
		\draw (7,1/3) node{\scriptsize $1$};
		\draw (4,1.3+1/3) node{\scriptsize $1$};
		\draw (a1) -- (a2) -- (a3)  (a4) -- (a5) -- (a6) -- (a0) -- (a1);
		\draw[thick, loosely dashed] (a3) -- (a4);
\end{tikzpicture}}
\newcommand{\eDBC}{
	\begin{tikzpicture}
		\draw (1,0) node(a0)[hrt]{};
		\draw (1+1,0) node(a1)[rt]{};
		\draw (2+1,0) node(a2)[rt]{};
		\draw (3+1,0) node(a3)[rt]{};
		\draw (5+1,0) node(a4)[rt]{};
		\draw (6+1,0) node(a5)[rt]{};
		\draw (7+1,0) node(a6)[rt]{};
		\draw (1,-1/3) node{$\alpha_0$};
		\draw (1+1,-1/3) node{$\alpha_1$};
		\draw (2+1,-1/3) node{$\alpha_2$};
		\draw (3+1,-1/3) node{$\alpha_3$};
		\draw (5+1,-1/3) node{$\alpha_{\ell-2}$};
		\draw (6+1,-1/3) node{$\alpha_{\ell-1}$};
		\draw (7+1,-1/3) node{$\alpha_\ell$};
		\draw (1,1/3) node{\scriptsize $1$};
		\draw (1+1,1/3) node{\scriptsize $2$};
		\draw (2+1,1/3) node{\scriptsize $2$};
		\draw (3+1,1/3) node{\scriptsize $2$};
		\draw (5+1,1/3) node{\scriptsize $2$};
		\draw (6+1,1/3) node{\scriptsize $2$};
		\draw (7+1,1/3) node{\scriptsize $2$};
		\draw (a1) -- (a2) -- (a3)  (a4) -- (a5);
		\draw[double distance = 2pt] (a5) --node[]{\Large $>$} (a6) (a0) --node[]{\Large $>$} (a1);
		\draw[thick, loosely dashed] (a3) -- (a4);
\end{tikzpicture}}
\newcommand{\eDBii}{
	\begin{tikzpicture}
		\draw (1,0) node(a1)[rt]{};
		\draw (2,0) node(a2)[rt]{};
		\draw (3,0) node(a0)[hrt]{};
		\draw (1,-1/3) node{$\alpha_1$};
		\draw (2,-1/3) node{$\alpha_2$};
		\draw (3,-1/3) node{$\alpha_0$};
		\draw (1,1/3) node{\scriptsize $1$};
		\draw (2,1/3) node{\scriptsize $2$};
		\draw (3,1/3) node{\scriptsize $1$};
		\draw[double distance = 2pt] (a1) --node[]{\Large $>$} (a2)  --node[]{\Large $<$} (a0);
\end{tikzpicture}}
\newcommand{\eDBell}{
	\begin{tikzpicture}
		\draw (2-0.7,0.7) node(a0)[hrt]{};
		\draw (2-0.7,-0.7) node(a1)[rt]{};
		\draw (2,0) node(a2)[rt]{};
		\draw (3,0) node(a3)[rt]{};
		\draw (5,0) node(a4)[rt]{};
		\draw (6,0) node(a5)[rt]{};
		\draw (7,0) node(a6)[rt]{};
		\draw (2-0.7,0.7-1/3) node{$\alpha_{0}$};
		\draw (2-0.7,-0.7-1/3) node{$\alpha_1$};
		\draw (2,-1/3) node{$\alpha_2$};
		\draw (3,-1/3) node{$\alpha_3$};
		\draw (5,-1/3) node{$\alpha_{\ell-2}$};
		\draw (6,-1/3) node{$\alpha_{\ell-1}$};
		\draw (7,-1/3) node{$\alpha_\ell$};
		\draw (2-0.7,0.7+1/3) node{\scriptsize $1$};
		\draw (2-0.7,-0.7+1/3) node{\scriptsize $1$};
		\draw (2,1/3) node{\scriptsize $2$};
		\draw (3,1/3) node{\scriptsize $2$};
		\draw (5,1/3) node{\scriptsize $2$};
		\draw (6,1/3) node{\scriptsize $2$};
		\draw (7,1/3) node{\scriptsize $2$};
		\draw (a1) -- (a2) -- (a3)  (a4) -- (a5) (a0) -- (a2);
		\draw[double distance = 2pt] (a5) --node[]{\Large $>$} (a6);
		\draw[thick, loosely dashed] (a3) -- (a4);
\end{tikzpicture}}
\newcommand{\eDC}{
	\begin{tikzpicture}
		\draw (1,0) node(a0)[hrt]{};
		\draw (1+1,0) node(a1)[rt]{};
		\draw (2+1,0) node(a2)[rt]{};
		\draw (3+1,0) node(a3)[rt]{};
		\draw (5+1,0) node(a4)[rt]{};
		\draw (6+1,0) node(a5)[rt]{};
		\draw (7+1,0) node(a6)[rt]{};
		\draw (1,-1/3) node{$\alpha_0$};
		\draw (1+1,-1/3) node{$\alpha_1$};
		\draw (2+1,-1/3) node{$\alpha_2$};
		\draw (3+1,-1/3) node{$\alpha_3$};
		\draw (5+1,-1/3) node{$\alpha_{\ell-2}$};
		\draw (6+1,-1/3) node{$\alpha_{\ell-1}$};
		\draw (7+1,-1/3) node{$\alpha_\ell$};
		\draw (1,1/3) node{\scriptsize $1$};
		\draw (1+1,1/3) node{\scriptsize $2$};
		\draw (2+1,1/3) node{\scriptsize $2$};
		\draw (3+1,1/3) node{\scriptsize $2$};
		\draw (5+1,1/3) node{\scriptsize $2$};
		\draw (6+1,1/3) node{\scriptsize $2$};
		\draw (7+1,1/3) node{\scriptsize $1$};
		\draw (a1) -- (a2) -- (a3)  (a4) -- (a5);
		\draw[double distance = 2pt] (a5) --node[]{\Large $<$} (a6) (a0) --node[]{\Large $>$} (a1);
		\draw[thick, loosely dashed] (a3) -- (a4);
\end{tikzpicture}}
\newcommand{\eDD}{
	\begin{tikzpicture}
		\draw (2-0.7,0.7) node(a0)[hrt]{};
		\draw (2-0.7,-0.7) node(a1)[rt]{};
		\draw (2,0) node(a2)[rt]{};
		\draw (3,0) node(a3)[rt]{};
		\draw (5,0) node(a4)[rt]{};
		\draw (5.7,0.7) node(a5)[rt]{};
		\draw (5.7,-0.7) node(a6)[rt]{};
		\draw (2-0.7,0.7-1/3) node{$\alpha_{0}$};
		\draw (2-0.7,-0.7-1/3) node{$\alpha_1$};
		\draw (2,-1/3) node{$\alpha_2$};
		\draw (3,-1/3) node{$\alpha_3$};
		\draw (5-0.1,-1/3) node{$\alpha_{\ell-2}$};
		\draw (5.7+0.1,0.7-1/3) node{$\alpha_{\ell-1}$};
		\draw (5.7,-0.7-1/3) node{$\alpha_\ell$};
		\draw (2-0.7,0.7+1/3) node{\scriptsize $1$};
		\draw (2-0.7,-0.7+1/3) node{\scriptsize $1$};
		\draw (2,1/3) node{\scriptsize $2$};
		\draw (3,1/3) node{\scriptsize $2$};
		\draw (5,1/3) node{\scriptsize $2$};
		\draw (5.7,0.7+1/3) node{\scriptsize $1$};
		\draw (5.7,-0.7+1/3) node{\scriptsize $1$};
		\draw (a1) -- (a2) -- (a3)  (a6) -- (a4) -- (a5) (a0) -- (a2);
		\draw[thick, loosely dashed] (a3) -- (a4);
\end{tikzpicture}}
\newcommand{\eDEvi}{
	\begin{tikzpicture}
		\draw (1,0) node(a1)[rt]{};
		\draw (3,-1) node(a2)[rt]{};
		\draw (2,0) node(a3)[rt]{};
		\draw (3,0) node(a4)[rt]{};
		\draw (4,0) node(a5)[rt]{};
		\draw (5,0) node(a6)[rt]{};
		\draw (3,-2) node(a0)[hrt]{};
		\draw (1,-1/3) node{$\alpha_1$};
		\draw (3+1/4,-1-1/3) node{$\alpha_2$};
		\draw (2,-1/3) node{$\alpha_3$};
		\draw (3+1/4,-1/3) node{$\alpha_4$};
		\draw (4,-1/3) node{$\alpha_5$};
		\draw (5,-1/3) node{$\alpha_6$};
		\draw (3,-2-1/3) node{$\alpha_0$};
		\draw (1,1/3) node{\scriptsize $1$};
		\draw (3-1/5,-1+1/3) node{\scriptsize $2$};
		\draw (2,1/3) node{\scriptsize $2$};
		\draw (3,1/3) node{\scriptsize $3$};
		\draw (4,1/3) node{\scriptsize $2$};
		\draw (5,1/3) node{\scriptsize $1$};
		\draw (3-1/5,-2+1/3) node{\scriptsize $1$};
		\draw (a1) -- (a3) -- (a4) -- (a5) -- (a6) (a4) -- (a2) -- (a0);
\end{tikzpicture}}
\newcommand{\eDEvii}{
	\begin{tikzpicture}
		\draw (0,0) node(a0)[hrt]{};
		\draw (1,0) node(a1)[rt]{};
		\draw (3,-1) node(a2)[rt]{};
		\draw (2,0) node(a3)[rt]{};
		\draw (3,0) node(a4)[rt]{};
		\draw (4,0) node(a5)[rt]{};
		\draw (5,0) node(a6)[rt]{};
		\draw (6,0) node(a7)[rt]{};
		\draw (0,-1/3) node{$\alpha_0$};
		\draw (1,-1/3) node{$\alpha_1$};
		\draw (3,-1-1/3) node{$\alpha_2$};
		\draw (2,-1/3) node{$\alpha_3$};
		\draw (3+1/4,-1/3) node{$\alpha_4$};
		\draw (4,-1/3) node{$\alpha_5$};
		\draw (5,-1/3) node{$\alpha_6$};
		\draw (6,-1/3) node{$\alpha_7$};
		\draw (0,1/3) node{\scriptsize $1$};
		\draw (1,1/3) node{\scriptsize $2$};
		\draw (3-1/5,-1+1/3) node{\scriptsize $2$};
		\draw (2,1/3) node{\scriptsize $3$};
		\draw (3,1/3) node{\scriptsize $4$};
		\draw (4,1/3) node{\scriptsize $3$};
		\draw (5,1/3) node{\scriptsize $2$};
		\draw (6,1/3) node{\scriptsize $1$};
		\draw (a0) -- (a1) -- (a3) -- (a4) -- (a5) -- (a6) -- (a7) (a4) -- (a2);
\end{tikzpicture}}
\newcommand{\eDEviii}{
	\begin{tikzpicture}
		\draw (1,0) node(a1)[rt]{};
		\draw (3,-1) node(a2)[rt]{};
		\draw (2,0) node(a3)[rt]{};
		\draw (3,0) node(a4)[rt]{};
		\draw (4,0) node(a5)[rt]{};
		\draw (5,0) node(a6)[rt]{};
		\draw (6,0) node(a7)[rt]{};
		\draw (7,0) node(a8)[rt]{};
		\draw (8,0) node(a0)[hrt]{};
		\draw (1,-1/3) node{$\alpha_1$};
		\draw (3,-1-1/3) node{$\alpha_2$};
		\draw (2,-1/3) node{$\alpha_3$};
		\draw (3+1/4,-1/3) node{$\alpha_4$};
		\draw (4,-1/3) node{$\alpha_5$};
		\draw (5,-1/3) node{$\alpha_6$};
		\draw (6,-1/3) node{$\alpha_7$};
		\draw (7,-1/3) node{$\alpha_8$};
		\draw (8,-1/3) node{$\alpha_0$};
		\draw (1,1/3) node{\scriptsize $2$};
		\draw (3-1/5,-1+1/3) node{\scriptsize $3$};
		\draw (2,1/3) node{\scriptsize $4$};
		\draw (3,1/3) node{\scriptsize $6$};
		\draw (4,1/3) node{\scriptsize $5$};
		\draw (5,1/3) node{\scriptsize $4$};
		\draw (6,1/3) node{\scriptsize $3$};
		\draw (7,1/3) node{\scriptsize $2$};
		\draw (8,1/3) node{\scriptsize $1$};
		\draw (a1) -- (a3) -- (a4) -- (a5) -- (a6) -- (a7) -- (a8) -- (a0) (a4) -- (a2);
\end{tikzpicture}}
\newcommand{\eDF}{
	\begin{tikzpicture}
		\draw (0,0) node(a0)[hrt]{};
		\draw (1,0) node(a1)[rt]{};
		\draw (2,0) node(a2)[rt]{};
		\draw (3,0) node(a3)[rt]{};
		\draw (4,0) node(a4)[rt]{};
		\draw (0,-1/3) node{$\alpha_0$};
		\draw (1,-1/3) node{$\alpha_1$};
		\draw (2,-1/3) node{$\alpha_2$};
		\draw (3,-1/3) node{$\alpha_3$};
		\draw (4,-1/3) node{$\alpha_4$};
		\draw (0,1/3) node{\scriptsize $1$};
		\draw (1,1/3) node{\scriptsize $2$};
		\draw (2,1/3) node{\scriptsize $3$};
		\draw (3,1/3) node{\scriptsize $4$};
		\draw (4,1/3) node{\scriptsize $2$};
		\draw (a0) -- (a1) -- (a2)  (a3) -- (a4);
		\draw[double distance = 2pt] (a2) --node[]{\Large $>$} (a3);
\end{tikzpicture}}
\newcommand{\eDG}{
	\begin{tikzpicture}
		\draw (1,0) node(a1)[rt]{};
		\draw (2,0) node(a2)[rt]{};
		\draw (3,0) node(a0)[hrt]{};
		\draw (1,-1/3) node{$\alpha_1$};
		\draw (2,-1/3) node{$\alpha_2$};
		\draw (3,-1/3) node{$\alpha_0$};
		\draw (1,1/3) node{\scriptsize $3$};
		\draw (2,1/3) node{\scriptsize $2$};
		\draw (3,1/3) node{\scriptsize $1$};
		\draw[double distance = 3pt] (a1) --node[]{\Large $<$} (a2);
		\draw (a1) -- (a2) -- (a0);
\end{tikzpicture}}

\newcommand{\DA}{
	\begin{tikzpicture}
		\draw (1,0) node(a1)[rt]{};
		\draw (2,0) node(a2)[rt]{};
		\draw (3,0) node(a3)[rt]{};
		\draw (5,0) node(a4)[rt]{};
		\draw (6,0) node(a5)[rt]{};
		\draw (7,0) node(a6)[rt]{};
		\draw (1,-1/3) node{$\alpha_1$};
		\draw (2,-1/3) node{$\alpha_2$};
		\draw (3,-1/3) node{$\alpha_3$};
		\draw (5,-1/3) node{$\alpha_{\ell-2}$};
		\draw (6,-1/3) node{$\alpha_{\ell-1}$};
		\draw (7,-1/3) node{$\alpha_\ell$};
		\draw (1,1/3) node{\scriptsize $1$};
		\draw (2,1/3) node{\scriptsize $1$};
		\draw (3,1/3) node{\scriptsize $1$};
		\draw (5,1/3) node{\scriptsize $1$};
		\draw (6,1/3) node{\scriptsize $1$};
		\draw (7,1/3) node{\scriptsize $1$};
		\draw (a1) -- (a2) -- (a3)  (a4) -- (a5) -- (a6);
		\draw[thick, loosely dashed] (a3) -- (a4);
\end{tikzpicture}}
\newcommand{\DB}{
	\begin{tikzpicture}
		\draw (1,0) node(a1)[rt]{};
		\draw (2,0) node(a2)[rt]{};
		\draw (3,0) node(a3)[rt]{};
		\draw (5,0) node(a4)[rt]{};
		\draw (6,0) node(a5)[rt]{};
		\draw (7,0) node(a6)[rt]{};
		\draw (1,-1/3) node{$\alpha_1$};
		\draw (2,-1/3) node{$\alpha_2$};
		\draw (3,-1/3) node{$\alpha_3$};
		\draw (5,-1/3) node{$\alpha_{\ell-2}$};
		\draw (6,-1/3) node{$\alpha_{\ell-1}$};
		\draw (7,-1/3) node{$\alpha_\ell$};
		\draw (1,1/3) node{\scriptsize $1$};
		\draw (2,1/3) node{\scriptsize $2$};
		\draw (3,1/3) node{\scriptsize $2$};
		\draw (5,1/3) node{\scriptsize $2$};
		\draw (6,1/3) node{\scriptsize $2$};
		\draw (7,1/3) node{\scriptsize $2$};
		\draw (a1) -- (a2) -- (a3)  (a4) -- (a5);
		\draw[double distance = 2pt] (a5) --node[]{\Large $>$} (a6);
		\draw[thick, loosely dashed] (a3) -- (a4);
\end{tikzpicture}}
\newcommand{\DC}{
	\begin{tikzpicture}
		\draw (1,0) node(a1)[rt]{};
		\draw (2,0) node(a2)[rt]{};
		\draw (3,0) node(a3)[rt]{};
		\draw (5,0) node(a4)[rt]{};
		\draw (6,0) node(a5)[rt]{};
		\draw (7,0) node(a6)[rt]{};
		\draw (1,-1/3) node{$\alpha_1$};
		\draw (2,-1/3) node{$\alpha_2$};
		\draw (3,-1/3) node{$\alpha_3$};
		\draw (5,-1/3) node{$\alpha_{\ell-2}$};
		\draw (6,-1/3) node{$\alpha_{\ell-1}$};
		\draw (7,-1/3) node{$\alpha_\ell$};
		\draw (1,1/3) node{\scriptsize $2$};
		\draw (2,1/3) node{\scriptsize $2$};
		\draw (3,1/3) node{\scriptsize $2$};
		\draw (5,1/3) node{\scriptsize $2$};
		\draw (6,1/3) node{\scriptsize $2$};
		\draw (7,1/3) node{\scriptsize $1$};
		\draw (a1) -- (a2) -- (a3)  (a4) -- (a5);
		\draw[double distance = 2pt] (a5) --node[]{\Large $<$} (a6);
		\draw[thick, loosely dashed] (a3) -- (a4);
\end{tikzpicture}}
\newcommand{\DD}{
	\begin{tikzpicture}
		\draw (1,0) node(a1)[rt]{};
		\draw (2,0) node(a2)[rt]{};
		\draw (3,0) node(a3)[rt]{};
		\draw (5,0) node(a4)[rt]{};
		\draw (5.7,0.7) node(a5)[rt]{};
		\draw (5.7,-0.7) node(a6)[rt]{};
		\draw (1,-1/3) node{$\alpha_1$};
		\draw (2,-1/3) node{$\alpha_2$};
		\draw (3,-1/3) node{$\alpha_3$};
		\draw (5-0.1,-1/3) node{$\alpha_{\ell-2}$};
		\draw (5.7+0.1,0.7-1/3) node{$\alpha_{\ell-1}$};
		\draw (5.7,-0.7-1/3) node{$\alpha_\ell$};
		\draw (1,1/3) node{\scriptsize $1$};
		\draw (2,1/3) node{\scriptsize $2$};
		\draw (3,1/3) node{\scriptsize $2$};
		\draw (5,1/3) node{\scriptsize $2$};
		\draw (5.7,0.7+1/3) node{\scriptsize $1$};
		\draw (5.7,-0.7+1/3) node{\scriptsize $1$};
		\draw (a1) -- (a2) -- (a3)  (a6) -- (a4) -- (a5);
		\draw[thick, loosely dashed] (a3) -- (a4);
\end{tikzpicture}}
\newcommand{\DEvi}{
	\begin{tikzpicture}
		\draw (1,0) node(a1)[rt]{};
		\draw (3,-1) node(a2)[rt]{};
		\draw (2,0) node(a3)[rt]{};
		\draw (3,0) node(a4)[rt]{};
		\draw (4,0) node(a5)[rt]{};
		\draw (5,0) node(a6)[rt]{};
		\draw (1,-1/3) node{$\alpha_1$};
		\draw (3,-1-1/3) node{$\alpha_2$};
		\draw (2,-1/3) node{$\alpha_3$};
		\draw (3+1/4,-1/3) node{$\alpha_4$};
		\draw (4,-1/3) node{$\alpha_5$};
		\draw (5,-1/3) node{$\alpha_6$};
		\draw (1,1/3) node{\scriptsize $1$};
		\draw (3-1/5,-1+1/3) node{\scriptsize $2$};
		\draw (2,1/3) node{\scriptsize $2$};
		\draw (3,1/3) node{\scriptsize $3$};
		\draw (4,1/3) node{\scriptsize $2$};
		\draw (5,1/3) node{\scriptsize $1$};
		\draw (a1) -- (a3) -- (a4) -- (a5) -- (a6) (a4) -- (a2);
\end{tikzpicture}}
\newcommand{\DEvii}{
	\begin{tikzpicture}
		\draw (1,0) node(a1)[rt]{};
		\draw (3,-1) node(a2)[rt]{};
		\draw (2,0) node(a3)[rt]{};
		\draw (3,0) node(a4)[rt]{};
		\draw (4,0) node(a5)[rt]{};
		\draw (5,0) node(a6)[rt]{};
		\draw (6,0) node(a7)[rt]{};
		\draw (1,-1/3) node{$\alpha_1$};
		\draw (3,-1-1/3) node{$\alpha_2$};
		\draw (2,-1/3) node{$\alpha_3$};
		\draw (3+1/4,-1/3) node{$\alpha_4$};
		\draw (4,-1/3) node{$\alpha_5$};
		\draw (5,-1/3) node{$\alpha_6$};
		\draw (6,-1/3) node{$\alpha_7$};
		\draw (1,1/3) node{\scriptsize $2$};
		\draw (3-1/5,-1+1/3) node{\scriptsize $2$};
		\draw (2,1/3) node{\scriptsize $3$};
		\draw (3,1/3) node{\scriptsize $4$};
		\draw (4,1/3) node{\scriptsize $3$};
		\draw (5,1/3) node{\scriptsize $2$};
		\draw (6,1/3) node{\scriptsize $1$};
		\draw (a1) -- (a3) -- (a4) -- (a5) -- (a6) -- (a7) (a4) -- (a2);
\end{tikzpicture}}
\newcommand{\DEviii}{
	\begin{tikzpicture}
		\draw (1,0) node(a1)[rt]{};
		\draw (3,-1) node(a2)[rt]{};
		\draw (2,0) node(a3)[rt]{};
		\draw (3,0) node(a4)[rt]{};
		\draw (4,0) node(a5)[rt]{};
		\draw (5,0) node(a6)[rt]{};
		\draw (6,0) node(a7)[rt]{};
		\draw (7,0) node(a8)[rt]{};
		\draw (1,-1/3) node{$\alpha_1$};
		\draw (3,-1-1/3) node{$\alpha_2$};
		\draw (2,-1/3) node{$\alpha_3$};
		\draw (3+1/4,-1/3) node{$\alpha_4$};
		\draw (4,-1/3) node{$\alpha_5$};
		\draw (5,-1/3) node{$\alpha_6$};
		\draw (6,-1/3) node{$\alpha_7$};
		\draw (7,-1/3) node{$\alpha_8$};
		\draw (1,1/3) node{\scriptsize $2$};
		\draw (3-1/5,-1+1/3) node{\scriptsize $3$};
		\draw (2,1/3) node{\scriptsize $4$};
		\draw (3,1/3) node{\scriptsize $6$};
		\draw (4,1/3) node{\scriptsize $5$};
		\draw (5,1/3) node{\scriptsize $4$};
		\draw (6,1/3) node{\scriptsize $3$};
		\draw (7,1/3) node{\scriptsize $2$};
		\draw (a1) -- (a3) -- (a4) -- (a5) -- (a6) -- (a7) -- (a8) (a4) -- (a2);
\end{tikzpicture}}
\newcommand{\DF}{
	\begin{tikzpicture}
		\draw (1,0) node(a1)[rt]{};
		\draw (2,0) node(a2)[rt]{};
		\draw (3,0) node(a3)[rt]{};
		\draw (4,0) node(a4)[rt]{};
		\draw (1,-1/3) node{$\alpha_1$};
		\draw (2,-1/3) node{$\alpha_2$};
		\draw (3,-1/3) node{$\alpha_3$};
		\draw (4,-1/3) node{$\alpha_4$};
		\draw (1,1/3) node{\scriptsize $2$};
		\draw (2,1/3) node{\scriptsize $3$};
		\draw (3,1/3) node{\scriptsize $4$};
		\draw (4,1/3) node{\scriptsize $2$};
		\draw (a1) -- (a2)  (a3) -- (a4);
		\draw[double distance = 2pt] (a2) --node[]{\Large $>$} (a3);
\end{tikzpicture}}
\newcommand{\DG}{
	\begin{tikzpicture}
		\draw (1,0) node(a1)[rt]{};
		\draw (2,0) node(a2)[rt]{};
		\draw (1,-1/3) node{$\alpha_1$};
		\draw (2,-1/3) node{$\alpha_2$};
		\draw (1,1/3) node{\scriptsize $3$};
		\draw (2,1/3) node{\scriptsize $2$};
		\draw[double distance = 3pt] (a1) --node[]{\Large $<$} (a2);
		\draw (a1) -- (a2);
\end{tikzpicture}}

\newcommand{\Dyaji}{
\begin{tikzpicture}[scale=0.8, transform shape]
	\draw (-1,0) node{$\lra$};
	\draw (0.5,0) node{$\varnothing$};
\end{tikzpicture}
}

\newcommand{\Fi}{
	
	\begin{tikzpicture}[scale=0.8, transform shape]
		\draw (-1/2-1/8,1.3/2) node{$\lra$};
		\draw (1,0) node(a1)[rt]{};
		\draw (2,0) node(a2)[rt]{};
		\draw (3,0) node(a3)[rt]{};
		\draw (5,0) node(a4)[rt]{};
		\draw (6,0) node(a5)[rt]{};
		\draw (7,0) node(a6)[rt]{};
		\draw (4,1.3) node(a0)[hrt]{};
		\draw (1,-1/2.5) node{$\bar\alpha_1^\omega$};
		\draw (2,-1/2.5) node{$\bar\alpha_2^\omega$};
		\draw (3,-1/2.5) node{$\bar\alpha_3^\omega$};
		\draw (5,-1/2.5) node{$\bar\alpha_{r-2}^\omega$};
		\draw (6,-1/2.5) node{$\bar\alpha_{r-1}^\omega$};
		\draw (7,-1/2.5) node{$\bar\alpha_r^\omega$};
		\draw (4,1.3-1/2.5) node{$\bar\alpha_0^\omega$};
		\draw (1,1/3) node{\scriptsize $1$};
		\draw (2.1,1/3) node{\scriptsize $1$};
		\draw (3,1/3) node{\scriptsize $1$};
		\draw (5,1/3) node{\scriptsize $1$};
		\draw (6-0.1,1/3) node{\scriptsize $1$};
		\draw (7,1/3) node{\scriptsize $1$};
		\draw (4,1.3+1/3) node{\scriptsize $1$};
		\draw (a1) -- (a2) -- (a3)  (a4) -- (a5) -- (a6) -- (a0) -- (a1);
		\draw[thick, loosely dashed] (a3) -- (a4);

\end{tikzpicture}
}
\newcommand{\Fii}{
	\begin{tikzpicture}[scale=0.8, transform shape]
		\draw (-1/2-1/8,0) node{$\lra$};
		\draw (1,0) node(a1)[hrt]{};
		\draw (2,0) node(a2)[rt]{};
		\draw (1,-1/2.5) node{$\bar\alpha_0^\omega$};
		\draw (2,-1/2.5) node{$\bar\alpha_1^\omega$};
		\draw (1,1/3) node{\scriptsize $1$};
		\draw (2,1/3) node{\scriptsize $1$};
		\draw (a1) -- (a2) ;
\end{tikzpicture}
}
\newcommand{\Fiii}{
	\begin{tikzpicture}[scale=0.8, transform shape]
		\draw (-1/2-1/8,0) node{$\lra$};
		\draw (1,0) node(a0)[hrt]{};
		\draw (1+1,0) node(a1)[rt]{};
		\draw (2+1,0) node(a2)[rt]{};
		\draw (3+1,0) node(a3)[rt]{};
		\draw (5+1,0) node(a4)[rt]{};
		\draw (6+1,0) node(a5)[rt]{};
		\draw (7+1,0) node(a6)[rt]{};
		\draw (1,-1/2.5) node{$\bar\alpha_0^\omega$};
		\draw (1+1,-1/2.5) node{$\bar\alpha_1^\omega$};
		\draw (2+1,-1/2.5) node{$\bar\alpha_2^\omega$};
		\draw (3+1,-1/2.5) node{$\bar\alpha_3^\omega$};
		\draw (5+1,-1/2.5) node{$\bar\alpha_{r-2}^\omega$};
		\draw (6+1,-1/2.5) node{$\bar\alpha_{r-1}^\omega$};
		\draw (7+1,-1/2.5) node{$\bar\alpha_r^\omega$};
		\draw (1,1/3) node{\scriptsize $1$};
		\draw (1+1,1/3) node{\scriptsize $1$};
		\draw (2+1,1/3) node{\scriptsize $1$};
		\draw (3+1,1/3) node{\scriptsize $1$};
		\draw (5+1,1/3) node{\scriptsize $1$};
		\draw (6+1,1/3) node{\scriptsize $1$};
		\draw (7+1,1/3) node{\scriptsize $1$};
		\draw (a1) -- (a2) -- (a3)  (a4) -- (a5);
		\draw[double distance = 2pt] (a5) --node[]{\Large $>$} (a6) (a0) --node[]{\Large $<$} (a1);
		\draw[thick, loosely dashed] (a3) -- (a4);
\end{tikzpicture}
}
\newcommand{\Fiv}{
	\begin{tikzpicture}[scale=0.8, transform shape]
		\draw (-1/2-1/8,0) node{$\lra$};
		\draw (1,0) node(a0)[hrt]{};
		\draw (1+1,0) node(a1)[rt]{};
		\draw (2+1,0) node(a2)[rt]{};
		\draw (3+1,0) node(a3)[rt]{};
		\draw (5+1,0) node(a4)[rt]{};
		\draw (6+1,0) node(a5)[rt]{};
		\draw (7+1,0) node(a6)[rt]{};
		\draw (1,-1/2.5) node{$\bar\alpha_0^\omega$};
		\draw (1+1,-1/2.5) node{$\bar\alpha_1^\omega$};
		\draw (2+1,-1/2.5) node{$\bar\alpha_2^\omega$};
		\draw (3+1,-1/2.5) node{$\bar\alpha_3^\omega$};
		\draw (5+1,-1/2.5) node{$\bar\alpha_{r-2}^\omega$};
		\draw (6+1,-1/2.5) node{$\bar\alpha_{r-1}^\omega$};
		\draw (7+1,-1/2.5) node{$\bar\alpha_r^\omega$};
		\draw (1,1/3) node{\scriptsize $1$};
		\draw (1+1,1/3) node{\scriptsize $2$};
		\draw (2+1,1/3) node{\scriptsize $2$};
		\draw (3+1,1/3) node{\scriptsize $2$};
		\draw (5+1,1/3) node{\scriptsize $2$};
		\draw (6+1,1/3) node{\scriptsize $2$};
		\draw (7+1,1/3) node{\scriptsize $2$};
		\draw (a1) -- (a2) -- (a3)  (a4) -- (a5);
		\draw[double distance = 2pt] (a5) --node[]{\Large $>$} (a6) (a0) --node[]{\Large $>$} (a1);
		\draw[thick, loosely dashed] (a3) -- (a4);
\end{tikzpicture}
}
\newcommand{\Fv}{
	\begin{tikzpicture}[scale=0.8, transform shape]
		\draw (-1/2-1/8,0) node{$\lra$};
		\draw (1,0) node(a0)[hrt]{};
		\draw (1+1,0) node(a1)[rt]{};
		\draw (2+1,0) node(a2)[rt]{};
		\draw (3+1,0) node(a3)[rt]{};
		\draw (5+1,0) node(a4)[rt]{};
		\draw (6+1,0) node(a5)[rt]{};
		\draw (7+1,0) node(a6)[rt]{};
		\draw (1,-1/2.5) node{$\bar\alpha_0^\omega$};
		\draw (1+1,-1/2.5) node{$\bar\alpha_1^\omega$};
		\draw (2+1,-1/2.5) node{$\bar\alpha_2^\omega$};
		\draw (3+1,-1/2.5) node{$\bar\alpha_3^\omega$};
		\draw (5+1,-1/2.5) node{$\bar\alpha_{r-2}^\omega$};
		\draw (6+1,-1/2.5) node{$\bar\alpha_{r-1}^\omega$};
		\draw (7+1,-1/2.5) node{$\bar\alpha_r^\omega$};
		\draw (1,1/3) node{\scriptsize $1$};
		\draw (1+1,1/3) node{\scriptsize $2$};
		\draw (2+1,1/3) node{\scriptsize $2$};
		\draw (3+1,1/3) node{\scriptsize $2$};
		\draw (5+1,1/3) node{\scriptsize $2$};
		\draw (6+1,1/3) node{\scriptsize $2$};
		\draw (7+1,1/3) node{\scriptsize $1$};
		\draw (a1) -- (a2) -- (a3)  (a4) -- (a5);
		\draw[double distance = 2pt] (a5) --node[]{\Large $<$} (a6) (a0) --node[]{\Large $>$} (a1);
		\draw[thick, loosely dashed] (a3) -- (a4);
\end{tikzpicture}
}
\newcommand{\Fvi}{
	\begin{tikzpicture}[scale=0.8, transform shape]
		\draw (-1/2-1/8,0) node{$\lra$};
		\draw (1,0) node(a0)[hrt]{};
		\draw (1+1,0) node(a1)[rt]{};
		\draw (2+1,0) node(a2)[rt]{};
		\draw (3+1,0) node(a3)[rt]{};
		\draw (5+1,0) node(a4)[rt]{};
		\draw (6+1,0) node(a5)[rt]{};
		\draw (7+1,0) node(a6)[rt]{};
		\draw (1,-1/2.5) node{$\bar\alpha_0^\omega$};
		\draw (1+1,-1/2.5) node{$\bar\alpha_1^\omega$};
		\draw (2+1,-1/2.5) node{$\bar\alpha_2^\omega$};
		\draw (3+1,-1/2.5) node{$\bar\alpha_3^\omega$};
		\draw (5+1,-1/2.5) node{$\bar\alpha_{r-2}^\omega$};
		\draw (6+1,-1/2.5) node{$\bar\alpha_{r-1}^\omega$};
		\draw (7+1,-1/2.5) node{$\bar\alpha_r^\omega$};
		\draw (1,1/3) node{\scriptsize $1$};
		\draw (1+1,1/3) node{\scriptsize $1$};
		\draw (2+1,1/3) node{\scriptsize $1$};
		\draw (3+1,1/3) node{\scriptsize $1$};
		\draw (5+1,1/3) node{\scriptsize $1$};
		\draw (6+1,1/3) node{\scriptsize $1$};
		\draw (7+1,1/3) node{\scriptsize $1$};
		\draw (a1) -- (a2) -- (a3)  (a4) -- (a5);
		\draw[double distance = 2pt] (a5) --node[]{\Large $>$} (a6) (a0) --node[]{\Large $<$} (a1);
		\draw[thick, loosely dashed] (a3) -- (a4);
\end{tikzpicture}
}
\newcommand{\Fvii}{
	\begin{tikzpicture}[scale=0.8, transform shape]
		\draw (-1/2-1/8,0) node{$\lra$};
		\draw (1,0) node(a0)[hrt]{};
		\draw (1+1,0) node(a1)[rt]{};
		\draw (2+1,0) node(a2)[rt]{};
		\draw (1,-1/2.5) node{$\bar\alpha_0^\omega$};
		\draw (1+1,-1/2.5) node{$\bar\alpha_2^\omega$};
		\draw (2+1,-1/2.5) node{$\bar\alpha_1^\omega$};
		\draw (1,1/3) node{\scriptsize $1$};
		\draw (1+1,1/3) node{\scriptsize $1$};
		\draw (2+1,1/3) node{\scriptsize $1$};
		\draw[double distance = 2pt] (a0) --node[]{\Large $<$} (a1) --node[]{\Large $>$} (a2);
\end{tikzpicture}
}
\newcommand{\Fviii}{
	\begin{tikzpicture}[scale=0.8, transform shape]
		\draw (-1/2-1/8,0) node{$\lra$};
		\draw (1,0) node(a1)[hrt]{};
		\draw (2,0) node(a2)[rt]{};
		\draw (1,-1/2.5) node{$\bar\alpha_0^\omega$};
		\draw (2,-1/2.5) node{$\bar\alpha_1^\omega$};
		\draw (1,1/3) node{\scriptsize $1$};
		\draw (2,1/3) node{\scriptsize $1$};
		\draw (a1) -- (a2) ;
\end{tikzpicture}
}
\newcommand{\Fix}{
	\begin{tikzpicture}[scale=0.8, transform shape]
		\draw (-1/2-1/8,0) node{$\lra$};
		\draw (2-0.7,0.7) node(a0)[hrt]{};
		\draw (2-0.7,-0.7) node(a1)[rt]{};
		\draw (2,0) node(a2)[rt]{};
		\draw (3,0) node(a3)[rt]{};
		\draw (5,0) node(a4)[rt]{};
		\draw (6,0) node(a5)[rt]{};
		\draw (7,0) node(a6)[rt]{};
		\draw (2-0.7,0.7-1/2.5) node{$\bar\alpha_{0}^\omega$};
		\draw (2-0.7,-0.7-1/2.5) node{$\bar\alpha_1^\omega$};
		\draw (2,-1/2.5) node{$\bar\alpha_2^\omega$};
		\draw (3,-1/2.5) node{$\bar\alpha_3^\omega$};
		\draw (5,-1/2.5) node{$\bar\alpha_{r-2}^\omega$};
		\draw (6,-1/2.5) node{$\bar\alpha_{r-1}^\omega$};
		\draw (7,-1/2.5) node{$\bar\alpha_r^\omega$};
		\draw (2-0.7,0.7+1/3) node{\scriptsize $1$};
		\draw (2-0.7,-0.7+1/3) node{\scriptsize $1$};
		\draw (2,1/3) node{\scriptsize $2$};
		\draw (3,1/3) node{\scriptsize $2$};
		\draw (5,1/3) node{\scriptsize $2$};
		\draw (6,1/3) node{\scriptsize $2$};
		\draw (7,1/3) node{\scriptsize $1$};
		\draw (a1) -- (a2) -- (a3)  (a4) -- (a5) (a0) -- (a2);
		\draw[double distance = 2pt] (a5) --node[]{\Large $<$} (a6);
		\draw[thick, loosely dashed] (a3) -- (a4);
\end{tikzpicture}
}
\newcommand{\Fx}{
	\begin{tikzpicture}[scale=0.8, transform shape]
		\draw (-1/2-1/8,0) node{$\lra$};
		\draw (1,0) node(a0)[hrt]{};
		\draw (1+1,0) node(a1)[rt]{};
		\draw (2+1,0) node(a2)[rt]{};
		\draw (3+1,0) node(a3)[rt]{};
		\draw (5+1,0) node(a4)[rt]{};
		\draw (6+1,0) node(a5)[rt]{};
		\draw (7+1,0) node(a6)[rt]{};
		\draw (1,-1/2.5) node{$\bar\alpha_0^\omega$};
		\draw (1+1,-1/2.5) node{$\bar\alpha_1^\omega$};
		\draw (2+1,-1/2.5) node{$\bar\alpha_2^\omega$};
		\draw (3+1,-1/2.5) node{$\bar\alpha_3^\omega$};
		\draw (5+1,-1/2.5) node{$\bar\alpha_{r-2}^\omega$};
		\draw (6+1,-1/2.5) node{$\bar\alpha_{r-1}^\omega$};
		\draw (7+1,-1/2.5) node{$\bar\alpha_r^\omega$};
		\draw (1,1/3) node{\scriptsize $1$};
		\draw (1+1,1/3) node{\scriptsize $1$};
		\draw (2+1,1/3) node{\scriptsize $1$};
		\draw (3+1,1/3) node{\scriptsize $1$};
		\draw (5+1,1/3) node{\scriptsize $1$};
		\draw (6+1,1/3) node{\scriptsize $1$};
		\draw (7+1,1/3) node{\scriptsize $1$};
		\draw (a1) -- (a2) -- (a3)  (a4) -- (a5);
		\draw[double distance = 2pt] (a5) --node[]{\Large $>$} (a6) (a0) --node[]{\Large $<$} (a1);
		\draw[thick, loosely dashed] (a3) -- (a4);
\end{tikzpicture}
}
\newcommand{\Fxi}{
	\begin{tikzpicture}[scale=0.8, transform shape]
		\draw (-1/2-1/8,0) node{$\lra$};
		\draw (1,0) node(a0)[hrt]{};
		\draw (1+1,0) node(a1)[rt]{};
		\draw (2+1,0) node(a2)[rt]{};
		\draw (1,-1/2.5) node{$\bar\alpha_0^\omega$};
		\draw (1+1,-1/2.5) node{$\bar\alpha_1^\omega$};
		\draw (2+1,-1/2.5) node{$\bar\alpha_2^\omega$};
		\draw (1,1/3) node{\scriptsize $1$};
		\draw (1+1,1/3) node{\scriptsize $2$};
		\draw (2+1,1/3) node{\scriptsize $1$};
		\draw[double distance = 3pt] (a1) --node[]{\Large $<$} (a2);
		\draw (a0) -- (a1) -- (a2);
\end{tikzpicture}
}
\newcommand{\Fxii}{
	\begin{tikzpicture}[scale=0.8, transform shape]
		\draw (-1/2-1/8,0) node{$\lra$};
		\draw (1,0) node(a0)[hrt]{};
		\draw (1+1,0) node(a1)[rt]{};
		\draw (2+1,0) node(a2)[rt]{};
		\draw (3+1,0) node(a3)[rt]{};
		\draw (4+1,0) node(a4)[rt]{};
		\draw (1,-1/2.5) node{$\bar\alpha_0^\omega$};
		\draw (1+1,-1/2.5) node{$\bar\alpha_1^\omega$};
		\draw (2+1,-1/2.5) node{$\bar\alpha_2^\omega$};
		\draw (3+1,-1/2.5) node{$\bar\alpha_3^\omega$};
		\draw (4+1,-1/2.5) node{$\bar\alpha_4^\omega$};
		\draw (1,1/3) node{\scriptsize $1$};
		\draw (1+1,1/3) node{\scriptsize $2$};
		\draw (2+1,1/3) node{\scriptsize $3$};
		\draw (3+1,1/3) node{\scriptsize $2$};
		\draw (4+1,1/3) node{\scriptsize $1$};
		\draw (a0) -- (a1) -- (a2) (a3) -- (a4);
		\draw[double distance = 2pt] (a2) --node[]{\Large $<$} (a3);
\end{tikzpicture}
}

\newcommand{\FFO}{
	\begin{tikzpicture}[scale=0.8, transform shape]
		\draw (1,0) node(a1)[rt]{};
		\draw (2,0) node(a2)[rt]{};
		\draw (3,0) node(a3)[rt]{};
		\draw (5,0) node(a4)[rt]{};
		\draw (6,0) node(a5)[rt]{};
		\draw (7,0) node(a6)[rt]{};
		\draw (4,1.3) node(a0)[hrt]{};
		\draw (a1) -- (a2) -- (a3)  (a4) -- (a5) -- (a6) -- (a0) -- (a1);
		\draw[thick, loosely dashed] (a3) -- (a4);
		\draw[densely dotted, {Stealth}-{Stealth}, shorten > = 1.5pt, shorten < = 1.5pt] (a0) to [out=180, in=45] (a1);
		\draw[densely dotted, {Stealth}-{Stealth}, shorten > = 1.5pt, shorten < = 1.5pt] (a0) to [out=0, in=180-45] (a6);
		\draw[densely dotted, {Stealth}-{Stealth}, shorten > = 1.5pt, shorten < = 1.5pt] (a1) to [out=-40, in=180+40] (a2);
		\draw[densely dotted, {Stealth}-{Stealth}, shorten > = 1.5pt, shorten < = 1.5pt] (a2) to [out=-40, in=180+40] (a3);
		\draw[densely dotted, {Stealth}-{Stealth}, shorten > = 1.5pt, shorten < = 1.5pt] (a4) to [out=-40, in=180+40] (a5);
		\draw[densely dotted, {Stealth}-{Stealth}, shorten > = 1.5pt, shorten < = 1.5pt] (a5) to [out=-40, in=180+40] (a6);
	\end{tikzpicture}
}
\newcommand{\FFi}{
	\begin{tikzpicture}[scale=0.8, transform shape]
		\draw (1,2);
		\draw (1,0) node(a1)[rt]{};
		\draw (2,0) node(a2)[rt]{};
		\draw (1+4,0) node(a3)[rt]{};
		\draw (1+3,0) node(a35)[rt]{};
		\draw (1+5,0) node(a4)[rt]{};
		\draw (1+6,0) node(a5)[rt]{};
		\draw (2+7,0) node(a6)[rt]{};
		\draw (1+4,1.3) node(a0)[hrt]{};
		\draw (a1) -- (a2)  (a35) -- (a3) -- (a4) -- (a5) (a6) -- (a0) -- (a1);
		\draw[thick, loosely dashed] (a2) -- (a35) (a5) -- (a6);
		
		\draw[densely dotted, {Stealth}-{Stealth}, shorten > = 1.5pt, shorten < = 1.5pt] (a0) to [out=-90, in=180-90] (a3);
		\draw[densely dotted, {Stealth}-{Stealth}, shorten > = 1.5pt, shorten < = 1.5pt] (a1) to [out=-30, in=180+30] (a4);
		\draw[densely dotted, {Stealth}-{Stealth}, shorten > = 1.5pt, shorten < = 1.5pt] (a2) to [out=-30, in=180+30] (a5);
		\draw[densely dotted, {Stealth}-{Stealth}, shorten > = 1.5pt, shorten < = 1.5pt] (a35) to [out=-30, in=180+30] (a6);
	\end{tikzpicture}
}
\newcommand{\FFii}{
	\begin{tikzpicture}[scale=0.8, transform shape]
		\draw (1,5/7);
		\draw (1,0) node(a0)[hrt]{};
		\draw (1+1,0) node(a1)[rt]{};
		\draw (2+1,0) node(a2)[rt]{};
		\draw (1,-1/2.5) node{$\alpha_0$};
		\draw (1+1,-1/2.5) node{$\alpha_2$};
		\draw (2+1,-1/2.5) node{$\alpha_1$};
		\draw[double distance = 2pt] (a0) --node[]{\Large $>$} (a1) --node[]{\Large $<$} (a2);
		\draw[densely dotted, {Stealth}-{Stealth}, shorten > = 1.5pt, shorten < = 1.5pt] (a0) to [out=50, in=180-50] (a2);
	\end{tikzpicture}
}
\newcommand{\FFiii}{
	\begin{tikzpicture}[scale=0.8, transform shape]
		\draw (2,3/2);
		\draw (2-0.7,0.7) node(a0)[hrt]{};
		\draw (2-0.7,-0.7) node(a1)[rt]{};
		\draw (2,0) node(a2)[rt]{};
		\draw (3,0) node(a3)[rt]{};
		\draw (5,0) node(a4)[rt]{};
		\draw (6,0) node(a5)[rt]{};
		\draw (7,0) node(a6)[rt]{};
		\draw (a1) -- (a2) -- (a3)  (a4) -- (a5) (a0) -- (a2);
		\draw[double distance = 2pt] (a5) --node[]{\Large $>$} (a6);
		\draw[thick, loosely dashed] (a3) -- (a4);
		\draw[densely dotted, {Stealth}-{Stealth}, shorten > = 1.5pt, shorten < = 1.5pt] (a0) to [out=180+40, in=180-40] (a1);
	\end{tikzpicture}
	}
\newcommand{\FFiv}{
	\begin{tikzpicture}[scale=0.8, transform shape]
		\draw (1,0) node(a0)[hrt]{};
		\draw (1+1,0) node(a1)[rt]{};
		\draw (2+1,0) node(a2)[rt]{};
		
		\draw (3+2,0) node(a3)[rt]{};
		\draw (4+2,0) node(a35)[rt]{};
		\draw (5+3,0) node(a4)[rt]{};
		\draw (6+3,0) node(a5)[rt]{};
		\draw (7+3,0) node(a6)[rt]{};
		\draw (1,-1/3) node{$\alpha_0$};
		\draw (1+1,-1/3) node{$\alpha_1$};
		\draw (2+1,-1/3) node{$\alpha_2$};
		\draw (3+2,-1/3) node{$\alpha_{\frac{\ell-1}{2}}$};
		\draw (4+2,-1/3) node{$\alpha_{\frac{\ell+1}{2}}$};
		\draw (5+3,-1/3) node{$\alpha_{\ell-2}$};
		\draw (6+3,-1/3) node{$\alpha_{\ell-1}$};
		\draw (7+3,-1/3) node{$\alpha_\ell$};
		\draw (a1) -- (a2) (a3) -- (a35) (a4) -- (a5);
		\draw[double distance = 2pt] (a5) --node[]{\Large $<$} (a6) (a0) --node[]{\Large $>$} (a1);
		\draw[thick, loosely dashed] (a2) -- (a3) (a35) -- (a4);
		\draw[densely dotted, {Stealth}-{Stealth}, shorten > = 1.5pt, shorten < = 1.5pt] (a0) to [out=30, in=180-30] (a6);
		\draw[densely dotted, {Stealth}-{Stealth}, shorten > = 1.5pt, shorten < = 1.5pt] (a1) to [out=30, in=180-30] (a5);
		\draw[densely dotted, {Stealth}-{Stealth}, shorten > = 1.5pt, shorten < = 1.5pt] (a2) to [out=30, in=180-30] (a4);
		\draw[densely dotted, {Stealth}-{Stealth}, shorten > = 1.5pt, shorten < = 1.5pt] (a3) to [out=50, in=180-50] (a35);
	\end{tikzpicture}
	}
\newcommand{\FFv}{
	\begin{tikzpicture}[scale=0.8, transform shape]
		\draw (1,0) node(a0)[hrt]{};
		\draw (1+1,0) node(a1)[rt]{};
		\draw (2+1,0) node(a2)[rt]{};
		
		\draw (3+2,0) node(a3)[rt]{};
		\draw (4+2,0) node(a355)[rt]{};
		\draw (4+2+1,0) node(a35)[rt]{};
		\draw (5+3+1,0) node(a4)[rt]{};
		\draw (6+3+1,0) node(a5)[rt]{};
		\draw (7+3+1,0) node(a6)[rt]{};
		\draw (1,-1/3) node{$\alpha_0$};
		\draw (1+1,-1/3) node{$\alpha_1$};
		\draw (2+1,-1/3) node{$\alpha_2$};
		\draw (3+2,-1/3) node{$\alpha_{\frac{\ell}{2}-1}$};
		\draw (4+2+1,-1/3) node{$\alpha_{\frac{\ell}{2}+1}$};
		\draw (4+2,-1/3) node{$\alpha_{\frac{\ell}{2}}$};
		\draw (5+3+1,-1/3) node{$\alpha_{\ell-2}$};
		\draw (6+3+1,-1/3) node{$\alpha_{\ell-1}$};
		\draw (7+3+1,-1/3) node{$\alpha_\ell$};
		\draw (a1) -- (a2) (a3) -- (a355) -- (a35) (a4) -- (a5);
		\draw[double distance = 2pt] (a5) --node[]{\Large $<$} (a6) (a0) --node[]{\Large $>$} (a1);
		\draw[thick, loosely dashed] (a2) -- (a3) (a35) -- (a4);
		\draw[densely dotted, {Stealth}-{Stealth}, shorten > = 1.5pt, shorten < = 1.5pt] (a0) to [out=30, in=180-30] (a6);
		\draw[densely dotted, {Stealth}-{Stealth}, shorten > = 1.5pt, shorten < = 1.5pt] (a1) to [out=30, in=180-30] (a5);
		\draw[densely dotted, {Stealth}-{Stealth}, shorten > = 1.5pt, shorten < = 1.5pt] (a2) to [out=30, in=180-30] (a4);
		\draw[densely dotted, {Stealth}-{Stealth}, shorten > = 1.5pt, shorten < = 1.5pt] (a3) to [out=50, in=180-50] (a35);
	\end{tikzpicture}
	}
\newcommand{\FFvi}{
	\begin{tikzpicture}[scale=0.8, transform shape]
		\draw (3,3/2);
		\draw (2-0.7,0.7) node(a0)[hrt]{};
		\draw (2-0.7,-0.7) node(a1)[rt]{};
		\draw (2,0) node(a2)[rt]{};
		\draw (3,0) node(a3)[rt]{};
		\draw (5,0) node(a4)[rt]{};
		\draw (5.7,0.7) node(a5)[rt]{};
		\draw (5.7,-0.7) node(a6)[rt]{};
		\draw (2-0.7,0.7-1/3) node{$\alpha_{0}$};
		\draw (2-0.7,-0.7-1/3) node{$\alpha_1$};
		\draw (2,-1/3) node{$\alpha_2$};
		\draw (3,-1/3) node{$\alpha_3$};
		\draw (5-0.1,-1/3) node{$\alpha_{\ell-2}$};
		\draw (5.7+0.1,0.7-1/3) node{$\alpha_{\ell-1}$};
		\draw (5.7,-0.7-1/3) node{$\alpha_\ell$};
		\draw (a1) -- (a2) -- (a3)  (a6) -- (a4) -- (a5) (a0) -- (a2);
		\draw[thick, loosely dashed] (a3) -- (a4);
		\draw[densely dotted, {Stealth}-{Stealth}, shorten > = 1.5pt, shorten < = 1.5pt] (a0) to [out=180+40, in=180-40] (a1);
		\draw[densely dotted, {Stealth}-{Stealth}, shorten > = 1.5pt, shorten < = 1.5pt] (a5) to [out=-40, in=40] (a6);
	\end{tikzpicture}
}
\newcommand{\FFviii}{
	\begin{tikzpicture}[scale=0.8, transform shape]
		\draw (3,3/2);
		\draw (2-0.7,0.7) node(a0)[hrt]{};
		\draw (2-0.7,-0.7) node(a1)[rt]{};
		\draw (2,0) node(a2)[rt]{};
		\draw (1+3-2,0) node(a35)[rt]{};
		\draw (2+3-2,0) node(a355)[rt]{};
		\draw (-2+5,0) node(a4)[rt]{};
		\draw (-2+5.7,0.7) node(a5)[rt]{};
		\draw (-2+5.7,-0.7) node(a6)[rt]{};
		\draw (2-0.7,0.7-1/3) node{$\alpha_{0}$};
		\draw (2-0.7,-0.7-1/3) node{$\alpha_1$};
		\draw (2,-1/3) node{$\alpha_2$};
		\draw (2+3-2,-1/3) node{$\alpha_{3}$};
		\draw (-2+5.7+0.1,0.7-1/3) node{$\alpha_{4}$};
		\draw (-2+5.7,-0.7-1/3) node{$\alpha_5$};
		\draw (a1) -- (a2) (a35) --  (a355)  (a6) -- (a4) -- (a5) (a0) -- (a2);
		\draw[thick, loosely dashed] (a2) -- (a35) (a355) -- (a4);
		\draw[densely dotted, {Stealth}-{Stealth}, shorten > = 1.5pt, shorten < = 1.5pt] (a0) to [out=180+40, in=180-40] (a1);
		\draw[densely dotted, {Stealth}-{Stealth}, shorten > = 1.5pt, shorten < = 1.5pt] (a5) to [out=-40, in=40] (a6);
		\draw[densely dotted, {Stealth}-{Stealth}, shorten > = 1.5pt, shorten < = 1.5pt] (a0) to [out=10, in=180-10] (a5);
		\draw[densely dotted, {Stealth}-{Stealth}, shorten > = 1.5pt, shorten < = 1.5pt] (a35) to [out=50, in=180-50] (a355);
	\end{tikzpicture}
}
\newcommand{\FFvii}{
	\begin{tikzpicture}[scale=0.8, transform shape]
		\draw (3,3/2);
		\draw (2-0.7,0.7) node(a0)[hrt]{};
		\draw (2-0.7,-0.7) node(a1)[rt]{};
		\draw (2,0) node(a2)[rt]{};
		\draw (1+3-2,0) node(a35)[rt]{};
		\draw (1+3-2,0) node(a355)[rt]{};
		\draw (-3+5,0) node(a4)[rt]{};
		\draw (-3+5.7,0.7) node(a5)[rt]{};
		\draw (-3+5.7,-0.7) node(a6)[rt]{};
		\draw (2-0.7,0.7-1/3) node{$\alpha_{0}$};
		\draw (2-0.7,-0.7-1/3) node{$\alpha_1$};
		\draw (2,-1/3) node{$\alpha_2$};
		\draw (-3+5.7+0.1,0.7-1/3) node{$\alpha_{3}$};
		\draw (-3+5.7,-0.7-1/3) node{$\alpha_4$};
		\draw (a1) -- (a2) (a35) --  (a355)  (a6) -- (a4) -- (a5) (a0) -- (a2);
		\draw[thick, loosely dashed] (a2) -- (a35) (a355) -- (a4);
		\draw[densely dotted, {Stealth}-{Stealth}, shorten > = 1.5pt, shorten < = 1.5pt] (a0) to [out=10, in=180-10] (a5);
		\draw[densely dotted, {Stealth}-{Stealth}, shorten > = 1.5pt, shorten < = 1.5pt] (a1) to [out=-10, in=-180+10] (a6);
	\end{tikzpicture}
}
\newcommand{\FFix}{
	\begin{tikzpicture}[scale=0.8, transform shape]
		\draw (2-0.7,0.7) node(a0)[hrt]{};
		\draw (2-0.7,-0.7) node(a1)[rt]{};
		\draw (2,0) node(a2)[rt]{};
		\draw (1+3,0) node(a35)[rt]{};
		\draw (2+3,0) node(a3)[rt]{};
		\draw (3+3,0) node(a355)[rt]{};
		\draw (3+5,0) node(a4)[rt]{};
		\draw (3+5.7,0.7) node(a5)[rt]{};
		\draw (3+5.7,-0.7) node(a6)[rt]{};
		\draw (2-0.7,0.7-1/3) node{$\alpha_{0}$};
		\draw (2-0.7,-0.7-1/3) node{$\alpha_1$};
		\draw (2,-1/3) node{$\alpha_2$};
		\draw (1+3,-1/3) node{$\alpha_{\frac{\ell}{2}-1}$};
		\draw (2+3,-1/3) node{$\alpha_{\frac{\ell}{2}}$};
		\draw (3+3,-1/3) node{$\alpha_{\frac{\ell}{2}+1}$};
		\draw (3+5-0.1,-1/3) node{$\alpha_{\ell-2}$};
		\draw (3+5.7+0.1,0.7-1/3) node{$\alpha_{\ell-1}$};
		\draw (3+5.7,-0.7-1/3) node{$\alpha_\ell$};
		\draw (a1) -- (a2) (a35) -- (a3) -- (a355)  (a6) -- (a4) -- (a5) (a0) -- (a2);
		\draw[thick, loosely dashed] (a2) -- (a35) (a355) -- (a4);
		\draw[densely dotted, {Stealth}-{Stealth}, shorten > = 1.5pt, shorten < = 1.5pt] (a0) to [out=10, in=180-10] (a5);
		\draw[densely dotted, {Stealth}-{Stealth}, shorten > = 1.5pt, shorten < = 1.5pt] (a1) to [out=-10, in=-180+10] (a6);
		\draw[densely dotted, {Stealth}-{Stealth}, shorten > = 1.5pt, shorten < = 1.5pt] (a2) to [out=20, in=180-20] (a4);
		\draw[densely dotted, {Stealth}-{Stealth}, shorten > = 1.5pt, shorten < = 1.5pt] (a35) to [out=30, in=180-30] (a355);
	\end{tikzpicture}
}
\newcommand{\FFx}{
	\begin{tikzpicture}[scale=0.8, transform shape]
		\draw (2-0.7,0.7) node(a0)[hrt]{};
		\draw (2-0.7,-0.7) node(a1)[rt]{};
		\draw (2,0) node(a2)[rt]{};
		\draw (1+3,0) node(a35)[rt]{};
		\draw (2+3,0) node(a355)[rt]{};
		\draw (2+5,0) node(a4)[rt]{};
		\draw (2+5.7,0.7) node(a5)[rt]{};
		\draw (2+5.7,-0.7) node(a6)[rt]{};
		\draw (2-0.7,0.7-1/3) node{$\alpha_{0}$};
		\draw (2-0.7,-0.7-1/3) node{$\alpha_1$};
		\draw (2,-1/3) node{$\alpha_2$};
		\draw (1+3,-1/3) node{$\alpha_{\frac{\ell-1}{2}}$};
		\draw (2+3,-1/3) node{$\alpha_{\frac{\ell+1}{2}}$};
		\draw (2+5-0.1,-1/3) node{$\alpha_{\ell-2}$};
		\draw (2+5.7+0.1,0.7-1/3) node{$\alpha_{\ell-1}$};
		\draw (2+5.7,-0.7-1/3) node{$\alpha_\ell$};
		\draw (a1) -- (a2) (a35) --  (a355)  (a6) -- (a4) -- (a5) (a0) -- (a2);
		\draw[thick, loosely dashed] (a2) -- (a35) (a355) -- (a4);
		\draw[densely dotted, {Stealth}-{Stealth}, shorten > = 1.5pt, shorten < = 1.5pt] (a0) to [out=180+40, in=180-40] (a1);
		\draw[densely dotted, {Stealth}-{Stealth}, shorten > = 1.5pt, shorten < = 1.5pt] (a5) to [out=-40, in=40] (a6);
		\draw[densely dotted, {Stealth}-{Stealth}, shorten > = 1.5pt, shorten < = 1.5pt] (a0) to [out=10, in=180-10] (a5);
		\draw[densely dotted, {Stealth}-{Stealth}, shorten > = 1.5pt, shorten < = 1.5pt] (a2) to [out=20, in=180-20] (a4);
		\draw[densely dotted, {Stealth}-{Stealth}, shorten > = 1.5pt, shorten < = 1.5pt] (a35) to [out=50, in=180-50] (a355);
	\end{tikzpicture}
}
\newcommand{\FFxi}{
	\begin{tikzpicture}[scale=0.8, transform shape]
		\draw (1,0) node(a1)[rt]{};
		\draw (3,-1) node(a2)[rt]{};
		\draw (2,0) node(a3)[rt]{};
		\draw (3,0) node(a4)[rt]{};
		\draw (4,0) node(a5)[rt]{};
		\draw (5,0) node(a6)[rt]{};
		\draw (3,-2) node(a0)[hrt]{};
		\draw (1,1/3) node{$\alpha_1$};
		\draw (3+1/4,-1-1/3) node{$\alpha_2$};
		\draw (2,1/3) node{$\alpha_3$};
		\draw (3+1/4,1/3) node{$\alpha_4$};
		\draw (4,1/3) node{$\alpha_5$};
		\draw (5,1/3) node{$\alpha_6$};
		\draw (3,-2-1/3) node{$\alpha_0$};
		\draw (a1) -- (a3) -- (a4) -- (a5) -- (a6) (a4) -- (a2) -- (a0);
		
		\draw[densely dotted, {Stealth}-{Stealth}, shorten > = 1.5pt, shorten < = 1.5pt] (a0) to [out=180, in=180+90] (a1);
		\draw[densely dotted, {Stealth}-{Stealth}, shorten > = 1.5pt, shorten < = 1.5pt] (a0) to [out=0, in=180+90] (a6);
		\draw[densely dotted, {Stealth}-{Stealth}, shorten > = 1.5pt, shorten < = 1.5pt] (a2) to [out=180, in=180+90] (a3);
		\draw[densely dotted, {Stealth}-{Stealth}, shorten > = 1.5pt, shorten < = 1.5pt] (a2) to [out=0, in=180+90] (a5);
	\end{tikzpicture}
}
\newcommand{\FFxii}{
	\begin{tikzpicture}[scale=0.8, transform shape]
		\draw (0,0) node(a0)[hrt]{};
		\draw (1,0) node(a1)[rt]{};
		\draw (3,-1) node(a2)[rt]{};
		\draw (2,0) node(a3)[rt]{};
		\draw (3,0) node(a4)[rt]{};
		\draw (4,0) node(a5)[rt]{};
		\draw (5,0) node(a6)[rt]{};
		\draw (6,0) node(a7)[rt]{};
		\draw (0,-1/3) node{$\alpha_0$};
		\draw (1,-1/3) node{$\alpha_1$};
		\draw (3,-1-1/3) node{$\alpha_2$};
		\draw (2,-1/3) node{$\alpha_3$};
		\draw (3+1/4,-1/3) node{$\alpha_4$};
		\draw (4,-1/3) node{$\alpha_5$};
		\draw (5,-1/3) node{$\alpha_6$};
		\draw (6,-1/3) node{$\alpha_7$};
		\draw (a0) -- (a1) -- (a3) -- (a4) -- (a5) -- (a6) -- (a7) (a4) -- (a2);
		\draw[densely dotted, {Stealth}-{Stealth}, shorten > = 1.5pt, shorten < = 1.5pt] (a0) to [out=30, in=180-30] (a7);
		\draw[densely dotted, {Stealth}-{Stealth}, shorten > = 1.5pt, shorten < = 1.5pt] (a1) to [out=30, in=180-30] (a6);
		\draw[densely dotted, {Stealth}-{Stealth}, shorten > = 1.5pt, shorten < = 1.5pt] (a3) to [out=30, in=180-30] (a5);
	\end{tikzpicture}
}

\allowdisplaybreaks

\newcommand{\varpivee}{\varpi^\vee}

\DeclareMathOperator{\height}{ht}
\DeclareMathOperator{\sttr}{re}
\newcommand{\re}{{\sttr}}

\newenvironment{eenumeratea}{\begin{enumerate}[label=(\alph*), leftmargin=*]
	}{\end{enumerate}}
\newenvironment{eenumerateA}{\begin{enumerate}[label=(\Alph*), leftmargin=*]
	}{\end{enumerate}}

\tikzstyle{wv} = [circle, draw, thin, inner sep=1.5pt, minimum size=1.5pt, fill = white]
\tikzstyle{bv} = [circle, draw, thin, inner sep=1.5pt, minimum size=1.5pt, fill = black]
\tikzstyle{vv} = [circle, draw, thin, inner sep=1.5pt, minimum size=1.5pt, fill = white]

\numberwithin{equation}{section}

\title[]{Root systems constructed by folding of the extended Dynkin diagrams}
\author[Ryo Uchiumi]{Ryo Uchiumi}
\date{\today}
\thanks{The author was supported by JSPS KAKENHI, Grant Number 25KJ1735}

\begin{document}

\begin{abstract}
	The extended affine Weyl group of a root system is the semidirect product of the corresponding Weyl group by its coweight lattice.
	The stabilizer subgroup of the extended affine Weyl group with respect to the corresponding fundamental alcove induces a subgroup of automorphisms of the extended Dynkin diagram.
	In this paper, we construct a finite root system by folding by the elements of the subgroup.
\end{abstract}

\maketitle

\tableofcontents



\section{Introduction}

Let $\PHI$ be an irreducible reduced root system in a Euclidean space $E$ equipped with innner product $(\cdot,\cdot)$.
The Weyl group $W$ of $\PHI$ is the group generated by reflections with respect to roots of $\PHI$.
The affine Weyl gourp $W_\aff$ is the group generated by affine reflections, and is isomorphic to the semidirect product of the Weyl group $W$ by the coroot lattice $\veeQ$.
On the other hand, the semidirect product of $W$ by the coweight lattice $Z$ is called the \textbf{extended affine Weyl group}, denoted by $\widehat{W_\aff}$.
It is clear that $W_\aff$ is a subgroup of $\widehat{W_\aff}$, but $\widehat{W_\aff}$ is not a Coxeter group.
These groups are subgroups of the group of affine transformations of $E$.

Let $\widehat{\OMEGA}$ be a stabilizer subgroup of $\widehat{W_\aff}$ with respect to the fundamental alcove $A_\circ$ of $\PHI$, whose closure is a fundamental domain for the action of $W_\aff$ on $E$.
It is well known that $\widehat{\OMEGA}$ is isomorphic to the quotient group $Z/\veeQ$.
In \cite{Garnier}, Garnier provided a fundamental domain for the action of $\widehat{W_\aff}$ on $E$, and in the process, clarified the structure of $\widehat{\OMEGA}$ and its image $\OMEGA = \pi(\widehat{\OMEGA})$ under the projection $\pi : \widehat{W_\aff} \lra W$ (he claimed that it was first obtained by Komrakov and Premet \cite{KomrakovPremet}).
In particular, it has been shown that $\OMEGA$ is a normal subgroup of the group of automorphisms of the extended Dynkin diagram of $\PHI$.

In this paper, we construct a finite root system 
\begin{align}
	\PHI_\re^\omega = \bigset{\beta^\omega}{\beta \in \PHI,\ \beta^\omega \neq 0},
\end{align}
where 
\begin{align}
	\beta^\omega = \frac{1}{o(\omega)}\sum_{t=1}^{o(\omega)}\omega^t(\beta)
\end{align}
and $o(\omega)$ denotes the order of $\omega$.
The configuration of $\PHI_\re^\omega$ corresponds to the folding of the extended Dynkin diagram of $\PHI$ by $\omega$.
Many studies, such as \cite[Section 9]{Carterbook}, \cite{Coelho, FeliksonTumarkinYildirim, KhastgirSasaki, Stembridge, Zuber}, are known about the construction of root systems and Lie algebras via the folding of the extended Dynkin diagrams.
This paper also discusses non-simply-laced cases and constructs root systems that differ slightly from those.
For each root system $\PHI$ and $\omega \in \OMEGA$, we identify the type of $\PHI_\re^\omega$ and clarify which roots disappear due to folding.

Let $\tilde{\alpha}$ be the highest root of $\PHI$ and $\alpha_0 \ceq -\tilde{\alpha}$.
For a basis (simple roots) $\DELTA$ of $\PHI$, define a set
\begin{align}
	\DELTA^\omega = \bigset{\alpha^\omega}{\alpha \in \DELTA,\ \omega^t(\alpha) \neq \alpha_0 \tforall t \in \mathbb{Z}}.
\end{align}
The main result in this paper is the following:
\begin{theorem}[see \cref{mainthm}, \cref{basistheorem}, \cref{typetheorem}]
	Let $\PHI$ be an irreducible reduced root system with a basis $\DELTA$.
	Then $\PHI_\re^\omega$ is an irreducible root system (not necessarily reduced), and $\DELTA^\omega$ is a basis of $\PHI_\re^\omega$.
\end{theorem}

The organization of this paper is as follows:
In \cref{sec2}, we primarily provide definitions and notations.
In particular, we will describe groups $\widehat{\OMEGA}$ and $\OMEGA$ based on \cite{Garnier} in \cref{sec2.3}.
In \cref{sec3},  we prove that $\PHI_\re^\omega$ is a root system (\cref{mainthm}). 
The proof is given in \cref{sec3.3}.
In particular, in \cref{sec3.1}, we will discuss the Euclidean space to which $\PHI_\re^\omega$ belongs.
In \cref{sec4}, we will show that $\DELTA^\omega$ is a basis of $\PHI_\re^\omega$.
The proof is based on the classification of root systems.
Additionally, we will clarify the roots that disappear due to folding.
In \cref{sec5}, we discuss folding of the extended Dynkin diagrams and determine the type of $\PHI_\re^\omega$.
A list of type of $\PHI_\re^\omega$ is given in \cref{tableP}.
The figures are attached at the end of this paper.

\section{Preliminaries} \label{sec2}

\subsection{Definition of root systems and some properties}\ 

\noindent
In this section, we recall definitions and properties of root systems, referring to \cite{Bourbaki}.

Let $E$ be a Euclidean space of dimension $\ell$ with inner product $(\cdot,\cdot)$.
Let $\alpha \in E$ be a non-zero element.
We define an element $\alpha^\vee \in E$ by 
\begin{align}
	\alpha^\vee \ceq \frac{2\alpha}{(\alpha,\alpha)}.
\end{align}
A \textbf{reflection} $s_\alpha$ is a linear map on $E$ defined by
\begin{align}
	s_\alpha(x) \ceq x - (\alpha^\vee,x)\alpha
\end{align}
for all $x \in E$.
A subspace 
\begin{align}
	H_\alpha \ceq \bigset{x \in E}{(\alpha,x) = 0}
\end{align}
is called a \textbf{hyperplane} orthogonal to $\alpha$.
It is easy to see that the reflection $s_\alpha$ fixes all elements of the hyperplane $H_\alpha$.
The hyperplane $H_\alpha$ divides $E$ into two (open) half-spaces $H_\alpha^+$ and $H_\alpha^-$:
\begin{align}
	H_\alpha^+ = \bigset{x \in E}{(\alpha,x) > 0},\qquad 
	H_\alpha^- = \bigset{x \in E}{(\alpha,x) < 0}.
\end{align}

A set $\PHI \subseteq E$ is a (finite) \textbf{root system} in $E$ if it satisfies the following conditions:
\begin{enumerate}[label={(R\arabic*)}]
	\item\label{R1} $\PHI$ is finite, does not contain $0$, and generates $E$;
	\item\label{R2} $s_\alpha(\beta) \in \PHI$ for all $\alpha,\beta \in \PHI$;
	\item\label{R3} $(\alpha^\vee,\beta) \in \mathbb{Z}$ for all $\alpha,\beta \in \PHI$.
\end{enumerate}
An element of $\PHI$ is called a \textbf{root} of $\PHI$.
The dimension of $E$ ($= \ell$) is called the \textbf{rank} of $\PHI$.
Note that the emptyset $\PHI = \varnothing$ is a root system in the trivial Euclidean space $E = \{0\}$.

A root system $\PHI \subseteq E$ is \textbf{reducible} if there exist root systems $\PHI_1 \subseteq E_1$ and $\PHI_2 \subseteq E_2$ such that $\PHI = \PHI_1 \sqcup \PHI_2$ and $E = E_1 \oplus E_2$.
A root system $\PHI \subseteq E$ is \textbf{irreducible} if it is not reducible.
We assume that a root system $\PHI$ is irreducible unless otherwise stated.

A root $\alpha \in \PHI$ is \textbf{indivisible} if $\frac{1}{2}\alpha \not\in \PHI$.
A root system $\PHI$ is \textbf{reduced} if all roots of $\PHI$ are indivisible.

\begin{proposition}[{\cite[\rom{6}, \S1.3, Proposition 8]{Bourbaki}}] \label{prop2.1}
	Let $\alpha,\beta \in \PHI$ are non-proportional roots.
	If $(\alpha,\alpha) \leq (\beta,\beta)$, then
	\begin{align}
		(\alpha,\beta) = \frac{(\beta,\beta)}{2} \tor -\frac{(\beta,\beta)}{2}.
	\end{align}
\end{proposition}

\begin{proposition}[{\cite[\rom{6}, \S1.3, Theorem 1]{Bourbaki}}] \label{prop2.2}
	Let $\alpha,\beta \in \PHI$.
	\begin{eenumerate}
		\item If $(\alpha,\beta) > 0$ and $\alpha \neq \beta$, then $\alpha - \beta \in \PHI$.
		\item If $(\alpha,\beta) < 0$ and $\alpha \neq -\beta$, then $\alpha + \beta \in \PHI$.
	\end{eenumerate}
\end{proposition}

A subgroup of the automorphisms of $E$ generated by the reflections $\Bigset{s_\alpha}{\alpha \in \PHI}$ is called the \textbf{Weyl group} of $\PHI$ and denoted by $W$.

The subset $\DELTA \ceq \{\alpha_1,\ldots,\alpha_\ell\} \subseteq \PHI$ is a \textbf{basis} of $\PHI$ if it satisfies the following conditions (\cite[\rom{6},\S1,7, Corollary 3]{Bourbaki}):
\begin{enumerate}[label=(B\arabic*)]
	\item $\DELTA$ is a basis for $E$;
	\item Every root of $\PHI$ can be expressed as the linear combination with non-negative or  non-positive integer coefficients of $\DELTA$;
	\item Every element of $\DELTA$ is indivisible.
\end{enumerate}
The element of $\DELTA$ is called a \textbf{simple root} of $\PHI$.
It is well known that $(\alpha_i,\alpha_j) \leq 0$ if $i \neq j$.

Let $\beta = \sum_{i=1}^\ell c_i\alpha_i \in \PHI$.
A root $\beta$ is \textbf{positive} (resp. \textbf{negative}) if $c_i \geq 0$ (resp. $c_i \leq 0$) for all $i \in \{1,\ldots,\ell\}$.
Let $\PHI^+$ (resp. $\PHI^-$) denote the set of positive (resp. negative) roots.
Every root is either positive or negative for $\DELTA$, that is, $\PHI = \PHI^+ \sqcup \PHI^-$.
The \textbf{height} of $\beta$ is a sum of all coefficients, denoted by $\height{\beta}$:
\begin{align}
	\height{\beta} = \sum_{i=1}^\ell c_i.
\end{align}
It is well known that there exists a unique root $\tilde{\alpha} \in \PHI$ with the maximum height.
The root $\tilde{\alpha}$ is called the \textbf{highest root} of $\PHI$.
Since $\tilde{\alpha} \in \PHI^+$, it can be expressed as a linear combination
\begin{align}
	\tilde{\alpha} = n_1\alpha_1 + \cdots + n_\ell\alpha_\ell
\end{align}
for $n_1,\ldots,n_\ell \in \mathbb{Z}_{>0}$.
Moreover, for any $\beta = \sum_{i=1}^\ell c_i \alpha_i \in \PHI^+$, we have $0 \leq c_i \leq n_i$ for all $i \in \{1,\ldots,\ell\}$.
Let $\alpha_0 \ceq -\tilde{\alpha} \in \PHI$, and  $n_0 \ceq 1$.
Then we have
\begin{align}
	n_0\alpha_0 + n_1\alpha_1 + \cdots + n_\ell\alpha_\ell = 0. \label{dependentdelta0}
\end{align}
We define $\DELTA_0 \ceq \{\alpha_0,\alpha_1,\ldots,\alpha_\ell\}$.

For a simple root $\alpha_i$, the reflection $s_{\alpha_i}$ is called a \textbf{simple reflection}.
The Weyl group $W$ is generated by the simple reflections.
In other words, an element $w \in W$ can be expressed as the product of simple reflections.
The \textbf{length} of $w \in W$ is the minimum number of simple reflections required to express $w$ as a product of simple reflections.
The \textbf{longest element} of $W$ is the element of $W$ with the maximum length.
The longest element $w_0$ is unique, and satisfies $w_0(\DELTA) = -\DELTA$, and $w_0^2 = 1$.


\subsection{Dynkin diagrams and extended Dynkin diagrams}\ 

\noindent
Let $\PHI$ be an irreducible root system with a basis $\DELTA = \{\alpha_1,\ldots,\alpha_\ell\}$ and the highest root $\tilde{\alpha}$.
For distinct integers $i,j \in \{0,1,\ldots,\ell\}$, let $f_{ij}$ be an integer defined as
\begin{align}
	f_{ij} \ceq \begin{cases*}
		\dfrac{(\alpha_i,\alpha_i)}{(\alpha_j,\alpha_j)}	& if $(\alpha_i,\alpha_j) \neq 0$, $\ (\alpha_i,\alpha_i) \geq (\alpha_j,\alpha_j)$;\\
		\dfrac{(\alpha_j,\alpha_j)}{(\alpha_i,\alpha_i)}	& if $(\alpha_i,\alpha_j) \neq 0$, $\ (\alpha_i,\alpha_i) < (\alpha_j,\alpha_j)$;\\
		0													& if $(\alpha_i,\alpha_j) = 0$.
	\end{cases*}
\end{align}
Construct a graph $\mathcal{D}(\PHI)$ with vertex set $\DELTA$ as follows:
\begin{enumerate}[label=(D\arabic*)]
	\item\label{D1} Draw $f_{ij}$ parallel edges connecting $\alpha_i$ and $\alpha_j$;
	\item\label{D2} Write an inequality symbol oriented towards $\alpha_j$ on the multiple edges between $\alpha_i$ and $\alpha_j$ if $(\alpha_i,\alpha_i) > (\alpha_j,\alpha_j)$.
\end{enumerate}
The graph $\mathcal{D}(\PHI)$ is called a \textbf{Dynkin diagram} of $\PHI$, and $\mathcal{D}(\PHI)$ is independent of the choice of $\DELTA$.
Similarly, let $\mathcal{D}_0(\PHI)$ be the graph with vertex set $\DELTA_0$ constructed in steps \cref{D1} and \cref{D2}.
The graph $\mathcal{D}_0(\PHI)$ is called an \textbf{extended Dynkin diagram} of $\PHI$.
It is clear that $\mathcal{D}(\PHI)$ is a subgraph of $\mathcal{D}_0(\PHI)$ induced by $\DELTA$.

\begin{proposition}[{\cite[\rom{6}, \S4, Theorem 3]{Bourbaki}}]
	Let $\PHI$ be an irreducible root system.
	Then its Dynkin diagram $\mathcal{D}(\PHI)$ is isomorphic to one of the graphs in \cref{figDD}, and its extended Dynkin diagram $\mathcal{D}_0(\PHI)$ is isomorphic to one of the graphs in \cref{figeDD}.
\end{proposition}
\begin{proposition}[{\cite[\rom{6}, \S4, Theorem 3]{Bourbaki}}]
	Let $\mathcal{D}$ be a graph in \cref{figDD}. 
	Then there exists an irreducible reduced root system such that its Dynkin diagram is ismorphic to $\mathcal{D}$, which is unique up to isomorphism.
\end{proposition}

Using Dynkin diagrams, we can classify irreducible reduced root systems up to isomorphism.
A root system of \textbf{type $\bm{BC_\ell}$} is a non-reduced root system whose extended Dynkin diagram is of type $BC_\ell$.
Such a root system is uniquely determined for each rank $\ell$ up to isomorphism.
The Dynkin diagram of the root system of type $BC_\ell$ is isomorphic to the graph of type $A_1$ or type $B_\ell$ ($\ell \geq 2$).
Therefore it is not possible to determine whether a root system with the Dynkin diagram of type $A_1$ or $B_\ell$ is reduced.
The root system of type $BC_\ell$ has three different root lengths.
We will refer to roots with the shortest lenth as \textbf{short roots}, those with the medium-length as \textbf{middle roots}, and those with the longest length as \textbf{long roots}.

%
%
%

\subsection{Stabilizer subgroups of the extended affine Weyl group with respect to the fundamental alcove}\ \label{sec2.3}

\noindent
Suppose that a root system $\PHI$ is irreducible and reduced.
The \textbf{coweight lattice} $Z$ and \textbf{coroot lattice} $\veeQ$ are lattices defined by 
\begin{align}
	Z &\ceq \bigset{x \in E}{(\alpha,x) \in \mathbb{Z} \tforall \alpha \in \PHI},\\
	\veeQ &\ceq \sum_{\alpha \in \PHI}\mathbb{Z}\alpha^\vee.
\end{align}
The coroot lattice $\veeQ$ is a subgroup of the coweight lattice $Z$ with a finite index $f \ceq (Z:\veeQ)$, called an \textbf{index of connection}.
Let $\DELTA^\vee \ceq \{\varpi_1^\vee,\ldots,\varpi_\ell^\vee\} \subseteq E$ be the dual basis of $\DELTA$, that is,
\begin{align}
	(\alpha_i,\varpi_j^\vee) = \delta_{ij}
\end{align}
for all $i,j \in \{1,\ldots,\ell\}$.

For $\alpha \in \PHI$ and $k \in \mathbb{Z}$, define an affine hyperplane $H_{\alpha}^k$ by
\begin{align}
	H_\alpha^k \ceq \bigset{x \in E}{(\alpha,x) = k}.
\end{align}
A connected component of the complement of the hyperplanes $\A^\aff = \tbigset{H_\alpha^k}{\alpha \in \PHI,\ k \in \mathbb{Z}}$ is called an \textbf{alcove}.
In particular, the set
\begin{align}
	A_\circ \ceq \bigset{x \in E}{(\alpha_i,x) > 0 \tforall i \in \{1,\ldots,\ell\},\ (\tilde{\alpha},x) < 1}
\end{align}
is an alcove, called a \textbf{fundamental alcove} of $\PHI$.
The closure $\overline{A_\circ}$ is the convex hull
\begin{align}
	\overline{A_\circ} = \conv\Bigset{\frac{\varpi_i^\vee}{n_i}}{i \in \{0,\dots,\ell\}},
\end{align}
where $\varpi_0^\vee \ceq 0$.

For a vector $v \in E$, let $t_v$ denote the translation map by $v$.
Let $s_{\alpha,k}$ denote the reflection with respect to $H_{\alpha}^k$.
Then we have
\begin{align}
	s_{\alpha,k} = t_{k\alpha^\vee} \circ s_{\alpha,0}.
\end{align}
for all $\alpha \in \PHI$ and $k \in \mathbb{Z}$.
The \textbf{affine Weyl group} of $\PHI$ is a group generated by the reflections $\tbigset{s_{\alpha,k}}{\alpha \in \PHI,\ k \in \mathbb{Z}}$, denoted by $W_\aff$.
The affine Weyl group is the semidirect product of the Weyl group $W$ by the coroot lattice $\veeQ$:
\begin{align}
	W_\aff = \veeQ \rtimes W.
\end{align}
It is well known that $W_\aff$ acts simply transitively on the set of alcoves and $\overline{A_\circ}$ is a fundamental domain for $W_\aff$ acting on $E$.

The semidirect group $\widehat{W_\aff} \ceq Z \rtimes W$ is called the \textbf{extended affine Weyl group} of $\PHI$.
The affine Weyl group $W_\aff$ is a subgroup of $\widehat{W_\aff}$, but the group $\widehat{W_\aff}$ is not a reflection group.
The action of $\widehat{W_\aff}$ on the set of alcoves is transitive but not simply transitive.
Moreover, $\overline{A_\circ}$ is not a fundamental domain for $\widehat{W_\aff}$ acting on $E$.

We consider the stabilizer subgroup 
\begin{align}
	\widehat{\OMEGA} \ceq \bigset{\widehat{\omega} \in \widehat{W_\aff}}{\widehat{\omega}(A_\circ) = A_\circ}.
\end{align}
We have a decomposition $\widehat{W_\aff} \cong W_\aff \rtimes \widehat{\OMEGA}$ and isomorphism
\begin{align}
	\widehat{\OMEGA} \cong \widehat{W_\aff}/W_\aff \cong Z/\veeQ.
\end{align}
Hence $\widehat{\OMEGA}$ is a finite group since $\#\widehat{\OMEGA} = f = (Z:\veeQ) < \infty$ (see \cref{table1}).

\begin{table}[]
	
	\caption{List of $\widehat{\OMEGA}$ and $f$}
	
	\begin{tabular}{ll|c|c}
		\multicolumn{2}{c|}{$\PHI$} & \multicolumn{1}{c|}{$\widehat{\OMEGA}$} & \multicolumn{1}{c}{$f$} \\  \hline\hline
		$A_\ell$      & ($\ell \geq 1$)            & $\mathbb{Z}/(\ell+1)\mathbb{Z}$                        & $\ell+1$                                     \\ 
		$B_\ell$      & ($\ell \geq 2$)            & $\mathbb{Z}/2\mathbb{Z}$                               & 2                                            \\ 
		$C_\ell$      & ($\ell \geq 3$)            & $\mathbb{Z}/2\mathbb{Z}$                               & 2                                            \\ 
		$D_\ell$      & ($\ell \geq 4$, even)      & $\mathbb{Z}/2\mathbb{Z} \times \mathbb{Z}/2\mathbb{Z}$ & 4                                            \\ 
		$D_\ell$      & ($\ell \geq 5$, odd)       & $\mathbb{Z}/4\mathbb{Z}$                               & 4                                            \\ 
		$E_6$         &                            & $\mathbb{Z}/3\mathbb{Z}$                               & 3                                            \\ 
		$E_7$         &                            & $\mathbb{Z}/2\mathbb{Z}$                               & 2                                            \\ 
		$E_8$         &                            & 1                                                      & 1                                            \\ %
		$F_4$         &                            & 1                                                      & 1                                            \\ 
		$G_2$         &                            & 1                                                      & 1                                            \\ 
	\end{tabular}
	
	\ 
	
	\label{table1}
	
\end{table}

Let $I \ceq \{0,\ldots,\ell\}$ and set
\begin{align}
	J \ceq \bigset{i \in I}{n_i = 1}.
\end{align}
If $j \in J$, then $\varpi_j^\vee$ is called a \textbf{minuscule coweight}.
It is known that minuscule coweights form representatives of the classes of $Z/\veeQ$ (cf. \cite[\rom{6}, Exercise 24]{Bourbaki}).

For $j \in J$, let $w_j$ be the longest element of the Weyl group corresponding to $\DELTA_0 \setminus \{\alpha_0,\alpha_j\}$.
Define $\omega_j \in W$ and $\widehat{\omega_j} \in \widehat{W_\aff}$ by
\begin{align}
	\omega_j \ceq w_jw_0,\qquad
	\widehat{\omega_j} \ceq t_{\varpi_j^\vee} \circ \omega_j.
\end{align}
Then we have the following (\cite[\rom{6}, \S2, Proposition 6]{Bourbaki}, \cite[Proposition-Definition 2.2.1]{Garnier}):
\begin{align}
	\widehat{\OMEGA} = \bigset{\widehat{\omega}_j \in \widehat{W_\aff}}{j \in J}.
	\label{hatomega}
\end{align}

\begin{proposition}[{\cite[Lemma 2.2.4]{Garnier}}]\ 
	The group $\widehat{\OMEGA}$ acts on the vertex set $\tbigset{\frac{\varpivee_i}{n_i}}{i \in I}$ of $\overline{A_\circ}$.
	In particular, $\widehat{\OMEGA}$ acts on the minuscule coweights $\tbigset{\varpivee_j}{j \in J}$.
\end{proposition}

For $j \in J$, let $\sigma_j$ and $\widehat{\sigma_j}$ be permutations on $I$ defined by 
\begin{align}
	\omega_j(\alpha_i) = \alpha_{\sigma_j(i)},\qquad
	\widehat{\omega_j}\left( \frac{\varpivee_i}{n_i} \right) = \frac{\varpivee_{\widehat{\sigma_j}(i)}}{n_{\widehat{\sigma_j}(i)}}.
\end{align}
for all $i \in I$.

\begin{proposition}[{\cite[Lemma 2.2.5]{Garnier}, \cite[Lemma 1]{KomrakovPremet}}] \label{prop2.6}
	Let $j \in J$.
	Then $\sigma_j$ and $\widehat{\sigma_j}$ satisfy the following:
	\begin{eenumerate}
		\item $\sigma_j = \widehat{\sigma_j}$;
		\item\label{2.6.2} $\sigma_j(0) = j$;
		\item\label{2.6.3} $n_i = n_{\sigma_j(i)}$ for all $i \in I$.
	\end{eenumerate}
\end{proposition}

Define a set
\begin{align}
	\OMEGA \ceq \bigset{\omega_j \in W}{j \in J}.
\end{align}
Let $\pi : \widehat{W_\aff} \lra W$ be the projection.
Then $\pi(\widehat{\OMEGA}) = \OMEGA$, and hence $\OMEGA$ is a subgroup of $W$ isomorphic to $\widehat{\OMEGA}$.

%
%

\begin{proposition}[{\cite[Lemma 2.2.5, Corollary 2.2.7]{Garnier}}]
	The group $\OMEGA$ acts on $\DELTA_0$.
	In particular, $\OMEGA$ can be regarded as a subgroup of automorphisms of the extended Dynkin diagram.
\end{proposition}

The permutation $\sigma_j$ is as shown in \cref{sigmaj} (cf. \cite[Table 2]{Garnier}, \cite[Table]{KomrakovPremet})

\begin{table}[]
	
	\caption{List of $J$ and $\sigma_j$}
	\label{sigmaj}
	
	\begin{tabular}{ll|c|c}
		\multicolumn{2}{c|}{$\PHI$} & \multicolumn{1}{c|}{$J$} & \multicolumn{1}{c}{$\sigma_j$} \\ \hline\hline
		$A_\ell$      & ($\ell \geq 1$)            & $J = \{0,\ldots,\ell\}$        	& $\sigma_j = (0\ 1\ \cdots \ \ell)^j \tforall j$  \\ 
		$B_\ell$      & ($\ell \geq 2$)            & $J = \{0.1\}$        				& $\sigma_1 = (0\ 1)$  \\ 
		$C_\ell$      & ($\ell \geq 3$)            & $J = \{0,\ell\}$    				& $\sigma_\ell = (0\ \ell)\displaystyle\prod_{i=1}^{\lfloor \frac{\ell-1}{2} \rfloor}(i\ \ \ell-i)$ \\ 
		$D_\ell$      & ($\ell \geq 4$, even)      & $J = \{0,1,\ell-1,\ell\}$ 			& $\left\{\begin{array}{lll}
			\sigma_1 = (0\ 1)(\ell-1\ \ \ell)\\
			\sigma_{\ell-1} = (0\ \ \ell-1)(1\ \ell)\displaystyle\prod_{i=2}^{\frac{\ell}{2}-1}(i\ \ \ell-i)\\
			\sigma_{\ell} = \sigma_1\sigma_{\ell-1} = (0\ \ell)(1\ \ \ell-1)\displaystyle\prod_{i=2}^{\frac{\ell}{2}-1}(i\ \ \ell-i)
		\end{array}\right.$ \\ 
		$D_\ell$      & ($\ell \geq 5$, odd)       & $J = \{0,1,\ell-1,\ell\}$   		& $\left\{\begin{array}{lll}
			\sigma_1 = \sigma_{\ell-1}^2 =\sigma_{\ell}^2 = (0\ 1)(\ell-1\ \ \ell)\\
			\sigma_{\ell-1} = (0\ \ \ell-1\ \ 1\ \  \ell)\displaystyle\prod_{i=2}^{\frac{\ell-1}{2}}(i\ \ \ell-i)\\
			\sigma_{\ell} = (0\ \ \ell\ \ 1\ \ \ell-1)\displaystyle\prod_{i=2}^{\frac{\ell-1}{2}}(i\ \ \ell-i)
		\end{array}\right.$ \\ 
		$E_6$         &                            & $J = \{0,1,6\}$  					& $\left\{\begin{array}{lll}
			\sigma_1 = (0\ 1\ 6)(2\ 3\ 5)\\
			\sigma_6 = \sigma_1^{-1} = (1\ 0\ 6)(2\ 5\ 3)
		\end{array}\right.$ \\ 
		$E_7$         &                            & $J = \{0,7\}$   					& $\sigma_7 = (0\ 7)(1\ 6)(3\ 5)$  \\ 
		$E_8$         &                            & $J = \{0\}$   						& ---                                            \\ 
		$F_4$         &                            & $J = \{0\}$						& ---                                            \\ 
		$G_2$         &                            & $J = \{0\}$						& ---                                            \\ 
	\end{tabular}
		
\

\end{table}

\section{Root systems constructed by folding}\label{sec3}

\subsection{Subspaces of fixed points}\ \label{sec3.1}

\noindent
Let $\PHI$ be an irreducible reduced root system.
We can assume that $\PHI$ is not of type $E_8$, $F_4$ or $G_2$, since $\OMEGA$ is trivial.

Let $j \in J \setminus \{0\}$, $\omega \ceq \omega_j \in \OMEGA$ and $\sigma \ceq \sigma_j$.
Let $o(\omega)$ denote the order of $\omega$. 
Let $S_0^j,\ldots,S_r^j$ denote all $\langle \sigma \rangle$-orbits in $I$, that is, $I = S_0^j \sqcup \cdots \sqcup S_r^j$m, where we assume that $0 \in S_0^j$.
Then \cref{prop2.6} \cref{2.6.2} implies that $S_0^j \subseteq J$, $j \in S_0^j$ and $\#S_0^j = o(\omega)$.

For $k \in \{0,\ldots,r\}$, let $\bar{n}_k$ be defined as $n_s$, where $s \in S_k^j$.
By \cref{prop2.6} \cref{2.6.3}, the definition of $\bar{n}_k$ is well-defined.
Moreover, we define $m_k$ by
\begin{align}
	m_k \ceq \frac{\#S_k^j \bar{n}_k}{\#S_0^j}.
\end{align}
Note that $m_0 = 1$, and we will see later that $m_k$ is a positive integer.

Let $\DELTA^\vee = \{\varpivee_1,\ldots,\varpivee_\ell\}$ be the dual basis of the simple roots $\DELTA$, and let $\varpivee_0 = 0$.
Then $\DELTA^\vee$ is a basis for $E$.
By definition of $\widehat{\omega} = t_{\varpivee_j}\omega$ and $\sigma$, we have
\begin{align}
	\omega(\varpivee_i) = \varpivee_{\sigma(i)} - n_i\varpivee_{j},\qquad
	\omega^t(\varpivee_i) = \varpivee_{\sigma^t(i)} - n_i\varpivee_{\sigma^t(0)}
\end{align}
for all $i \in I$ and $t \in \mathbb{Z}$.
For $k \in \{0,\ldots,r\}$, define $\pi_k^j \in E$ by
\begin{align}
	\pi_k^j \ceq \sum_{s \in S_k^j}\varpivee_s.
\end{align}
Then we have
\begin{align}
	\omega(\pi_k^j) = \pi_k^j - \#S_k^j \bar{n}_k \varpivee_j,\qquad
	\omega^t(\pi_k^j) = \pi_k^j - \#S_k^j \bar{n}_k \varpivee_{\sigma^t(0)}
\end{align}
for all $t \in \mathbb{Z}$.

Let $E^\omega$ denote the subspace of fixed points by $\omega$:
\begin{align}
	E^\omega \ceq \bigset{x \in E}{\omega(x) = x}.
\end{align}

\begin{definition}
For any $x \in E$, we define an element $x^\omega \in E$ by 
\begin{align}
	x^\omega \ceq \frac{1}{o(\omega)}\sum_{t=1}^{o(\omega)}\omega^t(x).
\end{align}
\end{definition}
It is clear that $x^\omega \in E^\omega$.
Furthermore, let $\phi_\omega : E \lra E^\omega$ be a map defined by $\phi_\omega(x) = x^\omega$.
Then this operator $\phi_\omega$ is linear and surjective.

For $k \in \{1,\ldots,r\}$, define $\bar{\pi}_k^\omega \in E^\omega$ by 
\begin{align}
	\bar{\pi}_k^\omega \ceq \phi_\omega(\pi_k^j) = \pi_k^j - m_k\pi_0^j.
\end{align}

\begin{lemma}\label{lem3.1}
	Let $x \in E$, and define $a_i \ceq (\alpha_i,x)$ for $i \in \{0,\ldots,\ell\}$.
	Then the following are equivalent:
	\begin{eenumeratei}
		\item\label{3.1i} $x \in E^\omega$;
		\item\label{3.1ii} For any $k \in \{0,\ldots,r\}$, $s_1,s_2\in S_k^j$ implies that $a_{s_1} = a_{s_2}$.
	\end{eenumeratei}
\end{lemma}
\begin{remark}
	Under the assumption of \cref{lem3.1}, we have
	\begin{align}
		x = \sum_{i=0}^\ell a_i\varpivee_i.
	\end{align}
	Furthermore, the equation \cref{dependentdelta0} implies that
	\begin{align}
		\sum_{i=0}^\ell n_ia_i = 0. \label{kei=0}
	\end{align}
\end{remark}
\begin{proof}[Proof of \cref{lem3.1}]
	First, suppose that $x \in E^\omega$.
	Then
	\begin{align}
		x = \omega(x) 
		&= \sum_{i=0}^\ell a_i(\varpivee_{\sigma(i)} - n_i\varpivee_j) \\
		&= \sum_{i=0}^\ell a_i\varpivee_{\sigma(i)} - \sum_{i=0}^\ell n_ia_i\varpivee_j \\
		&= \sum_{i=0}^\ell a_{\sigma^{-1}(i)}\varpivee_i.
	\end{align}
	Hence $a_{\sigma^{-1}(i)} = a_i$ for any $i \in I$, and \cref{3.1ii} holds.
	
	Second, suppose that the condition \cref{3.1ii} holds.
	Then 
	\begin{align}
		x = \sum_{i=0}^\ell a_i\varpivee_i = \sum_{k=0}^r \bar{a}_k \pi_k^j,
	\end{align}
	where $\bar{a}_k$ is defined as $a_s$ for $s \in S_k^j$.
	Hence we have
	\begin{align}
		\omega(x) = \sum_{k = 0}^r \bar{a}_k\omega(\pi_k^j)
		&= \sum_{k=0}^r\bar{a}_k(\pi_k^j - \#S_k^j \bar{n}_k \varpivee_j)\\
		&= \sum_{k=0}^r\bar{a}_k\pi_k^j - \sum_{k=0}^ra_k\#S_k^j \bar{n}_k \varpivee_j\\
		&= \sum_{k=0}^r\bar{a}_k\pi_k^j - \sum_{i=0}^\ell n_i a_i \varpivee_j\\
		&= \sum_{k=0}^r\bar{a}_k\pi_k^j\\
		&= x. \qedend
	\end{align}
\end{proof}

\begin{theorem}
	$\{\bar{\pi}_1^\omega,\ldots,\bar{\pi}_r^\omega\}$ is a basis for $E^\omega$.
\end{theorem}
\begin{proof}
	It is clear that $\{\bar{\pi}_1^\omega,\ldots,\bar{\pi}_r^\omega\}$ is linear independent.
	Let $x \in E^\omega$ and $a_i \ceq (\alpha_i,x)$ for $i \in \{0,\ldots,\ell\}$.
	The condition \cref{3.1ii} in \cref{lem3.1} and the equation \cref{kei=0} imply that 
	\begin{align}
		\sum_{k=0}^r m_k\bar{a}_k = \frac{1}{\#S_0^j}\sum_{i=0}^\ell n_ia_i = 0,
	\end{align}
	and hence we have
	\begin{align}
		\bar{a}_0 = -\sum_{k=1}^r m_k\bar{a}_k.
	\end{align}
	Therefore 
	\begin{align}
		x &= \sum_{k=0}^r \bar{a}_k\pi_k^j \\
		&= \sum_{k=1}^r \bar{a}_k\pi_k^j + \bar{a}_0\pi_0^j \\
		&= \sum_{k=1}^r \bar{a}_k\pi_k^j -\sum_{k=1}^r m_k\bar{a}_k\pi_k^j \\
		&= \sum_{k=1}^r \bar{a}_k \bar{\pi}_k^\omega.
	\end{align}
	Thus $\{\bar{\pi}_1^\omega,\ldots,\bar{\pi}_r^\omega\}$ is a spanning set of $E^\omega$.
\end{proof}

\subsection{Actions on the root systems}

We consider the $\omega$-action on the root syetem $\PHI$.
To be convenient for the action of $\omega$, we agree on the following:
When we express $\beta \in \PHI$ as a linear combination
\begin{align}
	\beta = \sum_{i=1}^\ell c_i \alpha_i,
\end{align}
we agree to let $c_0 = 0$ and consider
\begin{align}
	\beta = \sum_{i=0}^\ell c_i\alpha_i.
\end{align}
Then we have
\begin{align}
	\omega^t(\beta) = \sum_{i=0}^\ell c_i\alpha_{\sigma^t(i)} = \sum_{i=0}^\ell c_{\sigma^{-t}(i)}\alpha_{i}
\end{align}
for all $t \in \mathbb{Z}$.

\begin{lemma}\label{lem6.1}
	Let $\beta = \sum_{i=0}^\ell c_i \alpha_i \in \PHI^+$ and $t \in \mathbb{Z}$.
	Then the following are equivalent:
	\begin{eenumeratei}
		\item $\omega^t(\beta) \in \PHI^+$;
		\item $c_{\sigma^{-t}(0)} = 0$,
	\end{eenumeratei}
\end{lemma}
\begin{proof}
	
	It is clear that $\omega^t(\beta) \in \PHI^+$ if $c_{\sigma^{-t}(0)} = 0$.
	
	Suppose that $c_{\sigma^{-t}(0)} \neq 0$.
	We have $c_{\sigma^{-t}(0)} = 1$ since $0 \leq c_i \leq n_i$ for all $i \in \{1,\ldots,\ell\}$.
	Then
	\begin{align}
		\omega^t(\beta) 
		&= \sum_{i \neq \sigma^{-t}(0)}c_i\alpha_{\sigma^t(i)} + c_{\sigma^{-t}(0)}\alpha_0\\
		&= \sum_{i \neq 0} c_{\sigma^{-t}(i)}\alpha_{i} - \sum_{i=1}^\ell n_i\alpha_i\\
		&= \sum_{i=1}^\ell (c_{\sigma^{-t}(i)} - n_i)\alpha_i = \sum_{i=0}^\ell (c_{\sigma^{-t}(i)} - n_i)\alpha_i.
	\end{align}
	\cref{prop2.6} \cref{2.6.3} implies that $c_{\sigma^{-t}(i)} - n_i = c_{\sigma^{-t}(i)} - n_{\sigma^{-t}(i)} \leq 0$ for all $i \in \{1,\ldots,\ell\}$.
	Hence $\omega^t(\beta) \in \PHI^-$.
\end{proof}

\begin{lemma}\label{lem6.2}
	Let $\beta \in \PHI$ and $t \in \mathbb{Z}$.
	If $\omega^t(\beta ) \neq \beta$, then $(\beta,\, \omega^t(\beta)) \leq 0$.
\end{lemma}
\begin{proof}
	We can assume that $\beta = \sum_{i=0}^\ell c_i\alpha_i \in \PHI^+$ without loss of generality.
	Assume that $(\beta,\, \omega^t(\beta)) > 0$.
	Then $\omega^t(\beta) - \beta \in \PHI$ by \cref{prop2.2}.

	First, suppose that $c_{\sigma^{-t}(0)} = 1$.
	\cref{lem6.1} implies that
	\begin{align}
		\omega^t(\beta) - \beta = \sum_{i=0}^\ell (c_{\sigma^{-t}(i)} - c_i - n_i) \alpha_i \in \PHI^-. \label{rhsd}
	\end{align}
	The coefficients on the right-hand side of \cref{rhsd} are all non-positive since $c_{\sigma^{-t}(0)} - c_0 - n_0 = 0$.
	For any $k \in \mathbb{Z}_{>1}$, since $\sigma^{-kt}(0) \in S_0^j \subseteq J$, we have 
	\begin{align}
		-1 \leq c_{\sigma^{-kt}(i)} - c_{\sigma^{-(k-1)t}(i)} - n_{\sigma^{-(k-1)t}(i)} = c_{\sigma^{-kt}(i)} - c_{\sigma^{-(k-1)t}(i)} -1.
	\end{align}
	By using induction, we can see that $c_{\sigma^{-kt}(i)} = 1$ for all $k \in \mathbb{Z}_{>0}$.
	Let $k_0 \ceq o(\omega)$, then $\sigma^{-k_0t}(0) = 0$.
	Hence $c_0 = c_{\sigma^{-k_0t}(0)} = 1$, but this is a contradiction to $c_0 = 0$.
	
	Second, suppose that $c_{\sigma^{-t}(0)} = 0$.
	Then
	\begin{align}
		\omega^t(\beta) - \beta = \sum_{i=0}^\ell (c_{\sigma^{-t}(i)} - c_i) \alpha_i.
	\end{align}
	Since $\omega^t(\beta) \neq \beta$, there exists $i \in \{1,\ldots,\ell\}$ such that $c_{\sigma^{-t}(i)} \neq c_i$.
	Assume that $c_{\sigma^{-t}(i)} < c_i$.
	Then $\omega^t(\beta) - \beta \in \PHI^-$, and hence $c_{\sigma^{-(k+1)t}(i)} \leq c_{\sigma^{-kt}(i)}$ for all $k \in \mathbb{Z}_{>0}$.
	Let $k_0 \ceq o(\omega)$, then $\sigma^{-k_0t}(i) = i$.
	Hence we have
	\begin{align}
		c_i = c_{\sigma^{-k_0t}(i)} \leq c_{\sigma^{-(k_0-1)t}(i)} \leq \cdots \leq c_{\sigma^{-t}(i)} < c_i,
	\end{align}
	but this is a contradiction.
	Similarly, assuming $c_{\sigma^{-t}(i)} > c_i$ also yields a contradiction.
	
	Therefore we have $(\beta,\, \omega^t(\beta)) \leq 0$.
\end{proof}

\subsection{Proof of being root systems}\ \label{sec3.3}

\noindent
For $\beta \in \PHI$, let $O(\beta)$ and $N(\beta)$ be subsets of $\PHI$ defiend by  
\begin{align}
	O(\beta) \ceq \bigset{\omega^t(\beta) \in \PHI}{t \in \mathbb{Z}},\qquad 
	N(\beta) \ceq \bigset{\gamma \in O(\beta)}{\gamma \neq \pm\beta,\ (\beta,\gamma) \neq 0}.
\end{align}
We can see that $O(\omega(\beta)) = O(\beta)$ and $\#N(\omega(\beta)) = \#N(\beta)$ since $(\cdot,\cdot)$ is invariant under $W$.

We consider sets
\begin{align}
	\PHI^{\omega} &\ceq \phi_\omega(\PHI) = \bigset{\beta^{\omega} \in E^\omega}{\beta \in \PHI},&
	\PHI^{\omega}_\re &\ceq \bigset{\beta^{\omega} \in \PHI^{\omega}}{\beta^{\omega} \neq 0},\\
	\DELTA^{\omega}_0 &\ceq \phi_\omega(\DELTA_0) = \bigset{\alpha_i^{\omega} \in \PHI^{\omega}}{\alpha_i \in \DELTA_0},&
	\DELTA^{\omega} &\ceq \bigset{\alpha_i^{\omega} \in \PHI^{\omega}}{\alpha_i \in \DELTA_0\setminus O(\alpha_0)}.
\end{align}
For $k \in \{0,\ldots,r\}$, define
\begin{align}
	\bar{\alpha}_k^\omega \ceq \frac{1}{\#S_k^j}\sum_{s \in S_k^j}\alpha_s.
\end{align}
Then $\bar{\alpha}_k^\omega = \alpha_{s_k}^\omega$ for some $s_k \in S_k^j$, and 
\begin{align}
	\DELTA^\omega_0 = \{\bar{\alpha}_0^\omega,\ldots,\bar{\alpha}_r^\omega\},\qquad
	\DELTA^\omega = \{\bar{\alpha}_1^\omega,\ldots,\bar{\alpha}_r^\omega\}.
\end{align}

The main result in this paper is the following.

\begin{theorem}\label{mainthm}
	Let $\PHI$ is an irreducible reduced root system and $\omega \in \OMEGA\setminus \{1\}$. 
	Then $\PHI^{\omega}_\re$ is a (not necessarily reduced) root system in $E^\omega$.
\end{theorem}

We will discuss later that $\DELTA^\omega$ is a basis of $\PHI$.

\begin{remark}
	If $\PHI$ is of type $A_\ell$ and $\omega \in \OMEGA$ satisfies $o(\omega) = \ell + 1$, 
	then $S_0^j = I$, that is,
	\begin{align}
		\alpha_i^\omega = \frac{1}{o(\omega)}\sum_{i=0}^\ell \alpha_i = 0
	\end{align}
	for all $\alpha_i \in \DELTA_0$.
	Therefore we can see that $\PHI^\omega = \{0\}$ and $\PHI_\re^\omega = \emptyset$ (see also \cref{thm3.11}).
\end{remark}

We assume that $o(\omega) < \ell + 1$ if $\PHI$ is of type $A_\ell$.
First, we will show that $\DELTA_0^\omega \subseteq \PHI_\re^\omega$.

\begin{lemma} \label{lem3.9}
	Let $\beta_1,\beta_2 \in \PHI$.
	Then
	\begin{align}
		(\beta_1^\omega,\, \beta_2^\omega) = \frac{1}{\#O(\beta_2)}\sum_{\gamma_2 \in O(\beta_2)}(\beta_1,\gamma_2).
	\end{align}
\end{lemma}
\begin{proof}
	Recall that $(\cdot,\cdot)$ is invariant under $W$.
	We can see that the following:
	\begin{align}
		(\beta_1^{\omega},\, \beta_2^{\omega})
		&= \frac{1}{o(\omega)^2}\left( \sum_{t_1=1}^{o(\omega)}\omega^{t_1}(\beta_1),\ \sum_{t_2=1}^{o(\omega)}\omega^{t_2}(\beta_2) \right)\\
		&= \frac{1}{o(\omega)^2}\sum_{t_1=1}^{o(\omega)}\left( \beta_1,\ \sum_{t_2=1}^{o(\omega)}\omega^{t_2-t_1}(\beta_2) \right)\\
		&= \frac{1}{o(\omega)}\sum_{t=1}^{o(\omega)}\left( \beta_1,\ \omega^{t}(\beta_2) \right)\\
		&= \frac{1}{\#O(\beta_2)}\sum_{\gamma_2 \in O(\beta_2)}(\beta_1,\gamma_2). \qedend
	\end{align}
\end{proof}

\begin{lemma} \label{lem3.10}
	Let $\beta \in \PHI$.
	Then $\#N(\beta) \leq 2$.
	Furthermore, the following are equivalent:
	\begin{eenumeratei}
		\item $\beta^\omega \neq 0$;
		\item $-\beta \not\in O(\beta)$ and $\#N(\beta) < 2$.
	\end{eenumeratei}
\end{lemma}
	\begin{proof}
		Suppose that $\beta \in \PHI$ satisfies $-\beta \in O(\beta)$.
		Then there exists $t_0 \in \mathbb{Z}_{>0}$ such that $\omega^{t_0}(\beta) = -\beta$.
		If $\omega^t(\beta) \neq \pm\beta$, then \cref{lem6.2} implies that 
		\begin{align}
			0 \leq -(\beta,\, \omega^{t}(\beta)) = (\beta,\, \omega^{t_0 + t}(\beta)) \leq 0.
		\end{align}
		Hence we can see that $N(\beta) = \emptyset$, that is, $\#N(\beta) = 0$.
		Therefore, we have $O(\beta) = \{\beta, -\beta\}$, and it follows from \cref{lem3.9} that 
		\begin{align}
			(\beta^{\omega},\, \beta^{\omega}) = \frac{1}{\#O(\beta)}\sum_{\gamma \in O(\beta)}(\beta,\gamma) = \frac{(\beta,\beta) + (\beta,-\beta)}{\#O(\beta)} = 0.
		\end{align}
		Hence $\beta^\omega = 0$.
		
		Suppose that $\beta \in \PHI$ satisfies $-\beta \not\in O(\beta)$.
		\cref{lem6.2} and \cref{prop2.1} imply that 
		\begin{align}
			(\beta,\gamma) = -\frac{(\beta,\beta)}{2}
		\end{align}
		for all $\gamma \in N(\beta)$.
		Therefore
		\begin{align}
			(\beta^{\omega},\, \beta^{\omega}) = \frac{1}{\#O(\beta)}\sum_{\gamma \in O(\beta)}(\beta,\gamma) = \frac{1}{\#O(\beta)}\left((\beta,\beta) - \#N(\beta) \cdot \frac{(\beta,\beta)}{2}\right) = \frac{(2 - \#N(\beta)) \cdot (\beta,\beta)}{2 \cdot \#O(\beta)}.
			\label{aa}
		\end{align}
		Thus, $\#N(\beta) \leq 2$ since $(\beta^{\omega},\, \beta^{\omega}) \geq 0$.
		Furthermore, $\#N(\beta) = 2$ if and only if $\beta^{\omega} = 0$ under the assumption $-\beta \not\in O(\beta)$.
	\end{proof}

\begin{theorem} \label{thm3.11}
	Suppose that $o(\omega) < \ell + 1$ if $\PHI$ is of type $A_\ell$.
	Then $\bar{\alpha}_k^\omega \neq 0$ for all $k \in \{0,\ldots,r\}$, that is, $\DELTA_0^\omega \subseteq \PHI^\omega_\re$.
\end{theorem}
\begin{proof}
	Let $\alpha_{i_0} \in \DELTA_0$.
	It is clear that $-\alpha_{i_0} \not\in O(\alpha_{i_0})$ since $\OMEGA$ acts on $\DELTA_0$.
	
	Assume that $N(\alpha_{i_0}) = \{\alpha_{i_{-1}},\, \alpha_{i_1}\}$ for some $i_{-1},i_1 \in I$.
	Since $(\cdot,\cdot)$ is invariant under $W$, there exists $i_2 \in I$ such that $N(\alpha_{i_1}) = \{\alpha_{i_0},\alpha_{i_2}\}$.
	By the finiteness of $I$, we can take a sequence $(i_0,i_1,\ldots,i_h = i_{-1})$ of $I$ such that $N(\alpha_{i_p}) = \{\alpha_{i_{p-1}},\alpha_{i_{p+1}}\}$ for all $p \in \{0,\ldots,h\}$.
	Hence $\{\alpha_{i_0},\ldots,\alpha_{i_h}\} \subseteq O(\alpha_{i_0})$ forms a cycle in the extended Dynkin diagram.
	Thus $\PHI$ must be of type $A_\ell$ and $o(\omega) = \ell+1$.
	In other words, under the assumptions of this theorem, this is a contradiction.
	
	Hence $\#N(\alpha_{i_0}) < 2$, and \cref{lem3.10} implies that $\alpha_{i_0}^\omega \neq 0$.
\end{proof}

We will varify that the condition \cref{R1} holds.
Clearly, $\PHI_\re^\omega$ is a finite set not containing $0$.

\begin{lemma}
	The set $\DELTA^\omega$ is the dual basis of $\{\bar{\pi}_1^\omega,\ldots,\bar{\pi}_r^\omega\}$.
	Hence $\DELTA^\omega$ and $\PHI_\re^\omega$ generates $E^\omega$.
\end{lemma}
\begin{proof}
	Let $k_1\in \{1,\ldots,r\}$.
	For any $k_2 \in \{0,\ldots,r\}$,
	\begin{align}
		(\bar{\alpha}_{k_1}^\omega,\, \pi_{k_2}^j) = \frac{1}{\#S_{k_1}^j}\sum_{s_1 \in S_{k_1}^j}\sum_{s_2 \in S_{k_2}^j}(\alpha_{s_1},\, \varpivee_{s_2}) = \delta_{k_1,k_2}.
	\end{align}
	Hence, for any $k_3 \in \{1,\ldots,r\}$, we have
	\begin{align}
		(\bar{\alpha}_{k_1}^\omega,\, \bar{\pi}_{k_3}^\omega)
		&= \left(\bar\alpha_{k_1}^{\omega},\ \pi_{k_3}^j - m_{k_3}\pi_0^j\right) = (\bar\alpha_{k_1}^{\omega},\, \pi_{k_3}^j) - m_{k_3}(\bar\alpha_{k_1}^\omega,\, \pi_0^j) = \delta_{k_1,k_3}. \qedend
	\end{align}
\end{proof}

To show that the condition \cref{R2} holds, we descrive the reflection of $\beta^\omega$ as a product of reflections of $W$.

\begin{lemma}\label{vees}
	Let $\beta \in \PHI$ satisfy $\beta^\omega \in \PHI_\re^\omega$.
	Then
	\begin{align}
		(\beta^\omega)^\vee = \frac{2}{2 - \#N(\beta)} \sum_{\gamma \in O(\beta)}\gamma^\vee.
	\end{align}
\end{lemma}
\begin{proof}
	By \cref{aa}, we have
	\begin{align}
		(\beta^{\omega})^\vee = \frac{2\beta^{\omega}}{(\beta^{\omega},\, \beta^{\omega})}
		&= \frac{2 \cdot \#O(\beta)}{(2 - \#N(\beta)) \cdot (\beta,\beta)} \cdot \frac{2}{\#O(\beta)}\sum_{\gamma \in O(\beta)}\gamma\\
		&= \frac{2}{2 - \#N(\beta)} \sum_{\gamma \in O(\beta)}\frac{2\gamma}{(\gamma,\gamma)}\\
		&= \frac{2}{2 - \#N(\beta)}\sum_{\gamma \in O(\beta)}\gamma^\vee.  \qedend
	\end{align}
\end{proof}

Suppose that $\beta \in \PHI$ satisfies $\#N(\beta) = 1$.
Then, for any $\gamma \in O(\beta)$, there exists a unique root $\hat\gamma \in O(\beta)$ such that $N(\gamma) = \hat\gamma$, that is,
\begin{align}
	(\gamma,\gamma') = \begin{cases*}
		-\frac{(\gamma,\gamma)}{2} & if $\gamma' = \hat\gamma$;\\
		(\gamma,\gamma) & if $\gamma' = \gamma$;\\
		0 & otherwise.
	\end{cases*} \label{ggp}
\end{align}
Define a set
\begin{align}
	\hat{O}(\beta) \ceq \Bigset{\{\gamma,\hat\gamma\} \in \binom{O(\beta)}{2}}{\gamma \in O(\beta)},
\end{align}
where $\binom{O(\beta)}{2}$ is the set of subset $S$ of $O(\beta)$ satisfying $\#S= 2$.
For any distinct $\{\gamma_1,\hat{\gamma_1}\},\ \{\gamma_2,\hat{\gamma_2}\} \in \hat{O}(\beta)$, the uniqueness of $\hat{\gamma}$ implies that $\gamma_1,\hat{\gamma_1} \not\in \{\gamma_2,\hat\gamma_2\}$.
Therefore
\begin{align}
	O(\beta) = \bigsqcup_{\{\gamma,\hat\gamma\} \in \hat{O}(\beta)}\{\gamma,\hat\gamma\}.
\end{align}

Let $\beta \in \PHI$ satisfy $\beta^\omega \in \PHI_\re^\omega$.
Define $w_\beta \in W$ by
\begin{align}
	w_\beta \ceq \begin{cases*}
		\displaystyle\prod_{\gamma \in O(\beta)}s_{\gamma} & if $\#N(\beta) = 0$;\\
		\displaystyle\prod_{\{\gamma,\hat\gamma\} \in \hat{O}(\beta)}s_\gamma s_{\hat\gamma}s_\gamma & if $\#N(\beta) = 1$.
	\end{cases*} \label{sss}
\end{align}
Note that the reflections $s_\gamma$ and $s_{\gamma'}$ commute under the composition if $(\gamma,\gamma') = 0$.
Moreover, we can see that $s_\gamma s_{\hat\gamma}s_\gamma = s_{\hat\gamma}s_\gamma s_{\hat\gamma}$ for all $\{\gamma,\hat\gamma\} \in \hat{O}(\beta)$ since $\gamma$ and $\hat\gamma$ play the same role (see \cref{sssx} below for details).
Thus $w_\beta$ is well-defined.

\begin{lemma}
	Let $\beta \in \PHI$ satisfy $\beta^\omega \in \PHI_\re^\omega$.
	Then
	\begin{align}
		w_\beta(x) = 
		\begin{cases*}
			x - \displaystyle\sum_{\gamma \in O(\beta)} (\gamma^\vee,x)\gamma & if $\#N(\beta) = 0$;\\
			x - \displaystyle\sum_{\{\gamma,\hat\gamma\} \in \hat{O}(\beta)} (\gamma^\vee + \hat{\gamma}^\vee,\, x)(\gamma+\hat\gamma) & if $\#N(\beta) = 1$.
		\end{cases*}
	\end{align}
	for all $x \in E$.
\end{lemma}
\begin{proof}
	We will prove by induction.
	
	Suppose that $\#N(\beta) = 0$.
	Let $S \subsetneq O(\beta)$ and $\gamma_0 \in O(\beta) \setminus S$.
	Then $s_{\gamma_0}(\gamma) = \gamma$ since $(\gamma_0,\gamma) = 0$ for all $\gamma \in S$.
	Hence, for any $x \in E$, we have
	\begin{align}
		s_{\gamma_0}\left( x - \sum_{\gamma \in S}(\gamma^\vee,x)\gamma \right) = x - (\gamma_0^\vee,x)\gamma_0 - \sum_{\gamma \in S}(\gamma^\vee,x)\gamma 
		= x - \sum_{\gamma \in S \cup \{\gamma_0\}}(\gamma^\vee,x)\gamma
	\end{align}
	By induction, we have 
	\begin{align}
		\left(\prod_{\gamma \in O(\beta)}s_\gamma\right)(x) = x - \displaystyle\sum_{\gamma \in O(\beta)} (\gamma^\vee,x)\gamma.
	\end{align}
	
	Suppose that $\#N(\beta) = 1$.
	Let $\{\gamma,\hat\gamma\} \in \hat{O}(\beta)$.
	Then, for any $x \in E$, 
	\begin{align}
		s_\gamma s_{\hat\gamma} s_\gamma(x) = x - (\gamma^\vee + \hat{\gamma}^\vee,\, x)(\gamma+\hat\gamma). \label{sssx}
	\end{align}
	Let $S \subsetneq \hat{O}(\beta)$ and $\{\gamma_0,\hat{\gamma_0}\} \in \hat{O}(\beta)\setminus S$.
	Then $s_{\gamma_0}s_{\hat{\gamma_0}}s_{\gamma_0}(\gamma + \hat\gamma) = \gamma + \hat\gamma$ for all $\{\gamma,\hat\gamma\} \in S$.
	Hence, for any $x \in E$, we have
	\begin{align}
		s_{\gamma_0}s_{\hat{\gamma_0}}s_{\gamma_0}\left( x -  \sum_{\{\gamma,\hat\gamma\} \in S}(\gamma^\vee + \hat{\gamma}^\vee,\, x)(\gamma+\hat\gamma) \right) 
		&= x - (\gamma_0^\vee + \hat{\gamma_0}^\vee,\, x)(\gamma_0 + \hat{\gamma_0}) - \sum_{\{\gamma,\hat\gamma\} \in S}(\gamma^\vee + \hat{\gamma}^\vee,\, x)(\gamma+\hat\gamma)\\
		&= x - \sum_{\{\gamma,\hat\gamma\} \in S \sqcup \{\{\gamma_0,\hat{\gamma_0}\}\}}(\gamma^\vee + \hat{\gamma}^\vee,\, x)(\gamma+\hat\gamma).
	\end{align}
	By induction, we have 
	\begin{align}
		\left(\prod_{\{\gamma,\hat\gamma\} \in \hat{O}(\beta)}s_\gamma s_{\hat\gamma}s_\gamma\right)(x) = x - \displaystyle\sum_{\{\gamma,\hat\gamma\} \in \hat{O}(\beta)} (\gamma^\vee + \hat{\gamma}^\vee,\, x)(\gamma+\hat\gamma). \qedend
	\end{align}
\end{proof}

\begin{lemma}\label{lem6.8}
	Let $\beta \in \PHI$ satisfy $\beta^\omega \in \PHI_\re^\omega$.
	Then $w_\beta$ and $\omega$ commute under the composition.
\end{lemma}
\begin{proof}
	Suppose that $\#N(\beta) = 0$.
	For any $x \in E$, 
	\begin{align}
		w_\beta(\omega(x)) 
		&= \omega(x) - \sum_{\gamma \in O(\beta)}(\gamma^\vee,\omega(x)) \gamma\\
		&= \omega(x) - \sum_{\gamma \in O(\beta)}(\omega^{-1}(\gamma)^\vee,x) \gamma\\
		&= \omega(x) - \sum_{\gamma \in O(\beta)}(\gamma^\vee,x) \omega(\gamma)\\
		&= \omega(w_\beta(x)).
	\end{align}
	
	Suppose that $\#N(\beta) = 1$.
	For any $x \in E$,
	\begin{align}
		w_\beta(\omega(x))
		&= \omega(x) - \sum_{\{\gamma,\hat\gamma\} \in \hat{O}(\beta)} (\gamma^\vee + \hat{\gamma}^\vee,\, \omega(x))(\gamma+\hat\gamma)\\
		&= \omega(x) - \sum_{\{\gamma,\hat\gamma\} \in \hat{O}(\beta)} (\omega^{-1}(\gamma)^\vee + \omega^{-1}(\hat{\gamma})^\vee,\, x)(\gamma+\hat\gamma)\\
		&= \omega(x) - \sum_{\{\gamma,\hat\gamma\} \in \hat{O}(\beta)} (\gamma^\vee + \hat{\gamma}^\vee,\, x)(\omega(\gamma) + \omega(\hat\gamma))\\
		&= \omega(w_\beta(x)). \qedend
	\end{align}
\end{proof}

\begin{lemma}\label{lem6.7}
	Let $\beta \in \PHI$ satisfy $\beta^\omega \in \PHI_\re^\omega$.
	Then $s_{\beta^\omega} = w_\beta$ as the function on $E^\omega$.
\end{lemma}
\begin{proof}
	Let $x \in E^\omega$.
	Note that $(\gamma_1^\vee,x) = (\gamma_2^\vee,x)$ for all $\gamma_1,\gamma_2 \in O(\beta)$ since $\omega(x) = x$.
	\cref{vees} implies that
	\begin{align}
		s_{\beta^{\omega}}(x)
		&= x - ((\beta^{\omega})^\vee,x)\beta^{\omega}\\
		&= x - \left( \frac{2}{2 - \#N(\beta)} \sum_{\gamma_1 \in O(\beta)}(\gamma_1^\vee,x) \right)\frac{1}{\#O(\beta)}\sum_{\gamma_2 \in O(\beta)}\gamma_2\\
		&= x - \frac{2}{2 - \#N(\beta)}\sum_{\gamma_2 \in O(\beta)}\left( \frac{1}{\#O(\beta)} \sum_{\gamma_1 \in O(\beta)}(\gamma_1^\vee,x) \right)\gamma_2\\
		&= x - \frac{2}{2 - \#N(\beta)}\sum_{\gamma_2 \in O(\beta)} (\gamma_2^\vee,x)\gamma_2\\
		&= 
		\begin{cases*}
			x - \displaystyle\sum_{\gamma \in O(\beta)} (\gamma^\vee,x)\gamma & if $\#N(\beta) = 0$;\\
			x - 2\displaystyle\sum_{\gamma \in O(\beta)} (\gamma^\vee,x)\gamma & if $\#N(\beta) = 1$.
		\end{cases*} 
	\end{align}
	
	Suppose that $\#N(\beta) = 1$.
	Let $x \in E^\omega$ and $\{\gamma,\hat\gamma\} \in \hat{O}(\beta)$.
	Since  $(\gamma^\vee,x) = (\hat\gamma^\vee,x)$, we have
	\begin{align}
		\sum_{\{\gamma,\hat\gamma\} \in \hat{O}(\beta)} (\gamma^\vee + \hat{\gamma}^\vee,\, x)(\gamma+\hat\gamma)
		&= \sum_{\{\gamma,\hat\gamma\} \in \hat{O}(\beta)} \Bigl( 2(\gamma^\vee,x)\gamma + 2(\hat\gamma^\vee,x)\hat\gamma \Bigr)\\
		&= 2\displaystyle\sum_{\gamma \in O(\beta)} (\gamma^\vee,x)\gamma.
	\end{align}
	
	Therefore $s_{\beta^\omega}(x) = w_\beta(x)$ for all $x \in E^\omega$.
\end{proof}


\begin{lemma}
	Let $\beta_1,\beta_2 \in \PHI$ satisfy $\beta_1^\omega, \beta_2^\omega \in \PHI_\re^\omega$.
	Then $s_{\beta_1^{\omega}}(\beta_2^{\omega}) \in \PHI^{\omega}_\re$.
	Hence $s_{\beta^{\omega}}(\PHI^{\omega}_\re) = \PHI^{\omega}_\re$ for all $\beta^\omega \in \PHI_\re^\omega$.
\end{lemma}
\begin{proof}
	It follows from \cref{lem6.8} and \cref{lem6.7} that 
	\begin{align}
		s_{\beta_1^{\omega}}(\beta_2^{\omega})
		&= \frac{1}{o(\omega)}\sum_{t=1}^{o(\omega)}w_{\beta_1}(\omega^t(\beta_2))\\
		&= \frac{1}{o(\omega)}\sum_{t=1}^{o(\omega)}\omega^t(w_{\beta_1}(\beta_2))\\
		&= (w_{\beta_1}(\beta_2))^{\omega}.
	\end{align}
	Since $w_{\beta_1}(\beta_2) \in \PHI$, we have $s_{{\beta_1}^{\omega}}(\beta_2^{\omega})= (w_{\beta_1}(\beta_2))^{\omega} \in \PHI^{\omega}_\re$.
\end{proof}

We will varify that the condition \cref{R3} holds.
\begin{lemma}
	Let $\beta_1, \beta_2 \in \PHI$ satisfy $\beta_1^\omega, \beta_2^\omega \in \PHI_\re^\omega$.
	Then $((\beta_1^\omega)^\vee,\, \beta_2^\vee) \in \mathbb{Z}$.
\end{lemma}
\begin{proof}
	It follows from \cref{lem3.9} and \cref{aa} that
	\begin{align}
		((\beta_1^\omega)^\vee,\, \beta_2^\vee)
		&= \frac{2(\beta_1^\omega,\, \beta_2^\omega)}{(\beta_1^\omega,\, \beta_1^\omega)}\\
		&= \frac{2 \cdot \#O(\beta_1)}{(2 - \#N(\beta_1)) \cdot (\beta_1,\beta_1)} \cdot \frac{1}{\#O(\beta_1)}\sum_{\gamma_1 \in O(\beta_1)}(\gamma_1,\beta_2)\\
		&= \frac{2}{2 - \#N(\beta_1)} \sum_{\gamma_1 \in O(\beta_1)}(\gamma_1^\vee,\beta_2)
	\end{align}
	\cref{lem3.10} implies that $2 - \#N(\beta_1)$ divides $2$.
	Since $(\gamma_1^\vee,\beta) \in \mathbb{Z}$ for all $\gamma_1 \in O(\beta_1)$, we have $((\beta_1^\omega)^\vee,\, \beta_2^\vee) \in \mathbb{Z}$.
\end{proof}

The proof that $\PHI^\omega_\re$ is a root system is now complete.

\section{Simple roots} \label{S4.2}\label{sec4}

\subsection{Bases}\ 

\noindent
In this section, we will discuss bases of $\PHI_\re^\omega$ and their construction.
A basis of $\PHI_\re^\omega$ is obtained as follows.
We will leave the proof for later.

\begin{theorem}\label{basistheorem}
	Let $\PHI$ be an arbitrary irreducible redued root system, and $\omega \in \OMEGA$.
	Then $\DELTA^\omega$ is a basis of $\PHI_\re^\omega$.
	Furthermore, a lattice generated by $\{\bar\pi_1^\omega,\ldots,\bar\pi_r^\omega\}$ is the coweight lattice of $\PHI_\re^\omega$.
\end{theorem}

In the process of the proof, we will also verify the following.
\begin{corollary}
	Let $\PHI$ be an arbitrary irreducible reduced root system, and $j \in J$.
	For any $k \in \{1,\ldots,r\}$,
	\begin{align}
		m_k = \frac{\#S_k^j \bar{n}_k}{\#S_0^j}
	\end{align}
	is a positive integer.
\end{corollary}

Let $\beta = \sum_{i=0}^\ell c_i\alpha_i \in \PHI$.
Then we have
\begin{align}
	\beta^{\omega} = \frac{1}{o(\omega)} \sum_{t=1}^{o(\omega)}\omega^t\left(\sum_{i=0}^\ell c_i\alpha_i\right)
	= \sum_{i=0}^\ell c_i \left(\frac{1}{o(\omega)} \sum_{t=1}^{o(\omega)}\omega^t(\alpha_i)\right)
	= \sum_{i=0}^\ell c_i \alpha_i^{\omega}
	= \sum_{k=0}^r\bar{c}_k \bar\alpha_{k}^{\omega},
\end{align}
where
\begin{align}
	\bar{c}_k = \sum_{s \in S_k^j}c_s.
\end{align}
In particular, for $\alpha_0 = -\sum_{i=1}^\ell n_i\alpha_i$, we can see that
\begin{align}
	\bar{\alpha}_0^\omega = \alpha_0^\omega = -\sum_{k=1}^r m_k \bar\alpha_{k}^{\omega}. 
\end{align}
Hence $\beta^\omega$ can be expressed as a linear combination of $\DELTA^\omega$ as follows:
\begin{align}
	\beta^{\omega} = \sum_{k=1}^r \left(\bar{c}_k - m_k\bar{c}_0 \right)\bar\alpha_{k}^{\omega}. \label{lcom}
\end{align}

\subsection{Folding of diagrams}\ 

\noindent
\textbf{Folding} the extended Dynkin diagram $\mathcal{D}_0(\PHI)$ by $\omega \in \OMEGA$ refers to identifying vertices that lie on the same $\langle \omega\rangle$-orbits and redrawing the edges as described in \cref{D1} and \cref{D2}.
Let $\mathcal{D}_0(\PHI)^\omega$ denote the diagram obtained by folding $\mathcal{D}_0(\PHI)$ by $\omega$.
Then $\DELTA_0^\omega$ is the vertex set of $\mathcal{D}_0(\PHI)^\omega$.
Since $\DELTA^\omega$ is a basis of $\PHI_\re^\omega$, the subgraph $\mathcal{D}(\DELTA^\omega)$ of $\mathcal{D}_0(\PHI)^\omega$ induced by $\DELTA^\omega$ is isomorphic to the Dynkin diagram of $\PHI_\re^\omega$.
Hence we can identify the type of the root system $\PHI_\re^\omega$ by folding.

Let $k_1,k_2 \in \{0,\ldots,r\}$ be distinct.
It follows from \cref{lem3.9} that
\begin{align}
	(\bar{\alpha}_{k_1}^\omega,\bar{\alpha}_{k_2}^\omega) = \frac{1}{\#O(\alpha_{s_2})}\sum_{\gamma_2 \in O(\alpha_{s_2})}(\alpha_{s_1},\gamma_2), \label{aapro}
\end{align}
where $s_1 \in S_{k_1}^j$ and $s_2 \in S_{k_2}^j$, and it does not depend on the choice of $s_1$ and $s_2$.
Since $(\alpha_{s_1},\gamma_2) \leq 0$ for all $\gamma_2 \in O(\alpha_{s_2})$, we can see that $(\bar\alpha_{k_1},\bar\alpha_{k_2}) \neq 0$ if and only if $(\alpha_{s_1},\alpha_{s_2}) \neq 0$ for some $s_1 \in S_{k_1}^j$ and $s_2 \in S_{k_2}^j$.
Furthermore, the equation \cref{aa} implies that 
\begin{align}
	(\bar\alpha_{k_1}^\omega,\bar\alpha_{k_1}^\omega) = \frac{(2-\#N(\alpha_{s_{1}})) \cdot (\alpha_{s_{1}},\alpha_{s_{1}})}{2 \cdot \#S_{k_1}^j},\qquad
	(\bar\alpha_{k_2}^\omega,\bar\alpha_{k_2}^\omega) = \frac{(2-\#N(\alpha_{s_{2}})) \cdot (\alpha_{s_{2}},\alpha_{s_{2}})}{2 \cdot \#S_{k_2}^j},
\end{align}
and hence
\begin{align}
	\frac{(\bar\alpha_{k_1}^\omega,\bar\alpha_{k_1}^\omega)}{(\bar\alpha_{k_2}^\omega,\bar\alpha_{k_2}^\omega)} = \frac{(2-\#N(\alpha_{s_{1}})) \cdot \#S_{k_2}^j \cdot (\alpha_{s_{1}},\alpha_{s_{1}})}{(2-\#N(\alpha_{s_{2}})) \cdot \#S_{k_1}^j \cdot (\alpha_{s_{2}},\alpha_{s_{2}})},
	\label{foldrule}
\end{align}
where $s_1 \in S_{k_1}^j$ and $s_2 \in S_{k_2}^j$, and it does not depend on the choice of $s_1$ and $s_2$.
The equation \cref{foldrule} allows us to determine how the edge between $\bar\alpha_{k_1}^\omega$ and $\bar\alpha_{k_2}^\omega$ should be drawn.


\begin{lemma}
	Folding the extended Dynkin diagram $\mathcal{D}_0(\PHI)$ by $\omega$ is as shown in \cref{figfDDI} and \cref{figfDDII}.
\end{lemma}

The above arguments imply the irreducibility of $\PHI_\re^\omega$.
\begin{lemma}
	Let $\PHI$ be an arbitrary irreducible redued root system, and $\omega \in \OMEGA$.
	Then the root system $\PHI_\re^\omega$ is irreducible.
\end{lemma}

In most cases, we can determine the type of $\PHI_\re^\omega$ by $\mathcal{D}_0(\PHI)^\omega$.
However, only when $\mathcal{D}(\DELTA^\omega)$ is of type $A_1$ or $B_r$, we need to check whether $\PHI_\re^\omega$ is reduced or non-reduced.


The diagram $\mathcal{D}(\DELTA^\omega)$ is of type $A_1$ or $B_r$ if the following condition is satisfied:
\begin{enumerate}[label=(r-\arabic*)]
	\item\label{BCa} $\PHI$ is of type $B_\ell$ and $\ell \geq 3$;
	\item\label{BCb} $\PHI$ is of type $C_\ell$ and $\ell \geq 3$ is odd;
	\item\label{BCc} $\PHI$ is of type $D_\ell$ and $j = 1$;
	\item\label{BCd} $\PHI$ is of type $D_4$ and $j \in \{\ell-1,\, \ell\}$;
	\item\label{BCe} $\PHI$ is of type $D_\ell$, $\ell \geq 5$ is odd and $j \in \{\ell-1,\, \ell\}$.
\end{enumerate}

\begin{lemma}\label{reducedlemma}
	If $\PHI$ has a simple root $\alpha_i \in \DELTA$ such that $\#N(\alpha_i) = 1$, then $\PHI_\re^\omega$ is non-reduced (i.e. $\PHI_\re^\omega$ is of type $BC_r$).
\end{lemma}
\begin{proof}
	Suppose that $\alpha_{i_0} \in \DELTA$ satisfies $\#N(\alpha_{i_0}) = 1$. 
	Then $\alpha_{i_0}^\omega$ is a short simple root of $\PHI_\re^\omega$.
	Let $\alpha_{i_1} \in N(\alpha_{i_0})$.
	Then $\alpha_{i_0} + \alpha_{i_1} \in \PHI$, and
	\begin{align}
		(\alpha_{i_0} + \alpha_{i_1})^\omega = \alpha_{i_0}^\omega + \alpha_{i_1}^\omega = 2\alpha_{i_0}^\omega.
	\end{align}
	Thus $2\alpha_{i_0}^\omega \in \PHI_\re^\omega$, and hence $\PHI_\re^\omega$ is non-reduced. 
\end{proof}

Only in the case of \cref{BCb} and \cref{BCe} does it have a simple root $\alpha_i \in \DELTA$ such that $\#N(\alpha_i) = 1$.
We can see that $\PHI_\re^\omega$ becomes reduced only when the condition of the above is satisfied, although we will defer the proof that $\PHI_\re^\omega$ is reduced in cases \cref{BCa}, \cref{BCc} and \cref{BCd}.

\begin{lemma}\label{lem4.6}
	Of the five above, $\PHI_\re^\omega$ is non-reduced (i.e. $\PHI_\re^\omega$ is of type $BC_r$) only for \cref{BCb} and \cref{BCe}.
	As the result,  $\PHI_\re^\omega$ is non-reduced if and only if $\PHI$ has a simple root $\alpha_i \in \DELTA$ such that $\#N(\alpha_i) = 1$.
\end{lemma}

\begin{theorem}\label{typetheorem}
	The type of $\PHI_\re^\omega$ is as shown in \cref{tableP}.
\end{theorem}

\begin{table}[h]
	
	\caption{List of type of $\PHI_\re^\omega$}\label{tableP}
	
	\begin{tabular}{ll|c|c|c}
		\multicolumn{2}{c|}{$\PHI$}                                                        & $j$                              & $\PHI^{\omega}_\re$ 		& cf. \\ \hline\hline
		$A_\ell$                  &                         & $\gcd\{\ell+1,j\} = 1$		  & $\varnothing$         		& \multirow{2}{*}{\cref{ZA}} \\ 
		$A_\ell$                  &                         & $g \ceq \gcd\{\ell+1,j\} \neq 1$		  & $A_{g-1}$         		& \\ 
		$B_2$                     &                            & $j = 1$                              & $A_1$               			    & \multirow{2}{*}{\cref{ZB}} \\ 
		$B_\ell$                  & $(\ell \geq 3)$                        & $j = 1$                              & $B_{\ell-1}$            &\\ 
		$C_\ell$                  & $(\ell \geq 3$, odd)                  & $j = \ell$                           & $BC_{\frac{\ell-1}{2}_{_{}}}$    & \multirow{2}{*}{\cref{ZC}} \\ 
		$C_\ell$                  & $(\ell \geq 4$, even)                  & $j = \ell$                           & $C_{\frac{\ell}{2}_{_{}}}$    &\\ 
		$D_\ell$                  & $(\ell \geq 4$)                  & $j = 1 $                           & $B_{\ell-2}$    & \cref{ZD1} \\ 
		$D_\ell$                  & $(\ell \geq 4$, even)                  & $j \in \{\ell-1,\, \ell\}$                           & $C_{\frac{\ell}{2}_{_{}}}$    & \cref{ZDe} \\ 
		$D_\ell$                  & $(\ell \geq 5$, odd)                  & $j \in \{\ell-1,\, \ell\}$                           & $BC_{\frac{\ell-3}{2}_{_{}}}$    & \cref{ZDo} \\ 
		$E_6$                     &                           & $j \in \{1,6\}$                         & $G_2$          			        & \cref{ZE6} \\ 
		$E_7$                     &                            & $j = 7$                              & $F_4$               			    & \cref{ZE7}\\ 
	\end{tabular}
	
	\

\end{table}

\subsection{Disappearing roots and Extra roots}\ 

\noindent
Let $\PHI_{\omega,0}^+$ denote a set of positive roots that disappear due to folding by $\omega$:
\begin{align}
	\PHI_{\omega,0}^+ \ceq \bigset{\beta \in \PHI^+}{\beta^\omega = 0}.
\end{align}
By \cref{lcom}, a root $\beta = \sum_{i=1}^\ell c_i \alpha_i \in \PHI$ belongs to $\PHI_{\omega,0}^+$ if and only if
\begin{align}
	\bar{c}_k - m_k\bar{c}_0 = 0
\end{align}
for all $k \in \{1,\ldots,r\}$.

Let $\DELTA' \ceq \DELTA \setminus S_0^j$ and $\mathcal{D}(\DELTA')$ denote the subgraph of $\mathcal{D}(\PHI)$ induced by $\DELTA'$.
Then $\mathcal{D}(\DELTA')$ is isomorphic to some Dynkin diagram, and let $\PHI'$ be a root system corresponding to $\mathcal{D}(\DELTA')$.
Each $\omega \in \OMEGA$ can be regarded as a graph automorphism of $\mathcal{D}(\DELTA')$, and a root system $(\PHI')^\omega$, obtained by folding $\PHI'$ with $\omega$, is isomorphic to $\PHI_\re^\omega$.
Hence we have the following:
\begin{lemma}\label{lem4.8}
	Let $\beta^\omega \in \PHI_\re^\omega$.
	Then there exists $\beta' = \sum_{i=1}^\ell c_i\alpha_i \in \PHI$ such that $(\beta')^\omega = \beta^\omega$ and $c_s = 0$ for all $s \in S_0^j$.
\end{lemma}
The $\beta'$ in the above lemma satisfies $(\beta',\pi_0^j) = \sum_{s \in S_0^j}c_s = 0$.

\begin{definition}
	Let $\beta^\omega \in \PHI^\omega$ and $p \in \mathbb{Z}_{\geq 0}$.
	We say that $\beta^\omega$ has a \textbf{${\bm p}$-extra root} if there exists a root 
	$\beta' \in \PHI$ such that $(\beta')^\omega = \beta^\omega$ and $(\beta',\pi_0^j) = p$.
	We refer to such root $\beta'$ as \textbf{${\bm p}$-extra root} of $\beta^\omega$.
	Moreover, define a set $P(\beta^\omega)$ by
	\begin{align}
		P(\beta^\omega) &= \bigset{p \in \mathbb{Z}}{\text{$\beta^\omega$ has a $p$-extra root}}\\
		&= \bigset{(\gamma,\pi_0^j)}{\gamma \in \PHI,\ \gamma^\omega = \beta^\omega}.
	\end{align}
\end{definition}
For any $\beta^\omega \in \PHI_\re^\omega$, \cref{lem4.8} implies that $0$ must belong to $P(\beta^\omega)$.
Since $\#S_0^j = o(\omega)$, we have
\begin{align}
	P(\beta^\omega) \subseteq \{0,\, \pm1,\, \ldots,\, \pm(o(\omega)-1)\}
\end{align}
for all $\beta^\omega \in \PHI^\omega$.

In the next section, we will determine $\PHI_{\omega,0}$ for each $\PHI$ and $\omega$, and $P(\beta^\omega)$ for each $\beta^\omega \in (\PHI^\omega_\re)^+ \sqcup \{0\}$.

\section{Considerations for each type}\label{sec5}

\noindent
Let $\PHI$ be an irreducible reduced root system and $\omega \in \OMEGA$.

In this section, we will prove that $\DELTA^\omega$ is a basis for $\PHI_\re^\omega$ (proof of \cref{basistheorem}).
In addition, do the following:
\begin{iitemize}
	\item Prove that $\PHI_\re^\omega$ is reduced in cases \cref{BCa}, \cref{BCc} and \cref{BCd} (proof of \cref{lem4.6});
	\item Represent $\PHI_{\omega,0}^+$ explicitly;
	\item Determine $P(\beta^\omega)$ for each $\beta^\omega \in (\PHI^\omega_\re)^+ \sqcup \{0\}$.
\end{iitemize}

%
%
%
%
%
%

Let $\beta = \sum_{i=0}^\ell c_i\alpha_i \in \PHI$.
Recall that 
\begin{align}
	\beta^{\omega} = \sum_{k=1}^r \left(\bar{c}_k - m_k\bar{c}_0 \right)\bar\alpha_{k}^{\omega}, \label{cre}
\end{align}
as shown in \cref{lcom}.
To prove \cref{basistheorem}, it suffices to show that the following claim:
\begin{claim}\label{claim}
	For any $\beta \in \PHI^+$,
	the coefficients on the right-hand side of \cref{cre} satisfy one of the following:
	\begin{iitemize}
		\item all non-negative;
		\item all non-positive;
		\item all zero (then $\beta \in \PHI_{\omega,0}^+$).
	\end{iitemize}
\end{claim}
It is clear when $\bar{c}_0 = 0$ or $\beta^\omega = \alpha_0^\omega$.
Moreover, if $\bar{c}_k - m_k\bar{c}_0 = 0$ for all $k \in \{1,\ldots,r\}$ except exactly one, then the above claim holds.
We will show the above claim separately for each type of $\PHI$.

\subsection{Type $A_\ell$}\ 

\noindent
Let $\PHI$ be a root system of type $A_\ell$.
Define $\alpha_1,\ldots,\alpha_\ell$ by 
\begin{align}
	\alpha_1 \ceq e_1 - e_2,\quad \alpha_2 \ceq e_2 - e_3,\quad \ldots,\quad \alpha_\ell \ceq e_\ell - e_{\ell+1}
\end{align}
for the standard basis $\{e_1,\ldots,e_{\ell+1}\}$ of $\mathbb{R}^{\ell+1}$.
Then $\DELTA = \{\alpha_1,\ldots,\alpha_\ell\}$ is a basis of $\PHI$, and
\begin{align}
	\PHI^+ &= \bigset{e_{p_1} - e_{p_2}}{p_1,p_2 \in \{1,\ldots,\ell+1\},\ p_1 < p_2}\\
	&=\bigset{\alpha_{i_1} + \cdots + \alpha_{i_2}}{i_1,i_2 \in \{1,\ldots,\ell\},\ i_1 \leq i_2}
\end{align}
by \cite[Plate \rom{1}]{Bourbaki}.

Let $j \in J \setminus\{0\}$ and $\omega = \omega_j$.
Then $\sigma = \sigma_j = (0\ 1\ \cdots\ \ell)^j$.
We can see that $g \ceq \gcd\{\ell+1,\, j\} = r+1$ and $o(\omega) = \frac{\ell+1}{r+1}$.
Suppose that $o(\omega) \neq \ell+1$ (that is, $r \neq 0$) so that $\PHI_\re^\omega \neq \emptyset$.
Let $S_0^j,\ldots,S_r^j$ denote the $\langle \sigma \rangle$-orbits of $I$.
For any $k \in \{0,\ldots,r\}$, we may assume that
\begin{align}
	S_k^j = \bigset{\sigma^t(k)}{t \in \{1,\ldots,o(\omega)\}} = \bigset{k + bg}{b \in \{0,\ldots,o(\omega)-1\}}.
\end{align}
Since $\#S_k^j = o(\omega)$ and $\bar{n}_k = 1$, we have $m_k = 1$ for all $k \in \{0,\ldots,r\}$.


We consider \cref{claim}.
Let $\beta = \alpha_{i_1} + \cdots + \alpha_{i_2} \in \PHI^+$.
Then
\begin{align}
	\bar{c}_k = \#\bigset{s \in S_k^j}{i_1 \leq s \leq i_2}
\end{align}
for all $k \in \{0,\ldots,r\}$.
Suppose that $p \ceq \bar{c}_0 > 0$ and there exists $k_0 \in \{1,\ldots,r\}$ such that $\bar{c}_{k_0} - m_{k_0}\bar{c}_0 > 0$, that is, $\bar{c}_{k_0} > p$.
Let 
\begin{align}
	s_1\ceq \min\bigset{s \in S_{k_0}^j}{i_1 \leq s \leq i_2},\qquad
	s_2 \ceq \max\bigset{s \in S_{k_0}^j}{i_1 \leq s \leq i_2}.
\end{align}
Then $s_2 \geq s_1 + pg$ since $\bar{c}_{k_0} > p$.
Hence we have
\begin{align}
	\bar{c}_k \geq \#\bigset{i \in S_k^j}{s_1 \leq i \leq s_2} \geq p
\end{align}
for all $k \in \{1,\ldots,r\}$.
Therefore $\bar{c}_k - m_k\bar{c}_0 \geq p - p = 0$ for all $k \in \{1,\ldots,r\}$.

Assume that there exists $k_1,k_2 \in \{1,\ldots,r\}$ such that 
\begin{align}
	\bar{c}_{k_1} - m_{k_1}\bar{c}_0 < 0,\qquad
	\bar{c}_{k_2} - m_{k_2}\bar{c}_0 > 0.
\end{align}
It follows from the above that $\bar{c}_{k_1} - m_{k_1}\bar{c}_0 \geq 0$, but this is a contradiction.

Hence \cref{claim} holds true when $\PHI$ is of type $A_\ell$.

 \label{ZA}

\begin{proposition}\label{disapA}
	Let $\PHI$ be a root system of type $A_\ell$, and $j \in J \setminus\{0\}$.
	Then
	\begin{align}
		\PHI_{\omega,0}^+ 
		&= \bigset{e_{p_1} - e_{p_2}}{p_1,p_2 \in \{1,\ldots,\ell+1\},\ p_2 - p_1 \in g\mathbb{Z}_{>0}}\\
		&= \bigset{\alpha_{i_1} + \cdots + \alpha_{i_2}}{i_1,i_2 \in \{1,\ldots,\ell\},\ i_2 - i_1 + 1 \in g\mathbb{Z}_{>0}}.
	\end{align}
	Hence
	\begin{align}
		P(0) = \{\pm1,\, \ldots,\, \pm(o(\omega)-1)\}.
	\end{align}
\end{proposition}
\begin{proof}
	It is clear that $\alpha_{i_1} + \cdots + \alpha_{i_2} \in \PHI_{\omega,0}^+$ if $i_2 - i_1 + 1 \in g\mathbb{Z}_{>0}$.
	Suppose that $\alpha_{i_1} + \cdots + \alpha_{i_2} \in \PHI_{\omega,0}^+$.
	Then $p \ceq \bar{c}_0 = \bar{c}_1 = \cdots = \bar{c}_r$.
	Let $k_1 \in \{0,\ldots,r\}$ satisfy $i_1 \in S_{k_1}^j$.
	Then 
	\begin{align}
		\bigset{s \in S_{k_1}^j}{i_1 \leq s \leq i_2} = \bigset{i_1 + bg}{b \in \{0,\ldots,p-1\}}.
	\end{align}
	In particular, we have $i_2 < i_1 + pg$.
	For $k_0 \in \{0,\ldots,r\}$ defined as $k_0 = k_1 - 1$ if $k_1 > 0$, and $k_0 = r$ if $k_1 = 0$, we have
	\begin{align}
		\bigset{s \in S_{k_0}^j}{i_1 \leq s \leq i_2} = \bigset{i_1 + (g-1) + bg}{b \in \{0,\ldots,p-1\}}.
	\end{align}
	In particular, $i_1 + pg - 1 \leq i_2$.
	Hence $i_2 - i_1 + 1 = pg \in g\mathbb{Z}_{>0}$.
\end{proof}

By \cref{typetheorem}, $\PHI_\re^\omega$ is of type $A_{g-1}$ ($r = g-1$).
\begin{proposition}\label[prop]{extraA}
	Let $\PHI$ be a root system of type $A_\ell$, and $j \in J \setminus\{0\}$.
	For any $\beta^\omega \in (\PHI_\re^\omega)^+$, we have
	\begin{align}
		P(\beta^\omega) = \{0,\ \pm1,\ \ldots,\ \pm(o(\omega)-1)\}.
	\end{align}
\end{proposition}
\begin{proof}
	Suppose that $\beta^\omega = \bar\alpha_{k_1}^\omega + \cdots + \bar\alpha_{k_2}^\omega \in (\PHI_\re^\omega)^+$ for $k_1,k_2 \in \{1,\ldots,r\}$, and $k_1 \leq k_2$
	Define $\beta' \ceq \alpha_{k_1} + \cdots + \alpha_{k_2} \in \PHI^+$.
	Then $\beta'$ is a $0$-extra root of $\beta^\omega$.
	For any $p \in \{1,\ldots,o(\omega)-1\}$, \cref{disapA} implies that 
	\begin{align}
		\gamma_p^+ \ceq \alpha_{k_t + 1} + \cdots + \alpha_{k_t + 1 + pr},\qquad 
		\gamma_p^- \ceq \alpha_{k_1} + \cdots + \alpha_{k_1 + pr}
	\end{align}
	belong to $\PHI_{\omega,0}^+$, and $\beta' + \gamma_p^+,\ \beta' - \gamma_p^- \in \PHI$ satisfy $(\gamma_p^+,\pi_0^j) = (\gamma_p^-,\pi_0^j) = p$. 
	Furthermore, $(\beta' + \gamma_p^+)^\omega = (\beta' - \gamma_p^-)^\omega = \beta^\omega$.
	Hence $\beta' + \gamma_p^+$ is a $p$-extra root and $\beta' - \gamma_p^-$ is a $(-p)$-extra root of $\beta^\omega$.
\end{proof}

\subsection{Type $B_\ell$}\ 

\noindent
Let $\PHI$ be a root system of type $B_\ell$.
Define $\alpha_1,\ldots,\alpha_\ell$ by 
\begin{align}
	\alpha_1 \ceq e_1 - e_2,\quad \alpha_2 \ceq e_2 - e_3,\quad \ldots,\quad \alpha_\ell \ceq e_\ell
\end{align}
for the standard basis $\{e_1,\ldots,e_{\ell}\}$ of $\mathbb{R}^{\ell}$.
Then $\DELTA = \{\alpha_1,\ldots,\alpha_\ell\}$ is a basis of $\PHI$, and
\begin{align}
	\PHI^+ 
	&= \bigset{e_{p_0}}{p_0 \in \{1,\ldots,\ell\}} \sqcup \bigset{e_{p_1} - e_{p_2}}{p_1,p_2 \in \{1,\ldots,\ell\},\ p_1 < p_2}\\
	&\qquad  
	\sqcup \bigset{e_{p_1} + e_{p_2}}{p_1,p_2 \in \{1,\ldots,\ell\},\ p_1 < p_2} \\
	&= \bigset{\alpha_{i_0} + \cdots + \alpha_{\ell}}{i_0 \in \{1,\ldots,\ell\}} \sqcup \bigset{\alpha_{i_1} + \cdots + \alpha_{i_2}}{i_1,i_2 \in \{1,\ldots,\ell-1\},\ i_1 \leq i_2} \\
	&\qquad 
	\sqcup \bigset{\alpha_{i_1} + \cdots + \alpha_{i_2-1} + 2\alpha_{i_2} + \cdots + 2\alpha_{\ell}}{i_1,i_2 \in \{1,\ldots,\ell\},\ i_1 < i_2}
\end{align}
by \cite[Plate \rom{2}]{Bourbaki}.
The set of short roots of $\PHI^+$ is 
\begin{align}
	\bigset{e_{p_0}}{p_0 \in \{1,\ldots,\ell\}} = \bigset{\alpha_{i_0} + \cdots + \alpha_{\ell}}{i_0 \in \{1,\ldots,\ell\}}.
\end{align}

Let $j = 1$ and $\omega = \omega_1$.
Then $\sigma = (0\ 1)$ and $r = \ell-1$.
For any $k \in \{1,\ldots,r\}$, we may assume that
\begin{align}
	S_0^j = \{0,1\},\qquad
	S_k^j = \{k+1\}.
\end{align}
For any $k \in \{1,\ldots,r\}$, since $\bar{n}_k = 2$, we have $m_k = 1$.


We consider \cref{claim}.
Let $\beta = \sum_{i=1}^\ell c_i \alpha_i \in \PHI^+$ satisfy $\bar{c}_0 = c_1 = 1$.
Then $\beta$ is equal to either of the following:
\begin{eenumeratea}
	\item\label{Ba} $\alpha_1 + \cdots + \alpha_{i_2}$ for $i_2 \in \{1,\ldots,\ell\}$;
	\item\label{Bb} $\alpha_1 + \cdots + \alpha_{i_2-1} + 2\alpha_{i_2} + \cdots + 2\alpha_{\ell}$ for $i_2 \in \{2,\ldots,\ell\}$.
\end{eenumeratea}
In the case of \cref{Ba}, we have
\begin{align}
	\bar{c}_k - m_k\bar{c}_0 = \begin{cases*}
		1 - 1 = 0 	& if $k < i_2$;\\
		0 - 1 = -1 	& if $k \geq i_2$
	\end{cases*}
\end{align}
for all $k \in \{1,\ldots,r\}$.
In particular, if $i_2 = r + 1 = \ell$, then $\beta^\omega = 0$.

In the case of \cref{Bb}, we have
\begin{align}
	\bar{c}_k - m_k\bar{c}_0 = \begin{cases*}
		1 - 1 = 0 	& if $k < i_2 - 1$;\\
		2 - 1 = 1 	& if $k \geq i_2 - 1$
	\end{cases*}
\end{align}
for all $k \in \{1,\ldots,r\}$.
In particular, $\beta^\omega \neq 0$ regardless of $i_2$.

Hence \cref{claim} holds true when $\PHI$ is of type $B_\ell$.

If $\ell \geq 3$, then the diagram $\mathcal{D}(\DELTA^\omega)$ is of type $B_{\ell-1}$.
\begin{proposition}\label{reducedB}
	Let $\PHI$ be a root system of type $B_\ell$ for $\ell \geq 3$.
	Then $\PHI_\re^\omega$ is reduced.
\end{proposition}
\begin{proof}
	We can see that $\bar\alpha_r^\omega$ is a short simple root of $\PHI_\re^\omega$.
	Assume that $2\bar\alpha_r^\omega \in \PHI_\re^\omega$.
	Let $\beta = \sum_{i=1}^\ell c_i\alpha_i \in \PHI$ satisfy $\beta^\omega = 2\bar\alpha_r^\omega$.
	By the above observation, $\bar{c}_0$ must be $0$ since $\bar{c}_r - m_r\bar{c}_0 = 2$.
	Thus we have 
	\begin{align}
		\bar{c}_k = \bar{c}_k - m_k\bar{c}_0 = \begin{cases*}
			0 & if $k \neq r$;\\
			2 & if $k = r$.
		\end{cases*}
	\end{align}
	However, there exists no root $\beta$ that satisfies this condition, which is a contradiction.
	Hence $\PHI_\re^\omega$ is reduced in this case.
\end{proof}

\label{ZB}
\begin{proposition}\label{disapB}
	Let $\PHI$ be a root system of type $B_\ell$.
	Then
	\begin{align}
		\PHI_{\omega,0}^+ = \{e_1\} = \{\alpha_1 + \cdots + \alpha_\ell\}.
	\end{align}
	Hence 
	\begin{align}
		P(0) = \{\pm 1\}.
	\end{align}
\end{proposition}
\begin{proof}
	It follows from the above observation.
\end{proof}

By \cref{typetheorem}, if $\ell = 2$, then $\PHI_\re^\omega$ is of type $A_1$.

\begin{proposition}\label[prop]{extraB2}
	Let $\PHI$ be a root system of type $B_2$.
	Then $P(\bar\alpha_1^\omega) = \{0,\pm1\}$.
\end{proposition}
\begin{proof}
	We have $(\alpha_1 + 2\alpha_2,\, \pi_0^j) = 1$, $(-\alpha_1,\pi_0^j) = -1$, and
	\begin{align}
		(\alpha_1 + 2\alpha_2)^\omega = (-\alpha_1)^\omega = \alpha_2^\omega = \bar\alpha_1^\omega. \qedend
	\end{align}
\end{proof}

If $\ell \geq 3$, then $\PHI_\re^\omega$ is of type $B_{\ell-1}$ ($r = \ell-1$).

\begin{proposition}\label[prop]{extraB}
	Let $\PHI$ be a root system of type $B_\ell$ for $\ell \geq 3$.
	For $\beta^\omega \in (\PHI_\re^\omega)^+$, we have
	\begin{align}
		P(\beta^\omega) = \begin{cases*}
			\{0,\pm1\}	& if $\beta^\omega$ is a short root of $\PHI_\re^\omega$;\\
			\{0\}		& if $\beta^\omega$ is a long root of $\PHI_\re^\omega$.
		\end{cases*}
	\end{align}
\end{proposition}
\begin{proof}
	Suppose that $\beta^\omega \in (\PHI_\re^\omega)^+$ is a short root of $\PHI_\re^\omega$, that is, 
	\begin{align}
		\beta^\omega = \bar\alpha_{k_0}^\omega + \cdots + \bar\alpha_r^\omega
	\end{align}
	for some $k_0 \in \{1,\ldots,r\}$.
	Then $\beta' \ceq \alpha_{k_0+1} + \cdots + \alpha_{\ell}$ is a $0$-extra root of $\beta^\omega$.
	Let $\gamma \ceq \alpha_1 + \cdots + \alpha_\ell \in \PHI_{\omega,0}^+$.
	Then $\beta' + \gamma$ and $\beta' - \gamma$ belong to $\PHI$.
	Since $(\beta' + \gamma)^\omega = (\beta' - \gamma)^\omega = (\beta')^\omega = \beta^\omega$, the root $\beta' + \gamma$ is a $1$-extra root and $\beta' - \gamma$ is a $(-1)$-extra root of $\beta^\omega$.

	Let $\beta^\omega \in (\PHI_\re^\omega)^+$.
	Suppose that $\beta' = \sum_{i=1}^\ell c_i'\alpha_i \in \PHI^+$ be a $1$-extra root of $\beta^\omega$, that is, $c_1' = 1$.
	Since $(\beta')^\omega \in (\PHI_\re^\omega)^+$, $\beta'$ must satisfy the conditions in \cref{Bb}.
	Therefore the above observation implies that
	\begin{align}
		\beta^\omega = (\beta')^\omega = \bar\alpha_{k_0}^\omega + \cdots + \bar\alpha_r^\omega
	\end{align}
	for some $k_0 \in \{1,\ldots,r\}$, and hence $\beta^\omega$ is a short root of $\PHI_\re^\omega$.
	
	Suppose that $\beta' = \sum_{i=1}^\ell c_i'\alpha_i \in \PHI^-$ be a $(-1)$-extra root of $\beta^\omega$, that is, $c_1' = -1$.
	Since $(\beta')^\omega \in (\PHI_\re^\omega)^+$, $\beta'$ must satisfy the conditions in \cref{Ba}, that is, $\beta' = - (\alpha_1 + \cdots + \alpha_{i_2})$ for some $i_2 \in \{1,\ldots,\ell-1\}$
	Hence
	\begin{align}
		\beta^\omega = (\beta')^\omega = \bar\alpha_{i_2}^\omega + \cdots + \bar\alpha_r^\omega.
	\end{align} 
	Therefore $\beta^\omega$ is a short root of $\PHI_\re^\omega$.
\end{proof}

\subsection{Type $C_\ell$}\ 

\noindent
Let $\PHI$ be a root system of type $C_\ell$.
Define $\alpha_1,\ldots,\alpha_\ell$ by 
\begin{align}
	\alpha_1 \ceq e_1 - e_2,\quad \alpha_2 \ceq e_2 - e_3,\quad \ldots,\quad \alpha_\ell \ceq 2e_\ell
\end{align}
for the standard basis $\{e_1,\ldots,e_{\ell}\}$ of $\mathbb{R}^{\ell}$.
Then $\DELTA = \{\alpha_1,\ldots,\alpha_\ell\}$ is a basis of $\PHI$, and
\begin{align}
	\PHI^+ 
	&= \bigset{e_{p_1} - e_{p_2}}{p_1,p_2 \in \{1,\ldots,\ell\},\ p_1 < p_2} \sqcup \bigset{e_{p_1} + e_{p_2}}{p_1,p_2 \in \{1,\ldots,\ell\},\ p_1 < p_2}\\
	&\qquad  
	\sqcup \bigset{2e_{p_0}}{p_0 \in \{1,\ldots,\ell\}}  \\
	&= \bigset{\alpha_{i_1} + \cdots + \alpha_{i_2}}{i_1,i_2 \in \{1,\ldots,\ell-1\},\ i_1 \leq i_2}\\
	&\qquad \sqcup \bigset{\alpha_{i_1} + \cdots + \alpha_{i_2-1} + 2\alpha_{i_2} + \cdots + 2\alpha_{\ell-1} + \alpha_\ell}{i_1,i_2 \in \{1,\ldots,\ell\},\ i_1 < i_2}\\
	&\qquad  \sqcup \bigset{2\alpha_{i_0} + \cdots + 2\alpha_{\ell-1} +\alpha_{\ell}}{i_0 \in \{1,\ldots,\ell\}}
\end{align}
by \cite[Plate \rom{3}]{Bourbaki}.
The set of short roots of $\PHI^+$ is 
\begin{align}
	&\bigset{e_{p_1} - e_{p_2}}{p_1,p_2 \in \{1,\ldots,\ell\},\ p_1 < p_2} \sqcup \bigset{e_{p_1} + e_{p_2}}{p_1,p_2 \in \{1,\ldots,\ell\},\ p_1 < p_2} \\
	&\qquad = \bigset{\alpha_{i_1} + \cdots + \alpha_{i_2}}{i_1,i_2 \in \{1,\ldots,\ell-1\},\ i_1 \leq i_2} \\
	&\qquad \qquad \sqcup \bigset{\alpha_{i_1} + \cdots + \alpha_{i_2-1} + 2\alpha_{i_2} + \cdots + 2\alpha_{\ell-1} + \alpha_\ell}{i_1,i_2 \in \{1,\ldots,\ell\},\ i_1 < i_2}
\end{align}

Let $j = \ell$ and $\omega = \omega_\ell$.
Then $\sigma = (0\ \ell)\prod_{i=1}^{\lfloor \frac{\ell-1}{2} \rfloor}(i\ \ \ell-i)$ and $r = \lfloor \frac{\ell}{2} \rfloor$.
For any $k \in \{1,\ldots,r\}$, we may assume that
\begin{align}
	S_k^j = \{k,\, \ell-k\}.
\end{align}
For any $k \in \{1,\ldots,r\}$, since $\bar{n}_k = 2$, we have 
\begin{align}
	m_k = \begin{cases*}
		1 & if $k = \frac{\ell}{2}$;\\
		2 & if $k \neq \frac{\ell}{2}$.
	\end{cases*}
\end{align}

	
We consider \cref{claim}.
Let $\beta = \sum_{i=1}^\ell c_i \alpha_i \in \PHI^+$ satisfy $\bar{c}_0 = c_\ell = 1$.
Then $\beta$ is equal to one of the following:
\begin{eenumeratea}
	\item\label{Ca} $\alpha_{i_1} + \cdots + \alpha_\ell$ for $i_1 \in \{1,\ldots,\ell\}$;
	\item\label{Cb} $\alpha_{i_1} + \cdots + \alpha_{i_2 - 1} + 2\alpha_{i_2} + \cdots + 2\alpha_{\ell-1} + \alpha_\ell$ for $i_1,i_2 \in \{1,\ldots,\ell\}$ and $i_1 < i_2$;
	\item\label{Cc} $2\alpha_{i_0} + \cdots + 2\alpha_{\ell-1} + \alpha_\ell$ for $i_0 \in \{1,\ldots,\ell\}$.
\end{eenumeratea}

Suppose that $\beta$ satisfies \cref{Ca}.
If $i_1 > \frac{\ell}{2}$, then we have
\begin{align}
	\bar{c}_k - m_k\bar{c}_0 = \begin{cases*}
		1 - 2 = -1	& if $k \leq \ell - i_1$;\\
		0 - 2 = -2	& if $\ell - i_1 < k < \frac{\ell}{2}$;\\
		0 - 1 = -1	& if $k = \frac{\ell}{2}$
	\end{cases*}
\end{align}
for all $k \in \{1,\ldots,r\}$.
In particular, $\beta^\omega\neq 0$ regardless of $i_1$.
If $i_1 \leq \frac{\ell}{2}$, then we have
\begin{align}
	\bar{c}_k - m_k\bar{c}_0 = \begin{cases*}
		1 - 2 = -1	& if $k < i_1$;\\
		2 - 2 = 0	& if $i_1 \leq k < \frac{\ell}{2}$;\\
		1 - 1 = 0	& if $k = \frac{\ell}{2}$
	\end{cases*}
\end{align}
for all $k \in \{1,\ldots,r\}$.
In particular, if $i_1 = 1$, then $\beta^\omega = 0$.

Suppose that $\beta$ satisfies \cref{Cb}.
If $i_2 > i_1 > \frac{\ell}{2}$, then we have
\begin{align}
	\bar{c}_k - m_k\bar{c}_0 = \begin{cases*}
		2 - 2 = 0	& if $k \leq \ell - i_2$;\\
		1 - 2 = -1	& if $\ell - i_2 < k \leq \ell - i_1$;\\
		0 - 2 = -2	& if $\ell - i_1 < k < \frac{\ell}{2}$;\\
		0 - 1 = -1	& if $k = \frac{\ell}{2}$
	\end{cases*}
\end{align}
for all $k \in \{1,\ldots,r\}$.
In particular, $\beta^\omega\neq 0$ regardless of $i_1$ and $i_2$.
If $i_2 > \frac{\ell}{2} \geq i_1$ and $\ell - i_2 < i_1$, then we have
\begin{align}
	\bar{c}_k - m_k\bar{c}_0 = \begin{cases*}
		2 - 2 = 0	& if $k \leq \ell - i_2$;\\
		1 - 2 = -1	& if $\ell - i_2 < k < i_1$;\\
		2 - 2 = 0	& if $i_1 \leq k < \frac{\ell}{2}$;\\
		1 - 1 = 0	& if $k = \frac{\ell}{2}$
	\end{cases*}
\end{align}
for all $k \in \{1,\ldots,r\}$.
In particular, if $i_1 + i_2 = \ell + 1$, then $\beta^\omega = 0$.
If $i_2 > \frac{\ell}{2} \geq i_1$ and $\ell - i_2 \geq i_1$, then we have
\begin{align}
	\bar{c}_k - m_k\bar{c}_0 = \begin{cases*}
		2 - 2 = 0	& if $k < i_1$;\\
		3 - 2 = 1	& if $i_1 \leq  k \leq \ell - i_2$;\\
		2 - 2 = 0	& if $\ell - i_2 < k < \frac{\ell}{2}$;\\
		1 - 1 = 0	& if $k = \frac{\ell}{2}$
	\end{cases*}
\end{align}
for all $k \in \{1,\ldots,r\}$.
In particular, $\beta^\omega\neq 0$ regardless of $i_1$ and $i_2$.
If $\frac{\ell}{2} \geq i_2 > i_1$, then we have
\begin{align}
	\bar{c}_k - m_k\bar{c}_0 = \begin{cases*}
		2 - 2 = 0	& if $k < i_1$;\\
		3 - 2 = 1	& if $i_1 \leq  k < i_2$;\\
		4 - 2 = 2	& if $i_2 \leq k < \frac{\ell}{2}$;\\
		2 - 1 = 1	& if $k = \frac{\ell}{2}$
	\end{cases*}
\end{align}
for all $k \in \{1,\ldots,r\}$.
In particular, $\beta^\omega\neq 0$ regardless of $i_1$ and $i_2$.

Suppose that $\beta$ satisfies \cref{Cc}.
If $i_0 > \frac{\ell}{2}$, then we have
\begin{align}
	\bar{c}_k - m_k\bar{c}_0 = \begin{cases*}
		2 - 2 = 0	& if $k \leq \ell - i_0$;\\
		0 - 2 = -2	& if $\ell - i_0 < k < \frac{\ell}{2}$;\\
		0 - 1 = -1	& if $k = \frac{\ell}{2}$
	\end{cases*}
\end{align}
for all $k \in \{1,\ldots,r\}$.
In particular, $\beta^\omega\neq 0$ regardless of $i_0$.
If $i_0 \leq \frac{\ell}{2}$, then we have
\begin{align}
	\bar{c}_k - m_k\bar{c}_0 = \begin{cases*}
		2 - 2 = 0	& if $k < i_0$;\\
		4 - 2 = 2	& if $i_0 \leq k < \frac{\ell}{2}$;\\
		2 - 1 = 1	& if $k = \frac{\ell}{2}$
	\end{cases*}
\end{align}
for all $k \in \{1,\ldots,r\}$.
In particular, $\beta^\omega\neq 0$ regardless of $i_0$.

Hence \cref{claim} holds true when $\PHI$ is of type $C_\ell$.

\label{ZC}
\begin{proposition}\label{disapC}
	Let $\PHI$ be a root system of type $C_\ell$.
	Then
	\begin{align}
		\PHI_{\omega,0}^+ 
		& = \bigset{e_{p_1} + e_{\ell-p_1+1}}{p_1 \in \bigl\{1,\ldots, \lfloor \tfrac{\ell+1}{2} \rfloor\bigr\}}\\
		&= \Bigset{\alpha_{i_1} + \cdots + \alpha_{\ell-i_1} + 2\alpha_{\ell-i_1+1} + \cdots + 2\alpha_{\ell-1} + \alpha_\ell}{i_1 \in \bigl\{1,\ldots, \lfloor \tfrac{\ell+1}{2} \rfloor\bigr\}}.
	\end{align}
	Hence
	\begin{align}
		P(0) = \{\pm 1\}.
	\end{align}
\end{proposition}
\begin{proof}
	It follows from the above observation.
\end{proof}

By \cref{typetheorem}, if $\ell$ is odd, then $\PHI_\re^\omega$ is of type $BC_{\frac{\ell-1}{2}}$ ($r = \frac{\ell-1}{2}$).
If $\ell$ is even, then $\PHI_\re^\omega$ is of type $C_{\frac{\ell}{2}}$ ($r = \frac{\ell}{2}$).

\begin{proposition}\label[prop]{extraC}
	Let $\PHI$ be a root system of type $C_\ell$.
	For any $\beta^\omega \in (\PHI_\re^\omega)^+$, 
	\begin{align}
		P(\beta^\omega) = \{0,\pm1\}.
	\end{align}
\end{proposition}
\begin{proof}
	Let $\beta' \in \PHI^+$ be a $0$-extra root of $\beta^\omega$.
	Since $(\beta',\pi_0^j) = 0$, there exists $i_0,i_1 \in \{1,\ldots,\ell-1\}$ such that $i_0 \leq i_1$ and $\beta' = \alpha_{i_0} + \alpha_{i_0+1} + \cdots + \alpha_{i_1}$.
	Define $\gamma^+, \gamma^- \in \PHI_{\omega,0}^+$ by
	\begin{align}
		\gamma^+ &\ceq \begin{cases*}
			\alpha_{i_1 + 1} + \cdots + \alpha_{\ell-i_1-1} + 2\alpha_{\ell-i_1} + \cdots + 2\alpha_{\ell-1} + \alpha_\ell	& if $i_1 < \lfloor \frac{\ell+1}{2} \rfloor$;\\
			\alpha_{\ell - i_1} + \cdots + \alpha_{i_1} + 2\alpha_{i_1 + 1} + \cdots + 2\alpha_{\ell-1} + \alpha_\ell	& if $i_1 \geq \lfloor \frac{\ell+1}{2} \rfloor$,
		\end{cases*}\\
		\gamma^- &\ceq \begin{cases*}
			\alpha_{i_0} + \cdots + \alpha_{\ell-i_0} + 2\alpha_{\ell-i_0+1} + \cdots + 2\alpha_{\ell-1} + \alpha_\ell	& if $i_0 \leq \lfloor \frac{\ell+1}{2} \rfloor$;\\
			\alpha_{\ell - i_0+1} + \cdots + \alpha_{i_0-1} + 2\alpha_{i_0} + \cdots + 2\alpha_{\ell-1} + \alpha_\ell	& if $i_0 > \lfloor \frac{\ell+1}{2} \rfloor$.
		\end{cases*}
	\end{align}
	Then $\beta' + \gamma^+$ and $\beta' - \gamma^-$ belong to $\PHI$.
	Since $(\beta' + \gamma^+)^\omega = (\beta' - \gamma^-)^\omega = (\beta')^\omega = \beta^\omega$, the root $\beta' + \gamma^+$ is a $1$-extra root and $\beta' - \gamma^-$ is a $(-1)$-extra root of $\beta^\omega$.
\end{proof}

\subsection{Type $D_\ell$ ($j = 1$)}\ \label{S4.2.4}

\noindent
Let $\PHI$ be a root system of type $D_\ell$.
Define $\alpha_1,\ldots,\alpha_\ell$ by 
\begin{align}
	\alpha_1 \ceq e_1 - e_2,\quad \alpha_2 \ceq e_2 - e_3,\quad \ldots,\quad \alpha_{\ell-1} \ceq e_{\ell-1} - e_\ell,\quad \alpha_\ell \ceq e_{\ell-1} + e_\ell
\end{align}
for the standard basis $\{e_1,\ldots,e_{\ell}\}$ of $\mathbb{R}^{\ell}$.
Then $\DELTA = \{\alpha_1,\ldots,\alpha_\ell\}$ is a basis of $\PHI$, and
\begin{align}
	\PHI^+ 
	&= \bigset{e_{p_1} - e_{p_2}}{p_1,p_2 \in \{1,\ldots,\ell\},\ p_1 < p_2}\\ 
	&\qquad \sqcup \bigset{e_{p_0} + e_{\ell}}{p_0 \in \{1,\ldots,\ell-1\}}
	 \sqcup \bigset{e_{p_1} + e_{p_2}}{p_1,p_2 \in \{1,\ldots,\ell-1\},\ p_1 < p_2}\\
	&= \bigset{\alpha_{i_1} + \cdots + \alpha_{i_2}}{i_1,i_2 \in \{1,\ldots,\ell-1\},\ i_1 \leq i_2}\\
	&\qquad \sqcup \bigset{\alpha_{i_0} + \cdots + \alpha_{\ell-2} + \alpha_\ell}{i_0 \in \{1,\ldots,\ell-1\}}\\
	&\qquad \sqcup \bigset{\alpha_{i_1} + \cdots + \alpha_{i_2-1} + 2\alpha_{i_2} + \cdots + 2\alpha_{\ell-2} + \alpha_{\ell-1} + \alpha_\ell}{i_1,i_2 \in \{1,\ldots,\ell-1\},\ i_1 < i_2}
\end{align}
by \cite[Plate \rom{4}]{Bourbaki}.

Let $j = 1$ and $\omega = \omega_1$.
Then $\sigma = (0\ 1)(\ell-1\ \ \ell)$ and $r = \ell-2$.
For any $k \in \{1,\ldots,r\}$, we may assume that 
\begin{align}
	S_0^j = \{0,1\},\qquad
	S_k^j = \{k+1\},\qquad 
	S_r^j = \{\ell-1,\ \ell\}.
\end{align}
Since $\bar{n}_k = 2$ if $k \neq r$, and $\bar{n}_r = 1$, we have $m_k = 1$ for all $k \in \{1,\ldots,r\}$.


We consider \cref{claim}.
Let $\beta = \sum_{i=1}^\ell c_i\alpha_i \in \PHI^+$ satisfy $\bar{c}_0 = c_1 = 1$.
Then $\beta$ is equal to one of the following:
\begin{eenumeratea}
	\item\label{D1a} $\alpha_1 + \cdots + \alpha_{i_2}$ for $i_2 \in \{1,\ldots,\ell-1\}$;
	\item\label{D1b} $\alpha_1 + \cdots + \alpha_{\ell-2} + \alpha_\ell$;
	\item\label{D1c} $\alpha_1 + \cdots + \alpha_{i_2-1} + 2\alpha_{i_2} + \cdots + 2\alpha_{\ell-2} + \alpha_{\ell-1} + \alpha_\ell$ for $i_2 \in \{2,\ldots,\ell-1\}$.
\end{eenumeratea}

In the case of \cref{D1a}, we have
\begin{align}
	\bar{c}_k - m_k\bar{c}_0 = \begin{cases*}
		1 - 1 = 0	& if $k < i_2$;\\
		0 - 1 = -1	& if $k \geq i_2$
	\end{cases*}
\end{align}
for all $k \in \{1,\ldots,r\}$.
In particular, if $i_2 = \ell-1$, then $\beta^\omega = 0$.

In the case of \cref{D1b}, we have
\begin{align}
	\bar{c}_k - m_k\bar{c}_0 = 1 - 1 = 0
\end{align}
for all $k \in \{1,\ldots,r\}$.
Hence $\beta^\omega = 0$.

In the case of \cref{D1c}, we have
\begin{align}
	\bar{c}_k - m_k\bar{c}_0 = \begin{cases*}
		1 - 1 = 0	& if $k < i_2-1$;\\
		2 - 1 = 1	& if $k \geq i_2-1$
	\end{cases*}
\end{align}
for all $k \in \{1,\ldots,r\}$.
In particular, $\beta^\omega \neq 0$ regardless of $i_2$.

Hence \cref{claim} holds true when $\PHI$ is of type $D_\ell$ and $j = 1$.

\label{ZD1}
\begin{proposition}\label{disapD1}
	Let $\PHI$ be a root system of type $D_\ell$, and $j=1$.
	Then
	\begin{align}
		\PHI_{\omega,0}^+ = \{ e_1 - e_\ell,\ e_1 + e_\ell \} = \{\alpha_1 + \cdots + \alpha_{\ell-1},\quad 
		\alpha_1 + \cdots + \alpha_{\ell-2} + \alpha_\ell\}.
	\end{align}
	Hence 
	\begin{align}
		P(0) = \{\pm1\}.
	\end{align}
\end{proposition}
\begin{proof}
	It follows from the above observation.
\end{proof}

In this case, the diagram $\mathcal{D}(\DELTA^\omega)$ is of type $B_{\ell-2}$.

\begin{proposition}\label{reducedD1}
	Let $\PHI$ be a root system of type $D_\ell$, and $j = 1$.
	Then $\PHI_\re^\omega$ is reduced.
\end{proposition}
\begin{proof}
	We can see that $\bar\alpha_r^\omega$ is a short simple root of $\PHI_\re^\omega$.
	Assume that $2\bar\alpha_r^\omega \in \PHI_\re^\omega$.
	Let $\beta = \sum_{i=1}^\ell c_i\alpha_i \in \PHI$ satisfy $\beta^\omega = 2\bar\alpha_r^\omega$.
	According to the above observation, $\bar{c}_0$ must be $0$ since $\bar{c}_r - m_r\bar{c}_0 = 2$.
	Therefore we have
	\begin{align}
		\bar{c}_k = \bar{c}_k - m_k\bar{c}_0 = \begin{cases*}
			0 & if $k \neq r$;\\
			2 & if $k = r$.
		\end{cases*}
	\end{align}
	However, there exists no root $\beta$ that satisfies this condition, which is a contradiction.
	Hence $\PHI_\re^\omega$ is reduced in this case.
\end{proof}

Hence $\PHI_\re^\omega$ is of type $B_{\ell-2}$ ($r = \ell-2$).

\begin{proposition}\label[prop]{extraD1}
	Let $\PHI$ be a root system of type $D_\ell$, and $j = 1$.
	For $\beta^\omega \in (\PHI_\re^\omega)^+$, we have
	\begin{align}
		P(\beta^\omega) = \begin{cases*}
			\{0,\pm1\}	& if $\beta^\omega$ is a short root of $\PHI_\re^\omega$;\\
			\{0\}		& if $\beta^\omega$ is a long root of $\PHI_\re^\omega$.
		\end{cases*}
	\end{align}
\end{proposition}
\begin{proof}
	Suppose that $\beta^\omega \in (\PHI_\re^\omega)^+$ is a short root of $\PHI_\re^\omega$, that is, 
	\begin{align}
		\beta^\omega = \bar\alpha_{k_0}^\omega + \cdots + \bar\alpha_r^\omega
	\end{align}
	for some $k_0 \in \{1,\ldots,r\}$.
	Then $\beta' \ceq \alpha_{k_0+1} + \cdots + \alpha_{\ell-1}$ is a $0$-extra root of $\beta^\omega$.
	Define $\gamma^+,\gamma^- \in \PHI_{\omega,0}^+$ by
	\begin{align}
		\gamma^+ \ceq \alpha_1 + \cdots + \alpha_{\ell-2} + \alpha_\ell,\quad
		\gamma^- \ceq \alpha_1 + \cdots + \alpha_{\ell-2} + \alpha_{\ell-1}.
	\end{align}
	Then $\beta' + \gamma^+$ and $\beta' - \gamma^-$ belong to $\PHI$.
	Since $(\beta' + \gamma^+)^\omega = (\beta' - \gamma^-)^\omega = (\beta')^\omega = \beta^\omega$, the root $\beta' + \gamma^+$ is a $1$-extra root and $\beta' - \gamma^-$ is a $(-1)$-extra root of $\beta^\omega$.
	
	Let $\beta^\omega \in (\PHI_\re^\omega)^+$.
	Suppose that $\beta' = \sum_{i=1}^\ell c_i'\alpha_i \in \PHI$ be a $1$-extra root of $\beta^\omega$, that is, $c_1' = 1$.
	Since $(\beta')^\omega \in (\PHI_\re^\omega)^+$, the above observation implies that $\beta'$ must satisfy \cref{D1c}.
	Hence 
	\begin{align}
		\beta^\omega = (\beta')^\omega = \bar\alpha_{k_0}^\omega + \cdots + \bar\alpha_r^\omega
	\end{align}
	for some $k_0 \in \{1,\ldots,r\}$.
	Therefore $\beta^\omega$ is a short root of $\PHI_\re^\omega$.
	
	Suppose that $\beta' = \sum_{i=1}^\ell c_i'\alpha_i \in \PHI$ be a $(-1)$-extra root of $\beta^\omega$, that is, $c_1' = -1$.
	Since $(\beta')^\omega \in (\PHI_\re^\omega)^+$, $\beta'$ must satisfy \cref{D1a}.
	Hence
	\begin{align}
		\beta^\omega = (\beta')^\omega = \bar\alpha_{k_0}^\omega + \cdots + \bar\alpha_r^\omega
	\end{align}
	for some $k_0 \in \{1,\ldots,r\}$.
	Therefore $\beta^\omega$ is a short root of $\PHI_\re^\omega$.
\end{proof}

\subsection{Type $D_\ell$ ($\ell$ : even, $j \in \{\ell-1,\ell\}$)}\ 

\noindent
Let $\ell \in 2\mathbb{Z}$ and $\PHI$ be a root system of type $D_\ell$.
Suppose that $j \in \{\ell-1,\ell\}$.
We can assume that $j = \ell$ without loss of generality.
Then $\sigma = (0\ \ell)(1\ \ \ell-1)\prod_{i=2}^{\frac{\ell}{2}-1}(i\ \ \ell-i)$ and $r = \frac{\ell}{2}$.
For any $k \in \{0,\ldots,r\}$, we may assume that 
\begin{align}
	S_k^j = \{k,\ \ell - k\}.
\end{align}
Since $\bar{n}_k = 2$ if $k \neq 1$, and $\bar{n}_1 = 1$, we have 
\begin{align}
	m_k = \begin{cases*}
		1 & if $k \in \{1,r\}$;\\
		2 & otherwise.
	\end{cases*}
\end{align}


We consider \cref{claim}.
Let $\beta = \sum_{i=1}^\ell c_i\alpha_i \in \PHI^+$ satisfy $\bar{c}_0 = c_\ell = 1$.
The $\beta$ is equal to either of the following:
\begin{eenumeratea}
	\item\label{Dea} $\alpha_{i_0} + \cdots + \alpha_{\ell-2} + \alpha_\ell$ for $i_0 \in \{1,\ldots,\ell-1\}$;
	\item\label{Deb} $\alpha_{i_1} + \cdots + \alpha_{i_2-1} + 2\alpha_{i_2} + \cdots + 2\alpha_{\ell-2} + \alpha_{\ell-1} + \alpha_\ell$ for $i_1,i_2 \in \{1,\ldots,\ell-1\}$ and $i_1 < i_2$.
\end{eenumeratea}

Suppose that $\beta$ satisfies \cref{Dea}.
If $i_0 > r = \frac{\ell}{2}$, then 
\begin{align}
	\bar{c}_k - m_k\bar{c}_0 = \begin{cases*}
		0 - 1 = -1	& if $k = 1$;\\
		1 - 2 = -1	& if $1 < k \leq \ell - i_0$;\\
		0 - 2 = -2	& if $\ell - i_0 < k < r$;\\
		0 - 1 = -1  & if $k = r$
	\end{cases*}
	\label{Dea1}
\end{align}
for all $k \in \{1,\ldots,r\}$.
In particular, $\beta^\omega \neq 0$ regardless of $i_0$.
If $i_0 \leq \frac{\ell}{2}$, then 
\begin{align}
	\bar{c}_k - m_k\bar{c}_0 = \begin{cases*}
		1 - 1 = 0	& if $k = 1 = i_0$;\\
		0 - 1 = -1	& if $k = 1 < i_0$;\\
		1 - 2 = -1	& if $1 < k < i_0$;\\
		2 - 2 = 0	& if $i_0 \leq k < r$;\\
		1 - 1 = 0	& if $k = r$
	\end{cases*}
	\label{Dea2}
\end{align}
for all $k \in \{1,\ldots,r\}$.
In particular, if $i_0 = 1$, then $\beta^\omega = 0$.

Suppose that $\beta$ satisfies \cref{Deb}.
If $i_2 > i_1 > r = \frac{\ell}{2}$, then we have
\begin{align}
	\bar{c}_k - m_k\bar{c}_0 = \begin{cases*}
		1 - 1 = 0	& if $k = 1$;\\
		2 - 2 = 0	& if $1 < k \leq \ell - i_2$;\\
		1 - 2 = -1	& if $\ell - i_2 < k \leq \ell - i_1$;\\
		0 - 2 = -2	& if $\ell - i_1 < k < r$;\\
		0 - 1 = -1	& if $k = r$
	\end{cases*}
	\label{Deb1}
\end{align}
for all $k \in \{1,\ldots,r\}$.
In particular, $\beta^\omega\neq 0$ regardless of $i_1$ and $i_2$.
If $i_2 > \frac{\ell}{2} \geq i_1$ and $\ell - i_2 < i_1$, then we have
\begin{align}
	\bar{c}_k - m_k\bar{c}_0 = \begin{cases*}
		1 - 1 = 0	& if $k = 1$;\\
		2 - 2 = 0	& if $1 < k \leq \ell - i_2$;\\
		1 - 2 = -1	& if $\ell - i_2 < k < i_1$;\\
		2 - 2 = 0	& if $i_1 \leq k < r$;\\
		1 - 1 = 0	& if $k = r$
	\end{cases*}
	\label{Deb2}
\end{align}
for all $k \in \{1,\ldots,r\}$.
In particular, if $i_1 + i_2 = \ell + 1$, then $\beta^\omega = 0$.
If $i_2 > \frac{\ell}{2} \geq i_1$ and $\ell - i_2 \geq i_1$, then we have
\begin{align}
	\bar{c}_k - m_k\bar{c}_0 = \begin{cases*}
		2 - 1 = 1	& if $k = 1 = i_1$;\\
		1 - 1 = 0	& if $k = 1 < i_1$;\\
		2 - 2 = 0	& if $1 < k < i_1$;\\
		3 - 2 = 1	& if $1 < k$ and $i_1 \leq  k \leq \ell - i_2$;\\
		2 - 2 = 0	& if $\ell - i_2 < k < r$;\\
		1 - 1 = 0	& if $k = r$
	\end{cases*}
	\label{Deb3}
\end{align}
for all $k \in \{1,\ldots,r\}$.
In particular, $\beta^\omega\neq 0$ regardless of $i_1$ and $i_2$.
If $\frac{\ell}{2} \geq i_2 > i_1$, then we have
\begin{align}
	\bar{c}_k - m_k\bar{c}_0 = \begin{cases*}
		2 - 1 = 1	& if $k = 1 = i_1$;\\
		1 - 1 = 0	& if $k = 1 < i_1$;\\
		2 - 2 = 0	& if $1 < k < i_1$;\\
		3 - 2 = 1	& if $1 < k$ and $i_1 \leq  k < i_2$;\\
		4 - 2 = 2	& if $i_2 \leq k < r$;\\
		2 - 1 = 1	& if $k = r$
	\end{cases*}
	\label{Deb4}
\end{align}
for all $k \in \{1,\ldots,r\}$.
In particular, $\beta^\omega\neq 0$ regardless of $i_1$ and $i_2$.

Hence \cref{claim} holds true when $\ell \in 2\mathbb{Z}$, $\PHI$ is of type $D_\ell$ and $j \in \{\ell-1,\, \ell\}$.

\label{ZDe}
\begin{proposition}
	Let $\PHI$ be a root system of type $D_\ell$, $\ell$ is even, and $j \in \{\ell-1,\, \ell\}$.
	Then
	\begin{align}
		\PHI_{\omega,0}^+ 
		& = 
		\begin{cases*}
			\{e_1 - e_{\ell}\} \sqcup \bigset{e_{i} + e_{\ell-i+1}}{i \in \{2,\ldots,\tfrac{\ell}{2}\}} & if $j = \ell-1$;\\
			\bigset{e_{i} + e_{\ell-i+1}}{i \in \{1,\ldots,\tfrac{\ell}{2}\}} & if $j = \ell$.
		\end{cases*}\\
		&= \{\alpha_1 + \cdots + \alpha_{\ell-2} + \alpha_j\} \sqcup \bigset{\alpha_{i_1} + \cdots + \alpha_{\ell-i_1} + 2\alpha_{\ell-i_1+1} + \cdots + 2\alpha_{\ell-2} + \alpha_{\ell-1} + \alpha_\ell}{i_1 \in \{2,\ldots,\tfrac{\ell}{2}\}}.
	\end{align}
	Hence
	\begin{align}
		P(0) = \{\pm1\}.
	\end{align}
\end{proposition}
\begin{proof}
	It follows from the above observation.
\end{proof}

If $\ell = 4$, then the diagram $\mathcal{D}(\DELTA^\omega)$ is of type $B_2$.

\begin{proposition}
	Let $\PHI$ be a root system of type $D_4$, and $j \in \{\ell-1,\,\ell\}$.
	Then $\PHI_\re^\omega$ is reduced.
\end{proposition}
\begin{proof}
	Without loss of generality, we can assume that $j = \ell = 4$.
	We can see that $\bar\alpha_1^\omega$ is a short simple root of $\PHI_\re^\omega$.
	Assume that $2\bar\alpha_1^\omega \in \PHI_\re^\omega$.
	Let $\beta = \sum_{i=1}^4 c_i\alpha_i \in \PHI$ satisfy $\beta^\omega = 2\bar\alpha_1^\omega$.
	Then, since $m_1 = m_2 = 1$, we have
	\begin{align}
		\bar{c}_2 - \bar{c}_0 = 0,\quad \bar{c}_1 - \bar{c}_0 = 2.
	\end{align}
	Since $\bar{c}_1 \leq 2$, we can see that $\bar{c}_1 = 2$ and $\bar{c}_2 = \bar{c}_0 = 0$.
	Thus $\beta = \alpha_1 + \alpha_3$.
	However, the fact that $\alpha_1 + \alpha_3$ is a root of $\PHI$ contradicts \cref{prop4.2}.
	Therefore $\PHI_\re^\omega$ is reduced in this case.
\end{proof}

Hence $\PHI_\re^\omega$ is of type $B_{2}$ if $\ell = 4$ ($r = 2$).
If $\ell \geq 6$, then $\PHI_\re^\omega$ is of type $C_{\frac{\ell}{2}}$ ($r = \frac{\ell}{2}$).

\begin{proposition}\label[prop]{extraDe}
	Let $\PHI$ be a root system of type $D_\ell$, $\ell \geq 4$ is even, and $j \in \{\ell-1,\, \ell\}$.
	For $\beta^\omega \in (\PHI_\re^\omega)^+$, we have
	\begin{align}
		P(\beta^\omega) = \begin{cases*}
			\{0,\pm1\}	& if $\beta^\omega$ is a short root of $\PHI_\re^\omega$;\\
			\{0\}		& if $\beta^\omega$ is a long root of $\PHI_\re^\omega$.
		\end{cases*}
	\end{align}
\end{proposition}
\begin{proof}
	Without loss of generality, we can assume that $j = \ell$.
	Let  $\beta^\omega \in (\PHI_\re^\omega)^+$ be a short root of $\PHI_\re^\omega$.
	Suppose that
	\begin{align}
		\beta^\omega = \bar\alpha^\omega_{k_0} + \bar\alpha^\omega_{k_0+1} + \cdots + \bar\alpha^\omega_{k_1}
	\end{align}
	for some $k_0,k_1 \in \{1,\ldots,r-1\}$.
	Then $\beta' \ceq \alpha_{k_0} + \cdots + \alpha_{k_1}$ is a $0$-extra root of $\beta^\omega$.
	Define $\gamma^+, \gamma^- \in \PHI_{\omega,0}^+$ by
	\begin{align}
		\gamma^+ &\ceq 
			\alpha_{k_1 + 1} + \cdots + \alpha_{\ell-k_1-1} + 2\alpha_{\ell-k_1} + \cdots + 2\alpha_{\ell-1} + \alpha_\ell,\\
		\gamma^- &\ceq 
			\alpha_{k_0} + \cdots + \alpha_{\ell-k_0} + 2\alpha_{\ell-k_0+1} + \cdots + 2\alpha_{\ell-1} + \alpha_\ell.
	\end{align}
	Then $\beta' + \gamma^+$ and $\beta' - \gamma^-$ belong to $\PHI$.
	Since $(\beta' + \gamma^+)^\omega = (\beta' - \gamma^-)^\omega = (\beta')^\omega = \beta^\omega$, the root $\beta' + \gamma^+$ is a $1$-extra root and $\beta' - \gamma^-$ is a $(-1)$-extra root of $\beta^\omega$.

	Suppose that 
	\begin{align}
		\beta^\omega = \bar\alpha_{k_0}^\omega + \cdots + \bar\alpha^\omega_{k_1} + 2\bar\alpha^\omega_{k_1+1} + \cdots + 2\bar\alpha^\omega_{r-1} + \bar\alpha_r^\omega
	\end{align}
	for some $k_0, k_1 \in \{1,\ldots,r\}$ satisfying $k_0 \leq k_1$.
	Then $\beta' \ceq \alpha_{k_0} + \cdots + \alpha_{\ell - k_1 - 1}$ is a $0$-extra root of $\beta^\omega$.
	Define $\gamma^+, \gamma^- \in \PHI_{\omega,0}^+$ by
	\begin{align}
		\gamma^+ &\ceq \begin{cases*}
			\alpha_{r} + 2\alpha_{r+1} + \cdots + 2\alpha_{\ell-1} + \alpha_\ell	& if $k_1 = r = \frac{\ell}{2}$;\\
			\alpha_{k_1+1} + \cdots + \alpha_{\ell - k_1 - 1} + 2\alpha_{\ell - k_1} + \cdots + 2\alpha_{\ell-1} + \alpha_\ell	& if $k_1 < \frac{\ell}{2}$,
		\end{cases*}\\
		\gamma^- &\ceq 
			\alpha_{k_0} + \cdots + \alpha_{\ell-k_0} + 2\alpha_{\ell-k_0+1} + \cdots + 2\alpha_{\ell-1} + \alpha_\ell.
	\end{align}
	Then $\beta' + \gamma^+$ and $\beta' - \gamma^-$ belong to $\PHI$.
	Since $(\beta' + \gamma^+)^\omega = (\beta' - \gamma^-)^\omega = (\beta')^\omega = \beta^\omega$, the root $\beta' + \gamma^+$ is a $1$-extra root and $\beta' - \gamma^-$ is a $(-1)$-extra root of $\beta^\omega$.

	Let $\beta^\omega \in (\PHI_\re^\omega)^+$.
	Suppose that $\beta' = \sum_{i=1}^\ell c_i'\alpha_i \in \PHI$ be a $1$-extra root of $\beta^\omega$, that is, $c_\ell' = 1$.
	Then $(\beta')^\omega$ satisfies either \cref{Deb3} or \cref{Deb4} since $(\beta')^\omega \in (\PHI_\re^\omega)^+$.
	Hence $\beta^\omega = (\beta')^\omega$ is equal to either of the following:
	\begin{align}
		\bar\alpha^\omega_{i_1} + \cdots + \bar\alpha^\omega_{i_2 - 1} + 2\bar\alpha^\omega_{i_2} + \cdots + 2\bar\alpha^\omega_{r-1} + \bar\alpha^\omega_r,\qquad
		\bar\alpha^\omega_{i_1} + \cdots + \bar\alpha^\omega_{\ell-i_2}
	\end{align}
	Both of these are short roots of $\PHI_\re^\omega$.
	
	Suppose that $\beta' = \sum_{i=1}^\ell c_i'\alpha_i \in \PHI$ be a $(-1)$-extra root of $\beta^\omega$, that is, $c_\ell' = -1$.
	Then $(\beta')^\omega$ satisfies either \cref{Dea1}, \cref{Dea2}, \cref{Deb1}, or \cref{Deb2} since $(\beta')^\omega \in (\PHI_\re^\omega)^+$.
	In either case, we can see that $\beta^\omega = (\beta')^\omega$ is a short root of $\PHI_\re^\omega$.
\end{proof}

\subsection{Type $D_\ell$ ($\ell$ : odd, $j \in \{\ell-1,\ell\}$)}\ 

\noindent
Let $\ell \not\in 2\mathbb{Z}$ and $\PHI$ be a root system of type $D_\ell$.
Suppose that $j \in \{\ell-1,\ell\}$.
We can assume that $j = \ell$ without loss of generality.
Then $\sigma = (0\ \ell\ 1\ \ell-1)\prod_{i=2}^{\frac{\ell-1}{2}}$ and $r = \frac{\ell-3}{2}$.
For any $k \in \{1,\ldots,r\}$, we may assume that
\begin{align}
	S_0^j = \{0,\, 1,\, \ell-1,\, \ell\},\qquad 
	S_k^j = \{k+1,\, \ell-k-1\}.
\end{align}
For any $k \in \{1,\ldots,r\}$, since $\bar{n}_k = 2$, we have $m_k = 1$.

We consider \cref{claim}.
\begin{eenumerateA}
	\item\label{DoA} Let $\beta = \sum_{i=1}^\ell c_i\alpha_i \in \PHI^+$ satisfy $\bar{c}_0 = 1$.
	Without loss of generality, we can assume that $c_1 = 1$ and $c_{\ell-1} = c_\ell = 0$.
	Then $\beta = \alpha_1 + \cdots + \alpha_{i_2}$ for some $i_2 \in \{1,\ldots,\ell-2\}$.
	If $i_2 > \frac{\ell}{2}$, then we have
	\begin{align}
		\bar{c}_k - m_k\bar{c}_0 = \begin{cases*}
			1 - 1 = 0	& if $k < \ell - i_2 - 1$;\\
			2 - 1 = 1	& if $k \geq \ell - i_2 - 1$
		\end{cases*}
		\label{Do1}
	\end{align}
	for all $k \in \{1,\ldots,r\}$.
	In particular, $\beta^\omega\neq 0$ regardless of $i_2$.
	If $i_2 < \frac{\ell}{2}$, then 
	\begin{align}
		\bar{c}_k - m_k\bar{c}_0 = \begin{cases*}
			1 - 1 = 0	& if $k \leq i_2$;\\
			0 - 1 = -1	& if $k > i_2$
		\end{cases*}
		\label{Do2}
	\end{align}
	for all $k \in \{1,\ldots,r\}$.
	In particular, if $i_2 = \frac{\ell-3}{2} = r$, then $\beta^\omega = 0$.
	
	
	\item\label{DoB} Let $\beta = \sum_{i=1}^\ell c_i\alpha_i \in \PHI^+$ satisfy $\bar{c}_0 = 2$.
	Then $\beta$ is equal to one of the following:
	\begin{eenumeratea}
		\item\label{DoBa} $\alpha_1 + \cdots + \alpha_{\ell-2} + \alpha_{\ell-1}$;
		\item\label{DoBb} $\alpha_1 + \cdots + \alpha_{\ell-2} + \alpha_{\ell}$;
		\item\label{DoBc} $\alpha_{i_1} + \cdots + \alpha_{i_2-1} + 2\alpha_{i_2} + \cdots + 2\alpha_{\ell-2} + \alpha_{\ell-1} + \alpha_\ell$ for $i_1,i_2 \in \{2,\ldots,\ell-1\}$ and $i_1 < i_2$.
	\end{eenumeratea}
	
	In the case of \cref{DoBa} and \cref{DoBb}, we have
	\begin{align}
		\bar{c}_k - m_k\bar{c}_0 = 2 - 2 = 0
	\end{align}
	for all $k \in \{1,\ldots,r\}$.
	hence $\beta^\omega = 0$.
	
	Suppose that $\beta$ satisfies \cref{DoBc}.
	If $i_2 > i_1 > \frac{\ell}{2}$, then we have
	\begin{align}
		\bar{c}_k - m_k\bar{c}_0 = \begin{cases*}
			2 - 2 = 0	& if $k \leq \ell - i_2$;\\
			1 - 2 = -1	& if $\ell - i_2 < k \leq \ell - i_1$;\\
			0 - 2 = -2	& if $k > \ell - i_1$
		\end{cases*}
		\label{Do3}
	\end{align}
	for all $k \in \{1,\ldots,r\}$.
	In particular, $\beta^\omega\neq 0$ regardless of $i_1$ and $i_2$.
	If $i_2 > \frac{\ell}{2} > i_1$ and $\ell - i_2 < i_1$, then we have
	\begin{align}
		\bar{c}_k - m_k\bar{c}_0 = \begin{cases*}
			2 - 2 = 0	& if $k \leq \ell - i_2$;\\
			1 - 2 = -1	& if $\ell - i_2 < k < i_1$;\\
			2 - 2 = 0	& if $k \geq i_1$
		\end{cases*}
		\label{Do4}
	\end{align}
	for all $k \in \{1,\ldots,r\}$.
	In particular, if $i_1 + i_2 = \ell + 1$, then $\beta^\omega = 0$.
	If $i_2 > \frac{\ell}{2} > i_1$ and $\ell - i_2 \geq i_1$, then we have
	\begin{align}
		\bar{c}_k - m_k\bar{c}_0 = \begin{cases*}
			2 - 2 = 0	& if $k < i_1$;\\
			3 - 2 = 1	& if $i_1 \leq  k \leq \ell - i_2$;\\
			2 - 2 = 0	& if $k > \ell - i_2$
		\end{cases*}
		\label{Do5}
	\end{align}
	for all $k \in \{1,\ldots,r\}$.
	In particular, $\beta^\omega\neq 0$ regardless of $i_1$ and $i_2$.
	If $\frac{\ell}{2} \geq i_2 > i_1$, then we have
	\begin{align}
		\bar{c}_k - m_k\bar{c}_0 = \begin{cases*}
			2 - 2 = 0	& if $k < i_1$;\\
			3 - 2 = 1	& if $i_1 \leq  k < i_2$;\\
			4 - 2 = 2	& if $k \geq i_2$
		\end{cases*}
		\label{Do6}
	\end{align}
	for all $k \in \{1,\ldots,r\}$.
	In particular, $\beta^\omega\neq 0$ regardless of $i_1$ and $i_2$.


	\item\label{DoC} Let $\beta = \sum_{i=1}^\ell c_i\alpha_i \in \PHI^+$ satisfy $\bar{c}_0 = 3$, that is, $c_1 = c_{\ell-1} = c_\ell = 1$.
	Then 
	\begin{align}
		\beta = \alpha_1 + \cdots + \alpha_{i_2 - 1} + 2\alpha_{i_2} + \cdots + 2\alpha_{\ell-2} + \alpha_{\ell-1} + \alpha_\ell
	\end{align}
	for some $i_2 \in \{2,\ldots,\ell-1\}$.
	If $i_2 > \frac{\ell}{2}$, then we have
	\begin{align}
		\bar{c}_k - m_k\bar{c}_0 = \begin{cases*}
			3 - 3 = 0	& if $k \leq \ell - i_2 - 1$;\\
			2 - 3 = -1	& if $k > \ell - i_2 - 1$
		\end{cases*}
		\label{Do7}
	\end{align}
	for all $k \in \{1,\ldots,r\}$.
	In particular, if $i_2 = \frac{\ell+1}{2}$, then $\beta^\omega = 0$.
	If $i_2 < \frac{\ell}{2}$, then we have
	\begin{align}
		\bar{c}_k - m_k\bar{c}_0 = \begin{cases*}
			3 - 3 = 0	& if $k < i_2$;\\
			4 - 3 = 1	& if $i_2 \leq k$
		\end{cases*}
		\label{Do8}
	\end{align}
	for all $k \in \{1,\ldots,r\}$.
	In particular, $\beta^\omega\neq 0$ regardless of $i_1$ and $i_2$.


\end{eenumerateA}

Hence \cref{claim} holds true when $\ell \not\in 2\mathbb{Z}$, $\PHI$ is of type $D_\ell$ and $j \in \{\ell-1,\ell\}$.

\label{ZDo}
\begin{proposition}\label{disapDo}
	Let $\PHI$ be a root system of type $D_\ell$, $\ell$ is odd, and $j \in \{\ell-1,\, \ell\}$.
	Then
	\begin{align}
		\PHI_{\omega,0}^+ 
		&= \left\{ e_1 \pm e_{\frac{\ell+1}{2}},\ e_{\frac{\ell+1}{2}} \pm e_\ell,\ e_1 \pm e_\ell \right\} \sqcup \bigset{e_{i} + e_{\ell-i+1}}{i \in \{2,\ldots,\tfrac{\ell-1}{2}\}}\\
		&= \left\{\alpha_1 + \cdots + \alpha_{\frac{\ell-1}{2}}\right\} 
		\sqcup \bigset{\alpha_{\frac{\ell+1}{2}} + \cdots + \alpha_{\ell-2} + \alpha_p}{p \in \{\ell-1,\, \ell\}}\\ 
		&\qquad \sqcup \bigset{\alpha_1 + \cdots + \alpha_{\ell-2} + \alpha_p}{p \in \{\ell-1,\, \ell\}}\\ 
		&\qquad \sqcup\bigset{\alpha_{i_1} + \cdots + \alpha_{\ell-i_1} + 2\alpha_{\ell-i_1+1} + \cdots + 2\alpha_{\ell-2} + \alpha_{\ell-1} + \alpha_\ell}{i_1 \in \{2,\ldots,\tfrac{\ell-1}{2}\}}\\
		&\qquad \sqcup \left\{ \alpha_1 + \cdots + \alpha_{\frac{\ell-1}{2}} + 2\alpha_{\frac{\ell+1}{2}} + \cdots + 2\alpha_{\ell-2} + \alpha_{\ell-1} + \alpha_\ell \right\}.
	\end{align}
	Hence
	\begin{align}
		P(0) = \{\pm1,\pm2,\pm3\}.
	\end{align}
\end{proposition}
\begin{proof}
	It follows from the above observation.
\end{proof}

By \cref{reducedlemma}, $\PHI_\re^\omega$ is of type $BC_{\frac{\ell-3}{2}}$.

\begin{proposition}\label[prop]{extraDo}
	Let $\PHI$ be a root system of type $D_\ell$, $\ell \geq 5$ is odd, and $j \in \{\ell-1,\, \ell\}$.
	For $\beta^\omega \in (\PHI_\re^\omega)^+$, we have
	\begin{align}
		P(\beta^\omega) = \begin{cases*}
			\{0,\pm1,\pm2,\pm3\}	& if $\beta^\omega$ is a short root of $\PHI_\re^\omega$;\\
			\{0,\pm2\}		& if $\beta^\omega$ is either middle root or long root of $\PHI_\re^\omega$.
		\end{cases*}
	\end{align}
\end{proposition}
\begin{proof}
	Let  $\beta^\omega \in (\PHI_\re^\omega)^+$.
	Note that $\beta^\omega$ is a short root if and only if the coefficient of $\bar\alpha^\omega_r$ is equal to $1$.
	\begin{eenumeratea}
		\item Suppose that $\beta^\omega$ is a short root, that is, 
		\begin{align}
			\beta^\omega = \bar\alpha^\omega_{k_0} + \cdots + \bar\alpha^\omega_r
		\end{align}
		for some $k_0 \in \{1,\ldots,r\}$.
		Then 
		\begin{align}
			\beta' &\ceq \alpha_{k_0+1} + \cdots + \alpha_{r+1} = \alpha_{k_0+1} + \cdots + \alpha_{\frac{\ell-1}{2}},\\
			\beta'' &\ceq \alpha_{\ell-r-1} + \cdots + \alpha_{\ell-k_0-1} = \alpha_{\frac{\ell+1}{2}} + \cdots + \alpha_{\ell-k_0-1}
		\end{align}
		are $0$-extra roots of $\beta^\omega$.
		
		Define $\gamma_1^+, \gamma_1^- \in \PHI_{\omega,0}^+$ by
		\begin{align}
			\gamma_1^+ &\ceq 
			\alpha_1 + \cdots + \alpha_{\frac{\ell-1}{2}} = \alpha_1 + \cdots + \alpha_{\ell-r-2},\\
			\gamma_1^- &\ceq 
			\alpha_{\frac{\ell+1}{2}} + \cdots + \alpha_{\ell-1} = \alpha_{r+2} + \cdots + \alpha_{\ell-1}.
		\end{align}
		Then $\beta'' + \gamma_1^+$ and $\beta' - \gamma_1^-$ belong to $\PHI$.
		Since $(\beta'' + \gamma_1^+)^\omega = (\beta' - \gamma_1^-)^\omega = \beta^\omega$, the root $\beta'' + \gamma_1^+$ is a $1$-extra root and $\beta' - \gamma_1^-$ is a $(-1)$-extra root of $\beta^\omega$.
		
		Define $\gamma_2 \in \PHI_{\omega,0}^+$ by
		\begin{align}
			\gamma_2 \ceq \alpha_{k_0 + 1} + \cdots + \alpha_{\ell - k_0 - 1} + 2\alpha_{\ell - k_0} + \cdots + 2\alpha_{\ell-2} + \alpha_{\ell-1} + \alpha_\ell.
		\end{align}
		Then $\beta'' + \gamma_2$ and $\beta' - \gamma_2$ belong to $\PHI$.
		Since $(\beta'' + \gamma_2)^\omega = (\beta' - \gamma_2)^\omega = \beta^\omega$, the root $\beta'' + \gamma_2$ is a $2$-extra root and $\beta' - \gamma_2$ is a $(-2)$-extra root of $\beta^\omega$.
		
		Let 
		\begin{align}
			\gamma_3 \ceq \alpha_1 + \cdots + \alpha_{\frac{\ell-1}{2}} + 2\alpha_{\frac{\ell+1}{2}} + \cdots + 2\alpha_{\ell-2} + \alpha_{\ell-1} + \alpha_\ell \in \PHI_{\omega,0}^+.
		\end{align}
		Then $\beta' + \gamma_3$ and $\beta'' - \gamma_3$ belong to $\PHI$.
		Since $(\beta' + \gamma_3)^\omega = (\beta'' - \gamma_3)^\omega = \beta^\omega$, the root $\beta' + \gamma_3$ is a $3$-extra root and $\beta'' - \gamma_3$ is a $(-3)$-extra root of $\beta^\omega$.
		
		\item Suppose that $\beta^\omega$ is either middle root or long root.
		Then there exist $k_1 \in \{2,\ldots,r+1\}$ and $k_2 \in \{k_1,\,\ldots,\,\ell-k_1\}\setminus\{r+1\}$ such that
		\begin{align}
			\beta' \ceq \alpha_{k_1} + \cdots + \alpha_{k_2}
		\end{align}
		is a $0$-extra root of $\beta^\omega$.
		Note that $\beta^\omega$ is a long root if and only if $k_2 = \ell - k_1$.
		Define $\gamma^+,\gamma^- \in \PHI_{\omega,0}^+$ by
		\begin{align}
			\gamma^+ &\ceq \begin{cases*}
				\alpha_{k_2 + 1} + \cdots + \alpha_{\ell-k_2} + 2\alpha_{\ell-k_2+1} + \cdots + 2\alpha_{\ell-2} + \alpha_{\ell-1} + \alpha_\ell & if $k_2 < r+1 = \frac{\ell-1}{2}$;\\
				\alpha_{\ell - k_2} + \cdots + \alpha_{k_2} + 2\alpha_{k_2 + 1} + \cdots + 2\alpha_{\ell-2} + \alpha_{\ell-1} + \alpha_\ell & if $k_2 > r + 1 = \frac{\ell-1}{2}$,
			\end{cases*}\\
			\gamma^- &\ceq \alpha_{k_1} + \cdots + \alpha_{\ell-k_1} + 2\alpha_{\ell-k_1 + 1} + \cdots + 2\alpha_{\ell-2} + \alpha_{\ell-1} + \alpha_\ell
		\end{align}
		Then $\beta' + \gamma^+$ and $\beta' - \gamma^-$ belong to $\PHI$.
		Since $(\beta' + \gamma^+)^\omega = (\beta' - \gamma^-)^\omega = (\beta')^\omega = \beta^\omega$, the root $\beta' + \gamma^+$ is a $2$-extra root and $\beta' - \gamma^-$ is a $(-2)$-extra root of $\beta^\omega$.
	\end{eenumeratea}

	Let $\beta^\omega \in (\PHI_\re^\omega)^+$.
	Suppose that $\beta' = \sum_{i=1}^\ell c_i'\alpha_i \in \PHI$ be a $1$-extra root of $\beta^\omega$.
	Without loss of generality, we can assume that $c_1' = 1$ and $c_{\ell-1}' = c_\ell' = 0$.
	Then $(\beta')^\omega$ must satisfy \cref{Do1} since $(\beta')^\omega \in (\PHI_\re^\omega)^+$.
	Therefore 
	\begin{align}
		\beta^\omega = (\beta')^\omega = \bar{\alpha}^\omega_{\ell-i_2-1} + \cdots + \bar{\alpha}_r^\omega,
	\end{align}
	and hence $\beta^\omega$ is a short roots of $\PHI_\re^\omega$.
	In the same way, if $\beta^\omega$ has a $(-1)$-extra root, then $\beta^\omega$ is a short root of $\PHI_\re^\omega$.
	
	Suppose that $\beta' = \sum_{i=1}^\ell c_i'\alpha_i \in \PHI$ be a $3$-extra root of $\beta^\omega$, that is, $c_1' = c_{\ell-1}' = c_\ell' = 1$.
	Then $\beta'$ must satisfy \cref{Do8} since $(\beta')^\omega \in (\PHI_\re^\omega)^+$.
	Therefore
	\begin{align}
		\beta^\omega = (\beta')^\omega = \bar{\alpha}^\omega_{i_2} + \cdots + \bar{\alpha}_r^\omega,
	\end{align}
	anad hence $\beta^\omega$ is a short roots of $\PHI_\re^\omega$.
	In the same way, if $\beta^\omega$ has a $(-3)$-extra root, then $\beta^\omega$ is a short root of $\PHI_\re^\omega$.
\end{proof}

\subsection{Type $E_6$}\ 

\noindent
Let $\PHI$ be a root system of type $E_6$.
Then
\begin{align}
	\PHI^+ = \bigset{e_i-e_j,\ e_i+e_j}{i,j \in \{1,\ldots,5\},\ i > j} \sqcup \Bigset{\frac{\nu_1e_1 + \cdots + \nu_5e_5 - e_6 -e_7 + e_8}{2}}{
		\begin{lgathered}
			\nu_1,\ldots,\nu_5 \in \{-1,1\},\\ 
			\text{$\nu_1 \cdots \nu_5 = 1$}
		\end{lgathered}
	}
\end{align}
for the standard basis $\{e_1,\ldots,e_8\}$ of $\mathbb{R}^8$.
The roots
\begin{align}
	\DELTA = \left\{\frac{e_1 - e_2 - e_3 - e_4 - e_5 - e_6 - e_7 + e_8}{2}, \ e_1 + e_2,\ e_2 - e_1,\ e_3 - e_2,\ e_4 - e_3,\ e_5 - e_4 \right\}	
\end{align}
is a basis of $\PHI$ (see \cite[Plate \rom{5}]{Bourbaki}).
The coefficient matrix $C = (c_{ij})_{ij}$ of $\PHI^+$ with respect to $\DELTA$ is as follows:
\begin{align}
	\resizebox{\textwidth}{!}{$
	C = \left(\begin{array}{rrrrrrrrrrrrrrrrrrrrrrrrrrrrrrrrrrrr}
		1 & 0 & 0 & 0 & 0 & 0 & \ \, \ 1 & 0 & 0 & 0 & 0 & \ \,\ 1 & 0 & 0 & 0 & 0 & \ \,\ 1 & 1 & 0 & 0 & 0 & \ \,\ 1 & 1 & 0 & 0 & \ \,\ 1 & 1 & 0 & \ \,\ 1 & 1 & 0 & \ \,\ 1 & 1 & \ \,\ 1 & \ \,\ 1 & \ \,\ 1 \\
		0 & 1 & 0 & 0 & 0 & 0 & 0 & 1 & 0 & 0 & 0 & 0 & 1 & 1 & 0 & 0 & 1 & 0 & 1 & 1 & 0 & 1 & 0 & 1 & 1 & 1 & 1 & 1 & 1 & 1 & 1 & 1 & 1 & 1 & 1 & 2 \\
		0 & 0 & 1 & 0 & 0 & 0 & 1 & 0 & 1 & 0 & 0 & 1 & 1 & 0 & 1 & 0 & 1 & 1 & 1 & 0 & 1 & 1 & 1 & 1 & 1 & 1 & 1 & 1 & 2 & 1 & 1 & 2 & 1 & 2 & 2 & 2 \\
		0 & 0 & 0 & 1 & 0 & 0 & 0 & 1 & 1 & 1 & 0 & 1 & 1 & 1 & 1 & 1 & 1 & 1 & 1 & 1 & 1 & 1 & 1 & 1 & 2 & 2 & 1 & 2 & 2 & 2 & 2 & 2 & 2 & 2 & 3 & 3 \\
		0 & 0 & 0 & 0 & 1 & 0 & 0 & 0 & 0 & 1 & 1 & 0 & 0 & 1 & 1 & 1 & 0 & 1 & 1 & 1 & 1 & 1 & 1 & 1 & 1 & 1 & 1 & 1 & 1 & 1 & 2 & 1 & 2 & 2 & 2 & 2 \\
		0 & 0 & 0 & 0 & 0 & 1 & 0 & 0 & 0 & 0 & 1 & 0 & 0 & 0 & 0 & 1 & 0 & 0 & 0 & 1 & 1 & 0 & 1 & 1 & 0 & 0 & 1 & 1 & 0 & 1 & 1 & 1 & 1 & 1 & 1 & 1
	\end{array}\right).$}
\end{align}
Then 
\begin{align}
	\PHI^+ = \bigset{c_{1j}\alpha_1 + \cdots + c_{6j}\alpha_6}{j \in \{1,\ldots,36\}}.
\end{align}

Let $j \in J\setminus \{0\}$ and $\omega = \omega_j$
Without loss off generality, we can assume that $j = 1$.
Then $\sigma = (0\ 1\ 6)(2\ 3\ 5)$ and $r = 2$.
We may assume that 
\begin{align}
	S_0^j = \{0,1,6\},\qquad
	S_1^j = \{2,3,5\},\qquad
	S_2^j = \{4\}.
\end{align}
Furthermore, we have
\begin{align}
	m_1 = \frac{\#S_1^j \bar{n}_1}{\#S_0^j} = \frac{3 \cdot 2}{3} = 2,\qquad 
	m_2 = \frac{\#S_2^j \bar{n}_2}{\#S_0^j} = \frac{1 \cdot 3}{3} = 1.
\end{align}
Define a $2 \times 36$ matrix $X = (x_{ij})_{ij}$ by 
\begin{align}
	x_{1j} = c_{2j} + c_{3j} + c_{5j} - 2(c_{1j} + c_{6j}),\qquad 
	x_{2j} = c_{4j} - (c_{1j} + c_{6j}),
\end{align}
that is, we have $x_{1j}\bar{\alpha}_1^\omega + x_{2j}\bar{\alpha}_2^\omega = (c_{1j}\alpha_1 + \cdots + c_{6j}\alpha_6)^{\omega}$.
Then the matrix $X$ is as follows:
\begin{align}
	\resizebox{\textwidth}{!}{$
		X = \left(\begin{array}{rrrrrrrrrrrrrrrrrrrrrrrrrrrrrrrrrrrrr}
			-2 & 1 & 1 & 0 & 1 & -2 &\ \,\  -1 & 1 & 1 & 1 & -1 & \ \,\ -1 & 2 & 2 & 2 & -1 & \ \,\ 0 & 0 & 3 & 0 & 0 & \ \,\ 1 & -2 & 1 & 3 & \ \,\  1 & -1 & 1 & \ \,\ 2 & -1 & 2 & \ \,\ 0 & 0 & \ \,\ 1 &\ \,\  1 &\ \,\  2 \\
			-1 & 0 & 0 & 1 & 0 & -1 & -1 & 1 & 1 & 1 & -1 & 0 & 1 & 1 & 1 & 0 & 0 & 0 & 1 & 0 & 0 & 0 & -1 & 0 & 2 & 1 & -1 & 1 & 1 & 0 & 1 & 0 & 0 & 0 & 1 & 1 
		\end{array}\right).$}
\end{align}

Hence \cref{claim} holds true when $\PHI$ is of type $E_6$.

\label{ZE6}
\begin{proposition}\label{disapE6}
	Let $\PHI$ be a root system of type $E_6$.
	Then
	\begin{align}
		\PHI_{\omega,0}^+ 
		&= \left\{ \begin{lgathered}
			\frac{1}{2}(e_1 + e_2 + e_3 - e_4 - e_5 - e_6 - e_7 + e_8),\quad
			\frac{1}{2}(-e_1 - e_2 - e_3 + e_4 - e_5 - e_6 - e_7 + e_8),\quad
			e_5 + e_1,\quad 
			e_5 - e_1,\ \\
			\frac{1}{2}(-e_1 + e_2 + e_3 - e_4 + e_5 - e_6 - e_7 + e_8),\quad
			\frac{1}{2}(e_1 - e_2 - e_3 + e_4 + e_5 - e_6 - e_7 + e_8)
		\end{lgathered}
		\right\}\\
		&= \left\{ \begin{lgathered}
			\alpha_1 + \alpha_2 + \alpha_3 + \alpha_4,\quad
			\alpha_1 + \alpha_3 + \alpha_4 + \alpha_5,\quad
			\alpha_2 + \alpha_4 + \alpha_5 + \alpha_6,\quad
			\alpha_3 + \alpha_4 + \alpha_5 + \alpha_6,\ \\
			\alpha_1 + \alpha_2 + 2\alpha_3 + 2\alpha_4 + \alpha_5 + \alpha_6,\quad 
			\alpha_1 + \alpha_2 + \alpha_3 + 2\alpha_4 + 2\alpha_5 + \alpha_6
		\end{lgathered}\right\}.
	\end{align}
	Hence
	\begin{align}
		P(0) = \{\pm1,\pm2\}.
	\end{align}
\end{proposition}
\begin{proof}
	See matrices $C$ and $X$.
\end{proof}

By \cref{typetheorem}, $\PHI_\re^\omega$ is of type $G_2$.
A root system of type $G_2$ is a set
\begin{align}
	\PHI(G_2) = \Bigl\{
	\pm(e_1 - e_2),\ 
	\pm(-2e_1 + e_2 + e_3),\ 
	\pm(e_3 - e_1),\ 
	\pm(e_3 - e_2).\ 
	\pm(e_1 - 2e_2 + e_3),\ 
	\pm(-e_1 - e_2 + 2e_3)
	\Bigr\}
\end{align}
for the standard basis $\{e_1,e_2,e_3\}$ of $\mathbb{R}^3$.
The roots
\begin{align}
	\DELTA(G_2) = \left\{
	e_1 - e_2,\ -2e_1 + e_2 + e_3
	\right\}	
\end{align}
is a basis of $\PHI(G_2)$.
The coefficient matrix of $\PHI(G_2)^+$ with respect to $\DELTA(G_2)$ is as follows:
\begin{align}
	\left(\begin{array}{cccccc}
		1 & 0 & 1 & 2 & 3 & 3\\
		0 & 1 & 1 & 1 & 1 & 2
	\end{array}\right).
\end{align}

\begin{proposition}\label[prop]{extraE6}
	Let $\PHI$ be a root system of type $E_6$.
	For any $\beta^\omega \in (\PHI_\re^\omega)^+$, we have
	\begin{align}
		P(\beta^\omega) = \begin{cases*}
			\{0,\pm1,\pm2\} & if $\beta^\omega$ is a short root of $\PHI_\re^\omega$;\\
			\{0\}			& if $\beta^\omega$ is a long root of $\PHI_\re^\omega$.
		\end{cases*}
	\end{align}
\end{proposition}
\begin{proof}
	Observing matrices $C$ and $X$, we can see the following holds for any $p \in \{\pm1,\pm2\}$: If $\beta$ is a $p$-extra root, then $\beta^\omega$ is a short root.
	Conversely, every short root $\beta^\omega$ has a $p$-extra root.
\end{proof}

\subsection{Type $E_7$}\ 

\noindent
Let $\PHI$ be a root system of type $E_7$.
Then
\begin{align}
	\PHI^+ &= \bigset{e_i-e_j,\ e_i+e_j}{i,j \in \{1,\ldots,6\},\ i > j} \sqcup \bigl\{ e_8 - e_7 \bigr\} \\
	&\quad \qquad \qquad \sqcup \Bigset{\frac{\nu_1e_1 + \cdots + \nu_6e_6 - e_7 + e_8}{2}}{
		\begin{lgathered}
			\nu_1,\ldots,\nu_6 \in \{-1,1\},\\ 
			\text{$\nu_1 \cdots \nu_6 = -1$}
		\end{lgathered}
	}
\end{align}
for the standard basis $\{e_1,\ldots,e_8\}$ of $\mathbb{R}^8$.
The roots
\begin{align}
	\DELTA = \left\{\frac{e_1 - e_2 - e_3 - e_4 - e_5 - e_6 - e_7 + e_8}{2}, \ e_1 + e_2,\ e_2 - e_1,\ e_3 - e_2,\ e_4 - e_3,\ e_5 - e_4,\ e_6 - e_5 \right\}	
\end{align}
is a basis of $\PHI$ (see \cite[Plate \rom{6}]{Bourbaki}).
The coefficient matrix $C = (c_{ij})_{ij}$ of $\PHI^+$ with respect to $\DELTA$ is as follows:
\begin{align}
	\resizebox{\textwidth}{!}{\begin{minipage}{22cm}
			$C = \left(\begin{array}{cccccc}
				\begin{array}{ccccccc}
					1 & 0 & 0 & 0 & 0 & 0 & 0 \\
					0 & 1 & 0 & 0 & 0 & 0 & 0 \\
					0 & 0 & 1 & 0 & 0 & 0 & 0 \\
					0 & 0 & 0 & 1 & 0 & 0 & 0 \\
					0 & 0 & 0 & 0 & 1 & 0 & 0 \\
					0 & 0 & 0 & 0 & 0 & 1 & 0 \\
					0 & 0 & 0 & 0 & 0 & 0 & 1 
				\end{array} &
				\begin{array}{cccccc}
					1 & 0 & 0 & 0 & 0 & 0 \\
					0 & 1 & 0 & 0 & 0 & 0 \\
					1 & 0 & 1 & 0 & 0 & 0 \\
					0 & 1 & 1 & 1 & 0 & 0 \\
					0 & 0 & 0 & 1 & 1 & 0 \\
					0 & 0 & 0 & 0 & 1 & 1 \\
					0 & 0 & 0 & 0 & 0 & 1 
				\end{array} &
				\begin{array}{cccccc}
					1 & 0 & 0 & 0 & 0 & 0 \\
					0 & 1 & 1 & 0 & 0 & 0 \\
					1 & 1 & 0 & 1 & 0 & 0 \\
					1 & 1 & 1 & 1 & 1 & 0 \\
					0 & 0 & 1 & 1 & 1 & 1 \\
					0 & 0 & 0 & 0 & 1 & 1 \\
					0 & 0 & 0 & 0 & 0 & 1 
				\end{array} &
				\begin{array}{cccccc}
					1 & 1 & 0 & 0 & 0 & 0 \\
					1 & 0 & 1 & 1 & 0 & 0 \\
					1 & 1 & 1 & 0 & 1 & 0 \\
					1 & 1 & 1 & 1 & 1 & 1 \\
					0 & 1 & 1 & 1 & 1 & 1 \\
					0 & 0 & 0 & 1 & 1 & 1 \\
					0 & 0 & 0 & 0 & 0 & 1 
				\end{array} &
				\begin{array}{cccccc}
					1 & 1 & 0 & 0 & 0 & 0 \\
					1 & 0 & 1 & 1 & 1 & 0 \\
					1 & 1 & 1 & 1 & 0 & 1 \\
					1 & 1 & 1 & 2 & 1 & 1 \\
					1 & 1 & 1 & 1 & 1 & 1 \\
					0 & 1 & 1 & 0 & 1 & 1 \\
					0 & 0 & 0 & 0 & 1 & 1 
				\end{array}
				&
				\begin{array}{ccccc}
					1 & 1 & 0 & 0 & 1 \\
					1 & 1 & 1 & 1 & 0 \\
					1 & 1 & 1 & 1 & 1 \\
					2 & 1 & 2 & 1 & 1 \\
					1 & 1 & 1 & 1 & 1 \\
					0 & 1 & 1 & 1 & 1 \\
					0 & 0 & 0 & 1 & 1 
				\end{array}
			\end{array}\right.$\\
			$\left.\begin{array}{ccccccccccccc}
				\qquad \qquad & \qquad \qquad &
				\begin{array}{ccccc}
					1 & 1 & 1 & 0 & 0 \\
					1 & 1 & 1 & 1 & 1 \\
					1 & 1 & 2 & 1 & 1 \\
					1 & 2 & 2 & 2 & 2 \\
					1 & 1 & 1 & 2 & 1 \\
					1 & 1 & 0 & 1 & 1 \\
					1 & 0 & 0 & 0 & 1 
				\end{array}
				&
				\begin{array}{cccc}
					1 & 1 & 1 & 0 \\
					1 & 1 & 1 & 1 \\
					1 & 2 & 1 & 1 \\
					2 & 2 & 2 & 2 \\
					2 & 1 & 1 & 2 \\
					1 & 1 & 1 & 1 \\
					0 & 0 & 1 & 1 
				\end{array}
				&
				\begin{array}{cccc}
					1 & 1 & 1 & 0 \\
					1 & 1 & 1 & 1 \\
					1 & 2 & 2 & 1 \\
					2 & 2 & 2 & 2 \\
					2 & 2 & 1 & 2 \\
					1 & 1 & 1 & 2 \\
					1 & 0 & 1 & 1 
				\end{array}
				&
				\begin{array}{ccc}
					1 & 1 & 1 \\
					1 & 1 & 1 \\
					1 & 2 & 2 \\
					2 & 2 & 3 \\
					2 & 2 & 2 \\
					2 & 1 & 1 \\
					1 & 1 & 0 
				\end{array}
				&
				\begin{array}{ccc}
					1 & 1 & 1 \\
					2 & 1 & 1 \\
					2 & 2 & 2 \\
					3 & 2 & 3 \\
					2 & 2 & 2 \\
					1 & 2 & 1 \\
					0 & 1 & 1 
				\end{array}
				&
				\begin{array}{cc}
					1 & 1 \\
					1 & 2 \\
					2 & 2 \\
					3 & 3 \\
					2 & 2 \\
					2 & 1 \\
					1 & 1 
				\end{array}
				&
				\begin{array}{cc}
					1 & 1 \\
					1 & 2 \\
					2 & 2 \\
					3 & 3 \\
					3 & 2 \\
					2 & 2 \\
					1 & 1 
				\end{array}
				&
				\begin{array}{c}
					1 \\
					2 \\
					2 \\
					3 \\
					3 \\
					2 \\
					1
				\end{array}
				&
				\begin{array}{c}
					1 \\
					2 \\
					2 \\
					4 \\
					3 \\
					2 \\
					1 
				\end{array}
				&
				\begin{array}{c}
					1 \\
					2 \\
					3 \\
					4 \\
					3 \\
					2 \\
					1 
				\end{array}
				&
				\begin{array}{c}
					2 \\
					2 \\
					3 \\
					4 \\
					3 \\
					2 \\
					1 
				\end{array}
			\end{array}\right).$
	\end{minipage}}
\end{align}

Then 
\begin{align}
	\PHI^+ = \bigset{c_{1j}\alpha_1 + \cdots + c_{7j}\alpha_7}{j \in \{1,\ldots,63\}}.
\end{align}

Let $j = 7$ and $\omega = \omega_7$.
Then $\sigma = (0\ 7)(1\ 6)(3\ 5)$ and $r = 4$.
We may assume that 
\begin{align}
	S_0^j = \{0,7\},\qquad
	S_1^j = \{1,6\},\qquad
	S_2^j = \{3,5\},\qquad
	S_3^j = \{4\},\qquad
	S_4^j = \{2\}.
\end{align}
Furthermore, we have
\begin{align}
	m_1 = \frac{\#S_1^j\bar{n}_1}{\#S_0^j} = \frac{2 \cdot 2}{2} = 2,\qquad
	m_2 = \frac{2 \cdot 3}{2} = 3,\qquad
	m_3 = \frac{1 \cdot 4}{2} = 2,\qquad 
	m_4 = \frac{1 \cdot 2}{2} = 1.
\end{align}
Define a $4 \times 63$ matrix $X = (x_{ij})_{ij}$ by 
\begin{align}
	x_{1j} = c_{1j} + c_{6j} - 2c_{7j},\qquad
	x_{2j} = c_{3j} + c_{5j} - 3c_{7j},\qquad
	x_{3j} = c_{4j} - 2c_{7j},\qquad
	x_{4j} = c_{2j} - c_{7j},
\end{align}
that is, we have $x_{1j}\bar{\alpha}_1^\omega + \cdots + x_{4j}\bar{\alpha}_4^\omega = (c_{1j}\alpha_1 + \cdots + c_{7j}\alpha_7)^{\omega}$.
Then the matrix $X$ is as follows:

\begin{align}
	\resizebox{\textwidth}{!}{$\begin{minipage}{24cm} 
			$X = \left(\begin{array}{cccccccccc}
				\begin{array}{ ccccccc }
					1 & 0 & 0 & 0 & 0 & 1 & -2 \\
					0 & 0 & 1 & 0 & 1 & 0 & -3 \\
					0 & 0 & 0 & 1 & 0 & 0 & -2 \\
					0 & 1 & 0 & 0 & 0 & 0 & -1 
				\end{array}
				&
				\begin{array}{ cccccc }
					1 & 0 & 0 & 0 & 1 & -1 \\
					1 & 0 & 1 & 1 & 1 & -3 \\
					0 & 1 & 1 & 1 & 0 & -2 \\
					0 & 1 & 0 & 0 & 0 & -1 
				\end{array}
				&
				\begin{array}{ cccccc }
					1 & 0 & 0 & 0 & 1 & -1 \\
					1 & 1 & 1 & 2 & 1 & -2 \\
					1 & 1 & 1 & 1 & 1 & -2 \\
					0 & 1 & 1 & 0 & 0 & -1 
				\end{array}
				&
				\begin{array}{ cccccc }
					1 & 1 & 0 & 1 & 1 & -1 \\
					1 & 2 & 2 & 1 & 2 & -2 \\
					1 & 1 & 1 & 1 & 1 & -1 \\
					1 & 0 & 1 & 1 & 0 & -1 
				\end{array}
				&
				\begin{array}{ cccccc }
					1 & 2 & 1 & 0 & -1 & -1 \\
					2 & 2 & 2 & 2 & -2 & -1 \\
					1 & 1 & 1 & 2 & -1 & -1 \\
					1 & 0 & 1 & 1 & 0 & -1 
				\end{array}
				&
				\begin{array}{ ccccc }
					1 & 2 & 1 & -1 & 0 \\
					2 & 2 & 2 & -1 & -1 \\
					2 & 1 & 2 & -1 & -1 \\
					1 & 1 & 1 & 0 & -1 
				\end{array}
				& \quad
			\end{array}\right.\\
				\left.\begin{array}{ccccccccccccccccccccccccccccccccccccccc}
					\qquad \qquad &\qquad \qquad & \qquad &
					\begin{array}{ ccccc }
						0 & 2 & 1 & 1 & -1 \\
						-1 & 2 & 3 & 3 & -1 \\
						-1 & 2 & 2 & 2 & 0 \\
						0 & 1 & 1 & 1 & 0 
					\end{array}
					&
					\begin{array}{ cccc }
						2 & 2 & 0 & -1 \\
						3 & 3 & -1 & 0 \\
						2 & 2 & 0 & 0 \\
						1 & 1 & 0 & 0 
					\end{array}
					&
					\begin{array}{ cccc }
						0 & 2 & 0 & 0 \\
						0 & 4 & 0 & 0 \\
						0 & 2 & 0 & 0 \\
						0 & 1 & 0 & 0 
					\end{array}
					&
					\begin{array}{ ccc }
						1 & 0 & 2 \\
						0 & 1 & 4 \\
						0 & 0 & 3 \\
						0 & 0 & 1 
					\end{array}
					&
					\begin{array}{ ccc }
						2 & 1 & 0 \\
						4 & 1 & 1 \\
						3 & 0 & 1 \\
						2 & 0 & 0 
					\end{array}
					&
					\begin{array}{ cc }
						1 & 0 \\
						1 & 1 \\
						1 & 1 \\
						0 & 1 
					\end{array}
					&
					\begin{array}{ cc }
						1 & 1 \\
						2 & 1 \\
						1 & 1 \\
						0 & 1 
					\end{array}
					&
					\begin{array}{ c }
						1 \\
						2 \\
						1 \\
						1 
					\end{array}
					&
					\begin{array}{ c }
						1 \\
						2 \\
						2 \\
						1 
					\end{array}
					&
					\begin{array}{ c }
						1 \\
						3 \\
						2 \\
						1 
					\end{array}
					&
					\begin{array}{ c }
						2 \\
						3 \\
						2 \\
						1 
					\end{array}
				\end{array}\right).$\end{minipage}$}
	\end{align}
	

Hence \cref{claim} holds true when $\PHI$ is of type $E_7$.

\label{ZE7}
\begin{proposition}\label{disapE7}
	Let $\PHI$ be a root system of type $E_7$.
	Then
	\begin{align}
		\PHI_{\omega,0}^+ &= \left\{ \begin{lgathered}
			\alpha_2 + \alpha_3 + 2\alpha_4 + 2\alpha_5 + 2\alpha_6 + \alpha_7,\ \\
			\alpha_1 + \alpha_2 + 2\alpha_3 + 2\alpha_4 + \alpha_5 + \alpha_6 + \alpha_7,\quad
			\alpha_1 + \alpha_2 + \alpha_3 + 2\alpha_4 + 2\alpha_5 + \alpha_6 + \alpha_7
		\end{lgathered}\right\}.
	\end{align}
	In particular, if we set 
	\begin{align}
		\alpha_1 = \frac{1}{2}(e_1 - e_2 - e_3 - e_4 - e_5 - e_6 - e_7 + e_8),\quad 
		\alpha_2 = e_2 + e_1,\quad 
		\alpha_3 = e_2 - e_1,\\ 
		\alpha_4 = e_3 - e_2,\quad 
		\alpha_5 = e_4 - e_3,\quad 
		\alpha_6 = e_5 - e_4,\quad
		\alpha_7 = e_6 - e_5.
	\end{align}
	based on \cite{Bourbaki}, then 
	\begin{align}
		\PHI_{\omega,0}^+ = \left\{ \begin{lgathered}
			e_6 + e_5,\quad
			\frac{1}{2}(-e_1 + e_2 + e_3 - e_4 - e_5 + e_6 - e_7 + e_8),\quad
			\frac{1}{2}(e_1 - e_2 - e_3 + e_4 - e_5 + e_6 - e_7 + e_8)
		\end{lgathered}
		\right\}.
	\end{align}
	Hence
	\begin{align}
		P(0) = \{\pm1\}.
	\end{align}
\end{proposition}
\begin{proof}
	See matrices $C$ and $X$.
\end{proof}

By \cref{typetheorem}, $\PHI_\re^\omega$ is of type $F_4$.
A root system of type $F_4$ is a set
\begin{align}
	\PHI(F_4) &= \bigset{\pm(e_i-e_j),\ \pm(e_i+e_j),\ \pm e_k}{i,j,k \in \{1,2,3,4\},\ i < j} \\
	&\qquad \qquad \sqcup \Bigset{\frac{\nu_1e_1 + \nu_2e_2 + \nu_3e_3 + \nu_4e_4}{2}}{
		\nu_1,\nu_2,\nu_3,\nu_4 \in \{-1,1\}
	}
\end{align}
for the standard basis $\{e_1,e_2,e_3,e_4\}$ of $\mathbb{R}^4$.
The roots 
\begin{align}
	\DELTA(F_4) = \left\{
	e_2 - e_3,\ e_3 - e_4,\ e_4,\ \frac{e_1 - e_2 - e_3 - e_4}{2}
	\right\}	
\end{align}
is a basis of $\PHI(F_4)$.
The coefficient matrix of $\PHI(F_4)^+$ with respect to $\DELTA(F_4)$ is as follows:
\begin{align}
	\left(\begin{array}{cccccccccccccccccccccccccccc}
		1 & 0 & 0 & 0 & 1 & 0 & 0 & 1 & 0 & 0 & 1 & 1 & 0 & 1 & 1 & 0 & 1 & 1 & 1 & 1 & 1 & 1 & 1 & 2 \\
		0 & 1 & 0 & 0 & 1 & 1 & 0 & 1 & 1 & 1 & 1 & 1 & 1 & 1 & 2 & 1 & 1 & 2 & 2 & 2 & 2 & 2 & 3 & 3 \\
		0 & 0 & 1 & 0 & 0 & 1 & 1 & 1 & 1 & 2 & 1 & 2 & 2 & 2 & 2 & 2 & 2 & 2 & 2 & 3 & 3 & 4 & 4 & 4 \\
		0 & 0 & 0 & 1 & 0 & 0 & 1 & 0 & 1 & 0 & 1 & 0 & 1 & 1 & 0 & 2 & 2 & 1 & 2 & 1 & 2 & 2 & 2 & 2
	\end{array}\right)
\end{align}

\begin{proposition}\label[prop]{extraE7}
	Let $\PHI$ be a root system of type $E_7$.
	For any $\beta^\omega \in (\PHI_\re^\omega)^+$, we have
	\begin{align}
		P(\beta^\omega) = \begin{cases*}
			\{0,\pm1\}		& if $\beta^\omega$ is a short root of $\PHI_\re^\omega$;\\
			\{0\}			& if $\beta^\omega$ is a long root of $\PHI_\re^\omega$.
		\end{cases*}
	\end{align}
\end{proposition}
\begin{proof}
	Observing matrices $C$ and $X$, we can see the following holds for any $p \in \{-1,1\}$: If $\beta$ is a $p$-extra root, then $\beta^\omega$ is a short root.
	Conversely, every short root $\beta^\omega$ has a $p$-extra root.
\end{proof}

We have now shown that \cref{claim} holds true in all cases.
Hence the proof of \cref{basistheorem} is complete.

\begin{proposition}[Summary above]
	Let $\PHI$ be an arbitrary irreducible reduced root system.
	Then 
	\begin{align}
		P(0) = \{\pm1,\, \ldots,\, \pm(o(\omega)-1)\}.
	\end{align}
	For each $\beta^\omega \in (\PHI_\re^\omega)^+$, $P(\beta^\omega)$ is as shown in \cref{tableE}.
\end{proposition}

\begin{table}[h]
	
	\caption{List of $P(\beta^\omega)$}\label{tableE}

	\centering
	\scalebox{0.9}{
	\begin{tabular}{ll|c|c|Sc|c}
		\multicolumn{2}{c|}{$\PHI$}                             & $j$                              & $\PHI^{\omega}_\re$ 			& $P(\beta^\omega)$ & cf. \\ \hline\hline
		$A_\ell$  &                         & $\gcd\{\ell+1,j\} = 1$		   & $\varnothing$         			& --- & ---\\ 
		$A_\ell$  &                         & $g \ceq \gcd\{\ell+1,j\} \neq 1$ & $A_{g-1}$         				& $\{0,\, \pm1,\, \ldots,\, \pm(o(\omega)-1)\}$ & \cref{extraA}  \\ 
		$B_2$     &                         & $j = 1$                          & $A_1$               		   	& $\{0,\pm1\}$ & \cref{extraB2} \\ 
		$B_\ell$  & $(\ell \geq 3)$			& $j = 1$                          & $B_{\ell-1}$          		   	& $\begin{cases*}
			\{0,\pm1\}		& if $\beta^\omega$ is a short root of $\PHI_\re^\omega$;\\
			\{0\}			& if $\beta^\omega$ is a long root of $\PHI_\re^\omega$
		\end{cases*}$ & \cref{extraB}  \\ 
		$C_\ell$  & $(\ell \geq 3$, odd)	& $j = \ell$                       & $BC_{\frac{\ell-1}{2}_{_{}}}$ 	& \multirow{2}{*}{$\{0,\pm1\}$} & \multirow{2}{*}{\cref{extraC}} \\ 
		$C_\ell$  & $(\ell \geq 4$, even) 	& $j = \ell$                       & $C_{\frac{\ell}{2}_{_{}}}$    	&  & \\ 
		$D_\ell$  & $(\ell \geq 4$)   		& $j = 1 $                         & $B_{\ell-2}$    				& $\begin{cases*}
			\{0,\pm1\}		& if $\beta^\omega$ is a short root of $\PHI_\re^\omega$;\\
			\{0\}			& if $\beta^\omega$ is a long root of $\PHI_\re^\omega$
		\end{cases*}$ & \cref{extraD1} \\ 
		$D_\ell$  & $(\ell \geq 4$, even) 	& $j \in \{\ell-1,\, \ell\}$       & $C_{\frac{\ell}{2}_{_{}}}$    	& $\begin{cases*}
			\{0,\pm1\}		& if $\beta^\omega$ is a short root of $\PHI_\re^\omega$;\\
			\{0\}			& if $\beta^\omega$ is a long root of $\PHI_\re^\omega$
		\end{cases*}$ & \cref{extraDe} \\ 
		$D_\ell$  & $(\ell \geq 5$, odd)  	& $j \in \{\ell-1,\, \ell\}$       & $BC_{\frac{\ell-3}{2}_{_{}}}$  & $\begin{cases*}
			\{0,\pm1,\pm2,\pm3\}		& if $\beta^\omega$ is a short root of $\PHI_\re^\omega$;\\
			\{0,\pm2\}			& if $\beta^\omega$ is not a short root of $\PHI_\re^\omega$
		\end{cases*}$ & \cref{extraDo} \\ 
		$E_6$     &                         & $j \in \{1,6\}$                  & $G_2$          			    & $\begin{cases*}
			\{0,\pm1,\pm2\}		& if $\beta^\omega$ is a short root of $\PHI_\re^\omega$;\\
			\{0\}			& if $\beta^\omega$ is a long root of $\PHI_\re^\omega$
		\end{cases*}$ & \cref{extraE6} \\ 
		$E_7$     &                         & $j = 7$                          & $F_4$               			& $\begin{cases*}
			\{0,\pm1\}		& if $\beta^\omega$ is a short root of $\PHI_\re^\omega$;\\
			\{0\}			& if $\beta^\omega$ is a long root of $\PHI_\re^\omega$
		\end{cases*}$ & \cref{extraE7}\\ 
	\end{tabular}}

	\

\end{table}

\section*{Acknowledgement}

The author would like to thank Professor Masahiko Yoshinaga for the helpful discussions and comments on this research. 
The author also acknowledge support by JSPS KAKENHI, Grant Number 25KJ1735.


\bibliographystyle{amsplain}
\bibliography{bibfile}

\newpage 

\begin{figure}
	\begin{tabular}{m{0.5cm}m{2cm}m{7cm}}
		$A_\ell$ & ($\ell \geq 1$) & \DA     \\
		$B_\ell$ & ($\ell \geq 2$) & \DB     \\
		$C_\ell$ & ($\ell \geq 3$) & \DC     \\
		$D_\ell$ & ($\ell \geq 4$) & \DD     \\
		$E_6$    &                 & \DEvi   \\
		$E_7$    &                 & \DEvii  \\
		$E_8$    &                 & \DEviii \\
		$F_4$    &                 & \DF     \\
		$G_2$    &                 & \DG    
	\end{tabular}
	
	\caption{List of Dynkin diagrams (the numbers on the vertices represent the values of $n_i$)}
	\label{figDD}
\end{figure}

\begin{figure}
	\begin{tabular}{m{0.5cm}m{2cm}m{7cm}}
		$A_1$     &                 & \eDAi    \\
		$A_\ell$  & ($\ell \geq 2$) & \eDAell  \\
		$BC_\ell$ & ($\ell \geq 1$) & \eDBC    \\
		$B_2$     &                 & \eDBii   \\
		$B_\ell$  & ($\ell \geq 3$) & \eDBell  \\
		$C_\ell$  & ($\ell \geq 3$) & \eDC     \\
		$D_\ell$  & ($\ell \geq 4$) & \eDD     \\
		$E_6$     &                 & \eDEvi   \\
		$E_7$     &                 & \eDEvii  \\
		$E_8$     &                 & \eDEviii \\
		$F_4$     &                 & \eDF     \\
		$G_2$     &                 & \eDG    
	\end{tabular}
	
	\caption{List of extended Dynkin diagrams (the numbers on the vertices represent the values of $n_i$)}
	\label{figeDD}
\end{figure}

\begin{figure}
	\begin{flushleft}
		\begin{tabular}{m{0.5cm}m{3cm}m{3cm}m{1cm}m{1cm}m{1cm}m{5cm}}
			$\PHI$ & & $\mathcal{D}_0(\PHI)$ & $\lra$ & $\mathcal{D}_0(\PHI)^\omega$\\
			\\
			$A_\ell$                                      & ($\gcd\{\ell+1,\, j\} = 1$)                                        & \FFO    &&& \Dyaji &         \\
			\multicolumn{1}{c}{\multirow{2}{*}{$A_\ell$}} & \multirow{2}{*}{($\gcd\{\ell+1,\, j\} \neq 1$)}                    & \FFi    &        &                       \\
			\multicolumn{1}{c}{}                          &                                                                    && \Fi     &        &                       \\
			$B_2$                                         & ($j = 1$)                                                          & \FFii    & \Fii   \\
			\multirow{2}{*}{$B_\ell$}                     & \multirow{2}{*}{($\ell \geq 3$,  $j = 1$)}                         & \FFiii  &        &                       \\
			&                                                                    && \Fiii   &        &                       \\
			\multirow{2}{*}{$C_\ell$}                     & \multirow{2}{*}{($\ell \geq 3$, odd, $j = \ell$)}                  & \FFiv   &        &                       \\
			&                                                                    && \Fiv    &        &                       \\
			\multirow{2}{*}{$C_\ell$}                     & \multirow{2}{*}{($\ell \geq 4$, even, $j = \ell$)}                 & \FFv    &        &                       \\
			&                                                                    && \Fv     &        &                       
		\end{tabular}
	\end{flushleft}
	\caption{List of folding the extended Dynkin diagram I (the numbers on the vertices represent the values of $m_k$)}
	\label{figfDDI}
\end{figure}

\begin{figure}
	\begin{flushleft}
		\begin{tabular}{m{0.5cm}m{3cm}m{3cm}m{1cm}m{1cm}m{1cm}m{5cm}}
			$\PHI$ & & $\mathcal{D}_0(\PHI)$ & $\lra$ & $\mathcal{D}_0(\PHI)^\omega$\\
			\\
			\multirow{2}{*}{$D_\ell$}                     & \multirow{2}{*}{($\ell \geq 4$,  $j = 1$)}                         & \FFvi   &        &                       \\
			&                                                                    && \Fvi    &        &                       \\
			$D_4$                                         & ($j \in \{\ell-1,\, \ell\}$)                                       & \FFvii   & \Fvii  \\
			$D_5$                                         & ($j \in \{\ell-1,\, \ell\}$)                                       & \FFviii  & \Fviii \\
			\multirow{2}{*}{$D_\ell$}                     & \multirow{2}{*}{$\left(\,\begin{lgathered}
					\ell \geq 6,\ \text{even},\\ j \in \{\ell-1,\, \ell\}
				\end{lgathered}\,\right)$} & \FFix   &        &                       \\
			&                                                                    && \Fix    &        &                       \\
			\multirow{2}{*}{$D_\ell$}                     & \multirow{2}{*}{$\left(\,\begin{lgathered}
					\ell \geq 7,\ \text{odd},\\ j \in \{\ell-1,\, \ell\}
				\end{lgathered}\,\right)$}  & \FFx    &        &                       \\
			&                                                                    && \Fx     &        &                       \\
			$E_6$                                         & ($j \in \{1,6\}$)                                                  & \FFxi   &&  \Fxi   \\
			$E_7$                                         & ($j = 7$)                                                          & \FFxii  &&&  \Fxii 
		\end{tabular}
	\end{flushleft}

	\caption{List of folding the extended Dynkin diagram II (the numbers on the vertices represent the values of $m_k$)}
	\label{figfDDII}
\end{figure}

\end{document}